\documentclass[12pt,centertags,oneside]{amsart}
\usepackage{amsmath,amstext,amsthm,amscd,typearea,hyperref}
\usepackage{amssymb}
\usepackage{a4wide}
\usepackage[cal=
esstix,
bb=fourier,
scr=rsfso
]{mathalfa}
\usepackage{charter}
\usepackage{pdfsync}
\usepackage{dsfont}
\usepackage{mathabx}

\usepackage{color}
\usepackage[a4paper,width=16.2cm,top=3cm,bottom=3cm]{geometry}

\numberwithin{equation}{section}

\newcommand{\field}[1]{\mathbb{#1}} 
\newcommand{\C}{\field{C}} 
\newcommand{\N}{\field{N}} 

\newcommand{\R}{\field{R}} 
\newcommand{\Z}{\field{Z}} 
\newcommand{\T}{\field{T}} 

\DeclareMathOperator{\End}{End} 
\DeclareMathOperator{\Ker}{Ker}
\DeclareMathOperator{\spec}{Spec} 
\DeclareMathOperator{\Dom}{Dom} 
\DeclareMathOperator{\Ran}{Im}
\DeclareMathOperator{\supp}{supp}
\DeclareMathOperator{\diag}{diag}
\DeclareMathOperator{\Ad}{Ad}
\DeclareMathOperator{\im}{Im}

\newcommand{\ddbar}{\overline\partial}
\newcommand{\pr}{\partial}
\newcommand{\ol}{\overline}
\newcommand{\Td}{\widetilde}
\newcommand{\norm}[1]{\Vert#1\Vert}
\newcommand{\set}[1]{\left\{#1\right\}}
\newcommand{\To}{\rightarrow}

\newcommand{\mU}{\mathcal{U}}

\newcommand{\mL}{\mathcal{L}}
\newcommand{\mH}{\mathcal{H}}

\newcommand{\mO}{\mathcal{O}}
\newcommand{\mR}{\mathcal{R}}

\newcommand{\mP}{\mathcal{P}}
\newcommand{\mG}{\mathcal{G}}

\newcommand{\cL}{\cali{L}}

\def\cC{\mathscr{C}}
\def\cD{\mathscr{D}}
\def\cL{\mathscr{L}}

\def\cL{\mathscr{L}}

\newcommand{\goth}[1]{\mathfrak{#1}}

\newcommand{\kg}{\goth{g}}

\newtheorem{theorem}{Theorem}[section]

\newtheorem{lemma}[theorem]{Lemma}
\newtheorem{corollary}[theorem]{Corollary}
\theoremstyle{definition}
\newtheorem{definition}[theorem]{Definition}
\newtheorem{example}[theorem]{Example}
\newtheorem{remark}[theorem]{Remark}
\newtheorem{ass}[theorem]{Assumption}
\theoremstyle{definition}

\newtheorem{convention}[theorem]{Convention}

\numberwithin{equation}{section}

\newcommand{\abs}[1]{\lvert#1\rvert}
\newcommand{\comment}[1]{}

\begin{document}
\title[Geometric quantization on CR manifolds]
{Geometric quantization on CR manifolds}

\author{Chin-Yu Hsiao}

\address{Institute of Mathematics, Academia Sinica and 
National Center for Theoretical Sciences,\newline
    \mbox{\quad}\,Astronomy-Mathematics Building, 
No. 1, Sec. 4, Roosevelt Road, Taipei 10617, Taiwan}
\email{chsiao@math.sinica.edu.tw or chinyu.hsiao@gmail.com}
\author[Xiaonan Ma]{Xiaonan Ma}
\address{
Institut de Math\'ematiques de Jussieu-Paris Rive Gauche,
Universit\'e de Paris, CNRS, 
 F-75013
\newline
\mbox{\quad}\,Paris, France.
}
\email{xiaonan.ma@imj-prg.fr}

\author[George Marinescu]{George Marinescu}
\address{Universit{\"a}t zu K{\"o}ln,  Mathematisches Institut,
    Weyertal 86-90,   50931 K{\"o}ln, Germany\\
    \newline
    \mbox{\quad}\,Institute of Mathematics `Simion Stoilow', 
	Romanian Academy,
Bucharest, Romania}
\email{gmarines@math.uni-koeln.de}

\keywords{Szeg\H{o} kernel, moment map, CR manifolds} 
\subjclass[2000]{Primary: 58J52, 58J28; Secondary: 57Q10}

\begin{abstract}
Let $X$ be a compact connected orientable CR manifold
with the action of a connected compact Lie group $G$.
Under natural pseudoconvexity assumptions we show
that the CR Guillemin-Strernberg map is Fredholm at 
the level of Sobolev spaces of CR functions.
As application we study this map for holomorphic 
line bundles which are positive near the inverse image of $0$ 
by the momentum map.
We also show that ``quantization commutes with reduction''
for Sasakian manifolds.
\end{abstract}

\maketitle \tableofcontents

\section{Introduction and statement of the main results}
\label{s-gue170410}

The famous geometric quantization conjecture of 
Guillemin and Sternberg \cite{GS:82} states
that for a compact pre-quantizable symplectic manifold admitting 
a Hamiltonian action of a compact connected Lie group, the principle of 
``quantization commutes with reduction" holds.
This conjecture was first proved independently by Meinrenken \cite{M96} 
and Vergne \cite{Ver:96} for
the case where the Lie group is abelian, and  by Meinrenken \cite{M98}
 in the general case,
 then Tian-Zhang~\cite{TZ98} gave a purely analytic proof in general case
 with various generalizations,  see \cite{Ve02} for a survey
 and complete references on this subject.
In the case of a non-compact symplectic manifold $M$ 
with a compact connected Lie group action $G$, 
this question was solved by Ma-Zhang~\cite{MZ09,MZI}
as a solution to a conjecture of Vergne in her ICM 2006 plenary 
lecture \cite{Ve07}, see \cite{Ma10} for a survey. 
Paradan \cite{Par11} gave a new proof,
cf.\ also  the recent work \cite{HochsSong17}.
A natural choice for the quantum spaces  of a compact symplectic 
manifold is the kernel of the Dirac operator.

In~\cite{MZ}, Ma-Zhang established the asymptotic expansion 
of the $G$-invariant Bergman kernel for a positive line bundle $L$ over 
a compact symplectic manifold $M$ and by using the asymptotic expansion 
of $G$-invariant Bergman kernel, they could establish the
``quantization commutes with reduction'' theorem
when the power of the line bundle $L$ is high enough. 

On a compact K\"ahler manifold $M$ endowed with 
a prequantum line bundle $L$,
a natural choice for the Hilbert space of quantum states
is the space $H^0(M,L^m)$ of holomorphic sections of 
the tensor powers $L^m$.
The family of quantum spaces $H^0(M,L^m)$ indexed by $m\in\N$
plays an essential role in geometric quantization
and the semi-classical limit $m\to\infty$ allows
to recover the classical mechanics of the phase space $M$.

One can wrap up the family of spaces $H^0(M,L^m)$, $m\in\N$,
as subspaces of a single Hilbert space by considering
the $S^1$-bundle $X\subset L^*$, which is a strictly pseudoconvex
CR manifold and identifying $H^0(M,L^m)$ with the
$S^1$ isotypes of $m$-equivariant CR functions on $X$.
The Hilbert space sum $\widehat{\oplus}_{m\in\N}H^0(M,L^m)$
can be identified to the space $H^0_b(X)$ of $L^2$ CR functions on $X$
and the sum of the Bergman projections $B_m$ on $H^0(M,L^m)$
can be identified to the Szeg\H{o} projector $S$ on $H^0_b(X)$.
A fundamental fact is that the asymptotic behavior of $B_m$
is encoded in the singularities of the Szeg\H{o} kernel $S(\cdot,\cdot)$.  
We can thus think of $X$ as the quantizing principal bundle
of $M$ and of the space of $L^2$ CR functions as the quantum space
of $X$. In the presence of a $G$-action on $M$, which lifts to an action
on $L$, we have an induced $G$-action on $X$ and on $H^0_b(X)$.
 
In this way the idea emerges of quantizing a general CR manifold 
by using the analytic properties of the Szeg\H{o} kernel (cf.\ \cite{BG81}). 
In this paper we study the principle of 
``quantization commutes with reduction''  
for CR manifolds, in particular for Sasakian manifolds 
(see Theorems~\ref{t-180428zy}, \ref{t-gue180520m}). 
Sasakian geometry is an important odd-dimensional counterpart 
of K\"ahler geometry.
It is known that an irregular Sasakian manifold admits a compact CR torus
action (see~\cite[Section 3]{HHL17}, or Remark \ref{eq:t2.1a})
and the study of $G$-equivariant CR functions on 
a Sasakian manifold is important in Sasaki geometry. 
We refer to \cite{Al89,BoGa08,Gei97,Will02} for the fundamentals 
of contact and Sasakian reduction and examples.
One of our main tools will be we develop a $G$-invariant 
Fourier integral operator calculus which will be 
used to obtain the asymptotics of 
the $G$-invariant Szeg\H{o} kernel.

An important difference between the CR setting and the symplectic setting
is that the quantum spaces in the case of a compact symplectic 
manifolds are finite dimensional, whereas for 
the compact strictly pseudoconvex CR manifolds
that we consider the quantum spaces 
consist of CR functions and are infinite dimensional.


We now formulate the main results. We refer to Section~\ref{s:prelim}
for some notations and terminology used here.
Let $(X, T^{1,0}X)$ be a compact orientable 
CR manifold of dimension $2n+1$, $n\geq1$, 
where $T^{1,0}X$ denotes the CR structure of $X$.
Let $HX\subset TX$ be the Levi distribution of the CR structure
and let $\omega_0\in\cC^\infty(X,T^*X)$ be a global
non-vanishing real $1$-form annihilating $HX$, called characteristic
$1$-form.

Let $G$ be a $d$-dimensional connected compact Lie group with Lie 
algebra $\kg$ acting on $X$ by preserving $J$ and $\omega_0$.
Let $\mu: X\to \kg^{*}$ be the associated 
moment map $\mu: X\to \kg^{*}$ (cf.\ \eqref{E:cmpm}).
We will mainly work in the following setting.
\begin{ass}\label{a-gue170410I}
The $G$-action preserves the complex structure 
$J$ on $HX$ and the characteristic $1$-form $\omega_0$, 
it is free on $\mu^{-1}(0)$, and 
one of the following conditions are fulfilled:

(i) $\dim X\geq5$ and
the Levi form of $X$ is positive definite 
near $\mu^{-1}(0)$. 

(ii) $\dim X=3$,
the Levi form of $X$ is positive definite everywhere
and $\ddbar_{b}$ has closed range in $L^2$ on $X$.  
\end{ass}
Due to Lemma \ref{eq:t2.4}, Assumption~\ref{a-gue170410I}
implies that $0$ is a regular value of $\mu$, hence 
$\mu^{-1}(0)$ is a 
$d$-codimensional submanifold of $X$.
Let 
\begin{equation}\label{eq:2.25}
Y:=\mu^{-1}(0),\ \ X_{G}:=\mu^{-1}(0)/G.
\end{equation}
The space $X_{G}$ is called the CR reduction.
Under our hypotheses, if $\dim X_G\geq3$, $X_{G}$
is a strictly pseudoconvex CR manifold
with characteristic $1$-form (in this case also contact form) 
$\omega_{0,G}$ 
induced canonically by $\omega_{0}$, see \eqref{eq:2.32}.
If $\dim X_G=1$, then each of the finitely many components of $X_G$ 
is diffeomorphic to a circle.

Let $\ddbar_b:\cC^\infty(X)\To\Omega^{0,1}(X)$ 
and if $\dim X_G\geq3$,
let $\ddbar_{b,X_{G}}: \cC^\infty(X_{G})\To\Omega^{0,1}(X_{G})$ 
be the tangential Cauchy-Riemann operators on $X$ and $X_{G}$, 
respectively.
We extend $\ddbar_b$ and $\ddbar_{b,X_{G}}$ 
to $L^2$ spaces
by taking their weak maximal extension, see \eqref{eq:2.15}.
We consider the spaces of $L^2$ CR functions 
\begin{align}\label{eq:2.26a}
H^0_b(X):=\set{u\in L^2(X);\, \ddbar_bu=0},\ \ 
H^0_b(X_{G}):=\set{u\in L^2(X_{G});\, \ddbar_{b,X_{G}}u=0}.
\end{align}
If $\dim X_G=1$, so $X_G$ is a finite union of circles, we set 
$H^0_b(X_{G})$ to be the direct sum of the Hardy spaces
of the components, that is, the $L^2$ subspaces of functions
with vanishing Fourier coefficients of negative degree.
The common feature of the spaces $H^0_b(X_{G})$ for
$\dim X_G\geq3$ and $\dim X_G=1$ is the fact that they
are boundary values of holomorphic functions in a filling
of $X_G$ by a complex manifold (see Section \ref{s2.4}).

Then $H^0_b(X)$ is a (possible infinite dimensional) 
$G$-representation, its $G$-invariant part is the
$G$-invariant $L^2$ CR functions on $X$,
\begin{equation}\label{e-gue180227}
H^0_b(X)^G:=\set{u\in H^0_b(X);\, h^*u=u,\ \  
\mbox{for any $h\in G$}}. 
\end{equation} 
For every $s\in\mathbb R$, let $H^s(X)$ and $H^s(X_G)$ 
denote the Sobolev space of $X$ of order $s$ and 
the Sobolev space of $X_G$ of order $s$ respectively and 
let $(\,\cdot\,,\,\cdot\,)_s$ and $(\,\cdot\,,\,\cdot\,)_{X_G,s}$
be the inner products on $H^s(X)$ and $H^s(X_G)$ respectively 
(see \eqref{eq:ma5.3}).
 For every $s\in\mathbb R$,  put 
 \begin{equation}\label{e-gue200926yyd}
 \begin{split}
 &H^0_b(X)_s:=\set{u\in H^s(X);\,\mbox{$\ddbar_bu=0$ in the sense 
 of distributions}},\\
 &H^0_b(X_G)_s:=\set{u\in H^s(X_G);\,\mbox{$\ddbar_bu=0$
 in the sense of distributions}},\\
 &H^0_b(X)^G_s:=\set{u\in H^0_b(X)_s;\,\mbox{$h^*u=u$
 in the sense of distributions, for every $h\in G$}}.
 \end{split}
 \end{equation} 
If $\dim X_G=1$, we set 
$H^0_b(X_{G})_s$ to be the direct sum of the Hardy-Sobolev spaces
of the components, 
that is, the subspaces of $H^s(S^1)$ of distributions
with vanishing Fourier coefficients of negative degree.

Let $\iota:Y\To X$ be the natural inclusion and let 
$\iota^*:\cC^\infty(X)\To \cC^\infty(Y)$ be the pull-back by $\iota$.
Let 
$\iota_G:\cC^\infty(Y)^G\To\cC^\infty(X_{G})$
be the natural identification. Let 
\begin{equation}\label{180308Im}
\begin{split}
\sigma_{G}:H^0_{b}(X)^G\cap\cC^\infty(X)^G&\To 
H^{0}_{b}(X_{G}),\quad\sigma_{G}=\iota_G\circ \iota^*.
\end{split}
\end{equation}
The map \eqref{180308Im} is well defined, see the construction
of the CR reduction in Section \ref{s2.3}.  The map $\sigma_G$
does not extend to a bounded operator on $L^2$, so it necessary
to consider its extension to Sobolev spaces.
From Theorem~\ref{t-gue200926yyda}, $\sigma_G$ extends by 
density to a bounded operator
\begin{equation}\label{e:GSa}
\sigma_{G}=\sigma_{G,s}: H^0_b(X)^G_s\To 
H^0_b(X_{G})_{s-\frac{d}{4}},\ \ \mbox{for every $s\in\mathbb R$}. 
\end{equation}
This operator can be thought as a Guillemin-Sternberg map
in the CR setting. It maps the ``first quantize 
and then reduce'' space
(the space of $G$-invariant Sobolev CR functions on $X$)
to the ``first reduce and then quantize'' space (the space
Sobolev CR functions on $X_G$).
Indeed, from the point of view of
quantum mechanics, the Hilbert space structures play an essential role. 
It is natural, then, to investigate the extent to which the 
CR Guillemin-Sternberg map is Fredholm.
The main result of this work is the following. 

\begin{theorem}\label{t-180428zy}
Let $X$ be a compact orientable 
CR manifold 
and let $G$ be a connected compact Lie group acting on $X$
such that the $G$-action preserves $J$ and $\omega_{0}$ and
Assumption \ref{a-gue170410I} holds.
Suppose that $\ddbar_{b,X_{G}}$ has closed range in $L^2$.
Then, for every $s\in\mathbb R$, the 
CR Guillemin-Sternberg map \eqref{e:GSa}
is Fredholm. Actually, $\Ker\sigma_{G,s}$ and 
$(\Ran\sigma_{G,s})^\perp$ 
are finite dimensional subspaces of  
$\cC^\infty(X)\cap H^0_b(X)^G$ and 
$\cC^\infty(X_{G})\cap H^0_b(X_{G})$, 
respectively, $\Ker\sigma_{G,s}$ and the index 
${\rm dim\,}\Ker\sigma_{G,s}-{\rm dim\,}(\Ran\sigma_{G,s})^\perp$ 
are independent of $s$. 
\end{theorem}
It should be noticed that $(\Ran\sigma_{G,s})^\perp$ is given by 
\begin{equation}\label{eq:ma1.8}
(\Ran\sigma_{G,s})^\perp:=\set{u\in H^0_b(X_G)_{s-\frac{d}{4}};\, 
\mbox{$(\,\sigma_{G,s}v\,,u\,)_{X_G,s-\frac{d}{4}}=0$, 
for every $v\in H^0_b(X)^G_s$}}.
\end{equation}
Under Assumption \ref{a-gue170410I} (i) 
the hypothesis that $\dim X\geq5$ 
is used in order to have local subelliptic Sobolev estimates on the
set where the Levi form is positive definite 
(Theorem \ref{t-gue181012}) and leads to the fact that
the $G$-invariant Kohn Laplacian \eqref{e-gue181012}
has closed range in $L^2$. Note also that the Kohn Laplacian
on strictly pseudoconvex CR manifolds of dimension greater than 
or equal to five
has always closed range in $L^2$ but this is not true for all three 
dimensional strictly pseudoconvex CR manifolds 
(a detailed discussion about
the closed range of $\ddbar_{b}$ in $L^2$ can be 
found in Section \ref{s2.4}). In the case when $\dim X=3$ 
we will state in Theorem \ref{t:ma5.7} a version of 
Theorem \ref{t-180428zy} under
weaker hypotheses as Assumption \ref{a-gue170410I} (ii), namely that
$X$ is pseudoconvex of finite type and
$\ddbar_{b,X}$ has closed range in $L^2$.

We turn now our attention to Sasakian manifolds.
Let  $(X, T^{1,0}X)$ be a compact connected Sasakian manifold
(see Section \ref{s2.30}). We fix a contact form $\omega_0$
and an associated Reeb vector field $R$ as in \eqref{e:reeb}.
We assume that 
\begin{equation}\label{e-gue180520}
\mbox{$R$ is $G$-invariant}. 
\end{equation}
By Remark \ref{eq:t2.1a}, the $R$-action preserves $HX$, $J$ 
and the natural metric $g_{\omega_{0}}$ on $TX$. 
As $X$ is compact, this implies the $R$-action generates a 
compact torus $\T$-action on $X$ and this $\T$-action commutues with 
the $G$-action. Thus it naturally induces a $\T$-action on $X_{G}$ 
and the generator $R$ induces the Reeb vector field $\widehat R$ on 
$X_{G}$.
It is clear that $X_{G}$ is also a compact  
Sasakian manifold and 
$\widehat R$ preserves the CR structure $T^{1,0}X_{G}$, and 
$\widehat R$, $T^{1,0}X_{G}\oplus T^{0,1}X_{G}$ generate the 
complex tangent bundle of $X_{G}$. 

Now $H^0_{b}(X)^G$ and $H^0_{b}(X_{G})$ are both 
$\T$-Hilbert spaces, 
thus we have the decomposition of Hilbert spaces via the weight 
$\alpha\in \T^{*}(\simeq \Z^{\dim \T})$ of $\T$-action:
\begin{align}\label{e-180520mI}
H^0_{b}(X)^G= \oplus_{\alpha\in \T^{*}}
H^0_{b,\alpha}(X)^G, \qquad
H^0_{b}(X_{G})= \oplus_{\alpha\in \T^{*}}H^0_{b,\alpha}(X_{G}).
\end{align}
Both $H^0_{b,\alpha}(X)^G$ and $H^0_{b,\alpha}(X_{G})$ 
are finite dimensional
subspaces of $\cC^\infty(X)^G$ and $\cC^\infty(X_{G})$, 
respectively, as subspaces of the eigenspaces of the elliptic
operators $\ddbar_{b,X}^{*}\ddbar_{b,X}-R^2$, 
$\ddbar_{b,X_G}^{*}\ddbar_{b,X_{G}}-\widehat{R}^2$,
respectively, of eigenvalues $|\alpha(R)|^2$.
From the definition \eqref{180308Im} of the map $\sigma_{G}$ 
we see that
\begin{align}\label{eq:2.30a}
\sigma_{G} Ru=\widehat R\sigma_{G} u,\
\  \text{ for any } u\in H^0_b(X)^G,
\end{align}
and hence $\sigma_{G}$ maps $H^0_{b,\alpha}(X)^G$
to $H^0_{b,\alpha}(X_{G})$.
From this observation, Theorem~\ref{t-180428zy} and the fact that
$\ddbar_{b}$ has  
closed range in $L^2$ on Sasakian manifolds
(see~\cite{MY07}, also \S \ref{s2.4}), we deduce: 

\begin{theorem}[quantization commutes with 
reduction for Sasakian manifolds]\label{t-gue180520m}\
Let $X$ be a Sasakian manifold. 
Assume that the Reeb vector field is $G$-invariant.
Then with the exception of finitely many $\alpha$
the map
\begin{align}\label{eq:2.31a}
\sigma_{G}: H^0_{b,\alpha}(X)^G\To H^0_{b,\alpha}(X_{G})
\end{align}
is an isomorphism.
\end{theorem}
 
We now apply Theorem~\ref{t-180428zy} to the case of 
complex manifolds. 
Let $(L,h^L)$ be a Hermitian holomorphic line bundle over a connected 
compact complex manifold $(M,J)$ with ${\rm dim\,}_{\mathbb C}M=n$, 
where $J$ denotes the complex structure of $TM$ and $h^L$ 
is a Hermitian metric of $L$. 
We denote by $R^L$ the Chern curvature of $(L,h^L)$. 
We assume that $G$ acts holomorphically on $(M,J)$, 
and that the action lifts to a holomorphic action on $L$. 
We assume further that $h^L$ is preserved by the $G$-action. 
Then $R^L$ is a $G$-invariant form. Let 
$\tilde{\mu}: M \to \mathfrak{g}^*$ be the moment map
defined by the Kostant formula \eqref{0.3}.

Assume that $0 \in \mathfrak{g}^*$ is regular and the action of $G$ on 
$\tilde{\mu}^{-1}(0)$ is free. If $iR^L$ is positive near 
$\tilde{\mu}^{-1}(0)$, 
then the analogue of the Marsden-Weinstein reduction holds. 
More precisely, the complex structure $J$ on $M$ induces 
a complex structure $J_G$ 
on $M_G:=\tilde{\mu}^{-1}(0)/G$, for which the line bundle 
$L_G:=L/G$ is a 
holomorphic line bundle over $M_G$. For $m\in\mathbb N$, 
let $H^0(M,L^m)^G$ denote the space of $G$-invariant holomorphic 
sections with values in $L^m$. 

Let $X:=\set{v\in L^*;\, \abs{v}^2_{h^{L^*}}=1}$ be the circle bundle 
of $L^*$, where $h^{L^*}$ is the Hermitian metric 
on $L^*$ induced by $h^L$. 
Let $e^{i\theta}$ be the natural $S^1$-action on the fibers of $X$.
The action of $G$ on $M$ lifts to a CR action on $X$. 
For every $m\in\mathbb N$, put 
\begin{align}\label{eq:2.32a}
H^0_{b,m}(X)^G:=\set{u\in H^0_b(X)^G;\, 
\mbox{$(e^{i\theta})^*u=e^{im\theta}u$ 
on $X$, for every 
$e^{i\theta}\in S^1$}}.
\end{align}
 It is easy to check that for every 
$m\in\mathbb N$ there are canonical isomorphisms
\begin{equation}\label{eq:2.32b}
H^0_{b,m}(X)^G\cong H^0(M,L^m)^G\:\:, \quad
H^0_{b,m}(X_G)\cong H^0(M_G,L_G^m). 
\end{equation}
On account of \eqref{eq:2.32b}, Theorem~\ref{t-180428zy} 
(under Assumption \ref{a-gue170410I} (i))
and the correspondence between the curvature
$iR^L$ and the Levi form of $X$ (cf.\ Section \ref{s2.30})
we deduce: 
\begin{theorem}\label{t-gue181211}
Let $M$ be a compact connected complex manifold, $\dim_\C M\geq2$,
and $(L,h^L)$ be a Hermitian holomorphic line bundle over $M$.
Let $G$ be a connected compact Lie group 
acting holomorphically on $M$ and whose action lifts to $(L,h^L)$.
Suppose that $iR^L$ is positive near $\tilde{\mu}^{-1}(0)$
and $G$ acts freely on $\tilde{\mu}^{-1}(0)$.
Then for $m$ large enough, 
the canonical map between $H^0(M,L^m)^G$ 
and $H^{0}(M_G,L_G^m)$ by restriction
 is an isomorphism,  in particular 
\begin{equation}\label{e-gue181211a}
\dim H^0(M,L^m)^G=\dim H^{0}(M_G,L_G^m).
\end{equation}
\end{theorem} 
If $L$ is positive on the whole $M$, 
this canonical isomorphisms 
between $H^0(M,L^m)^G$ and $H^{0}(M_G,L_G^m)$ was constructed 
in \cite{GS:82,Z} for $m=1$
(see also for the metric aspect of this isomorphism
\cite[(0.27), Corollary 4.13]{MZ} for $m$ large enough).
This implies that the map $\sigma_G$ in 
Theorem \ref{t-180428zy} is actually an isomorphism in the case
of the circle bundle $X$ of $L^*$. Thus we expect in many 
situations (in particular, if additionally $X$ is strictly pseudoconvex), 
that $\sigma_G$ in Theorem \ref{t-180428zy} is an isomorphism.
Theorem \ref{t-gue181211} gives a version of the result of \cite{GS:82,Z}
for large enough powers of $L$ by requiring the positivity of $iR^L$ 
only on $\tilde{\mu}^{-1}(0)$.
In the case $\dim_C M=1$, hence $\dim_C X=3$, 
Assumption \ref{a-gue170410I} (ii) corresponds to $L$ being positive 
everywhere on $M$, so an application of Theorem \ref{t-180428zy}
does not bring anything new.
We will give in Theorem \ref{eq:t6.1} a version for 
almost complex manifolds of Theorem \ref{t-gue181211}.\\

In the rest of Introduction we explain some technical aspects,
in particular some results on the $G$-invariant Szeg\H{o} projection,
to establish Theorem \ref{t-180428zy}.

We introduce a $G$-invariant Hermitian metric $g=g^{\C TX}$
on $X$ as in Lemma \ref{l-gue181212}
which fixes the $L^2$ spaces on $X$. 
The \emph{$G$-invariant Szeg\H{o} projection} is
the orthogonal projection
\begin{equation}\label{e-gue181008}
S_G: L^2(X)\To H^0_b(X)^G
\end{equation} 
with respect to $(\,\cdot\,,\,\cdot\,)$.  
The \emph{$G$-invariant Szeg\H{o} kernel} 
$S_G(x,y)\in \cD'(X\times X)$ is the distribution kernel of $S_G$.  
In Theorem~\ref{t-gue180505I}, we will prove that $S_G$
is a complex Fourier integral operator and in 
Theorem~\ref{t-gue200211yyd}, we will show the regularity property
\begin{equation}\label{e-gue200211yyd}
S_G: \cC^\infty(X)\To H^0_b(X)^G\cap\cC^\infty(X).
\end{equation}
From \eqref{e-gue200211yyd} we conclude that 
$H^0_b(X)^G\cap\cC^\infty(X)$ is dense in $H^0_b(X)^G$. 


Let $S_{X_{G}}:L^2(X_{G})\To H^{0}_{b}(X_{G})$ 
be the orthogonal projection 
with respect to $(\,\cdot\,,\,\cdot\,)_{X_{G}}$ (cf.\ Convention
\ref{conv_met}). 
It follows from general results~\cite{BouSj76,Hsiao08} that
if $\ddbar_{b,X_{G}}$
has closed range in $L^2$,
then $S_{X_G}$ is a Fourier integral operator with complex phase and 
also a pseudodifferential operator on $X_{G}$. In particular, 
\begin{equation}\label{e-gue200212yyd}
S_{X_{G}}: \cC^\infty(X_{G})\To H^0_b(X_{G})\cap\cC^\infty(X_{G}).
\end{equation}
which implies that 
$H^0_b(X_{G})\cap\cC^\infty(X_{G})$ is dense in $H^0_b(X_{G})$
if $\ddbar_{b,X_{G}}$ has closed range in $L^2$.

We consider the linear map 
\begin{equation}\label{e:rx}
\mR_x:\underline{\kg }_x\To\underline{\kg }_x,\quad
u\mapsto \mR_xu,\ \ \langle\,\mR_xu,v\,\rangle
=\langle\,d\omega_0(x),Ju\wedge v\,\rangle.
\end{equation}
For $x\in Y$ we denoty by $Y_x=\set{h.x;\, h\in G}$ 
the $G$-orbit of $Y$; then $Y_x$ is a 
$d$-dimensional submanifold of $X$.
The $G$-invariant Hermitian metric 
$g$ induces a volume form $dv_{Y_x}$ on $Y_x$. Put 
\begin{equation}\label{e-gue180308m}
f_G(x)=\abs{\det\mR_x}^{-\frac{1}{4}}\sqrt{V_{{\rm eff\,}}(x)}\in
\cC^\infty(Y)^G \, \, \,  \text{  with  }
V_{{\rm eff\,}}(x):=\int_{Y_x}dv_{Y_x},
\end{equation}
where 
$\cC^\infty(Y)^G$ 
denotes the space of $G$-invariant smooth functions on
$Y=\mu^{-1}(0)$.
 
Let $E: \cC^\infty(X_{G})\To\cC^\infty(X_{G})$ be a classical
elliptic pseudodifferential operator with principal symbol 
$p_E(x,\xi)=\abs{\xi}^{-d/4}$.
Let 
\begin{equation}\label{e-gue180308Im}
\begin{split}
\sigma:H^0_{b}(X)^G\cap\cC^\infty(X)^G&\To 
H^{0}_{b}(X_{G}),
\quad\sigma=S_{X_{G}}\circ E\circ \iota_G\circ f_G\circ \iota^*\,.
\end{split}
\end{equation}
It turns out that 
the operator $\sigma$ in \eqref{e-gue180308Im} is bounded, 
see Corollary~\ref{c-gue180515p}, and thus
extends by density to a bounded operator
\begin{equation}\label{e:GS}
\sigma: H^0_b(X)^G\To H^0_b(X_{G}).
\end{equation}

We have encoded in the definition \eqref{e-gue180308Im} some
corrections in order to obtain good analytic properties of $\sigma$. 
One correction is the multiplication with the function $f_G$
from \eqref{e-gue180308m}; this reflects the need to
reconcile the volume forms on $\mu^{-1}(0)$ and on $X_G$\,.
The multiplication by $f_{G}$ changes the CR character of the result,
therefore the need to project back to the CR space by $S_{X_G}$.
Here comes the role of $E$, which is more subtle
(see also Remark~\ref{r-gue181220mp} below). 
Ideally, the map $\sigma$ should be unitary.
But we can content ourselves to 
require that $\sigma^{\,*}\sigma$ is ``microlocally close" to 
$S_G$, where $\sigma^{\,*}:  H^0_b(X_{G})\To\cD'(X)$ 
is the adjoint of $\sigma$. In other words, we want 
$\sigma^{\,*}\sigma$ to be a complex Fourier integral 
operator with the same phase, the same order and the same leading 
symbol as $S_G$. To achieve this, we need to take $E$ to be
a classical elliptic pseudodifferential operator with principal symbol 
$p_E(x,\xi)=\abs{\xi}^{-d/4}$.

The main  technical result of this work is the following. 

\begin{theorem}\label{t-gue180428zy}
Under the assumption of Theorem \ref{t-180428zy}, 
the map $\sigma$ is Fredholm. Actually, $\Ker\sigma$ and 
$(\Ran\sigma)^\perp$ 
are finite dimensional subspaces of  
$\cC^\infty(X)\cap H^0_b(X)^G$ and 
$\cC^\infty(X_{G})\cap H^0_b(X_{G})$, 
respectively. 
\end{theorem}

\begin{remark}\label{r-gue181220mp}
Note that the definition of $\sigma$ depends on the choice of the
elliptic pseudodifferential operator $E$. We actually show that for 
any classical elliptic pseudodifferential operator $E$ with the same
principal symbol 
$p_E(x,\xi)=\abs{\xi}^{-d/4}$, the map 
$\sigma: H^0_b(X)^G\To H^0_b(X_{G})$ is Fredholm. 
Up to lower order terms of $E$, the map $\sigma$ is 
a canonical choice. 
The  elliptic pseudodifferential operator $E$ corresponds to the power 
$m^{-d/4}$ in the isomorphism map between 
$H^0(M,L^m)^G$ and $H^{0}(M_G,L_G^m)$ in complex case. 
Here we use the same notations as in the discussion after 
Theorem~\ref{t-gue180520m}. More precisely,  
Ma-Zhang \cite[Theorem 0.9]{MZ} showed that the map 
\begin{equation}\label{e-gue181214my}
\begin{split}
\sigma_m: H^0(M,L^m)^G\To H^{0}(M_G,L_G^m),
\quad\sigma_m=m^{-d/4}B^m_{M_G}
\circ\iota_G\circ f_{G}\circ \iota^*,
\end{split}
\end{equation}
is an asymptotic isometry if $m$ is large enough, where $\iota_G$, 
$\iota^*$ and $f_{G}\in\cC^\infty(M)^G$ 
are defined as in the discussion before \eqref{e-gue180308Im} 
and $B^m_{M_G}:L^2(M_G, L^m_G)\To H^{0}(M_G,L_G^m)$ 
is the orthogonal projection. When we change $m^{-d/4}$ in 
\eqref{e-gue181214my} to any $m$-depend function with order 
$m^{-d/4}+O(m^{-\frac{d}{4}-1})$, 
we still have an isomorphism between $H^0(M,L^m)^G$ and 
$H^{0}(M_G,L_G^m)$ 
for $m$ large. Moreover, in view of Theorem~\ref{t-180428zy}, 
to get $L^2$ isomorphism it makes sense to take 
an elliptic pseudodifferential operator $E$ of order $-\frac{d}{4}$. 
\comment{Comparing \eqref{e-gue181214my} with our map 
\eqref{e-gue180308Im}, the power
$m^{-d/4}+O(m^{-\frac{d}{4}-1})$ 
corresponds to an elliptic pseudodifferential operator $E$. 
Roughly speaking, 
when we sum over all $m$, the power 
$m^{-d/4}+O(m^{-\frac{d}{4}-1})$ 
will become an elliptic pseudodifferential operator $E$ with 
the principal symbol 
$p_E(x,\xi)=\abs{\xi}^{-d/4}$ and different lower order terms 
$O(m^{-\frac{d}{4}-1})$ will only change lower symbols of $E$.}
\end{remark}

\begin{remark}\label{r-gue181023}
In this work, we do not assume that $\ddbar_b$ has closed range in 
$L^2$ on $X$. 
We will show in Section~\ref{s-gue181012} that under the assumption 
that the Levi form is positive on $Y=\mu^{-1}(0)$,
the $G$-invariant Kohn Laplacian has closed range in $L^2$ 
and this is enough to obtain a full asymptotic expansion 
for the $G$-invariant Szeg\H{o} kernel $S_G(x,y)$ 
(see Theorem~\ref{t-gue180505I}). 
In order to show that the $G$-invariant Kohn Laplacian 
has closed range in $L^2$ we need the hypothesis that the dimension
of $X$ is greater than or equal to five.

The asymptotic expansion for 
$S_G$ is also a new result. In~\cite{HsiaoHuang17}, 
Hsiao and Huang obtained an asymptotic expansion for $S_G$ 
under the assumption that $\ddbar_b$ has closed range in $L^2$ on $X$.  
In~\cite{HsiaoHuang17}, Hsiao and Huang established 
"quantization commutes with reduction'' results for CR manifolds 
with $S^1$-action. 
The spaces considered in~\cite{HsiaoHuang17} are finite dimensional. 
To handle the infinite dimensional case, we need to develop a new kind 
of complex Fourier integral calculus.
\end{remark}

This paper is organized as follows. In Section \ref{s:prelim}, 
we review some facts on CR and Sasakian manifolds, 
CR reduction and Szeg\H{o} kernel.
In Section \ref{s-gue180430}, 
we establish asymptotic expansion for the $G$-invariant Szeg\H{o} kernel.
In Section \ref{s-gue180308}, we study the distribution kernel 
of the map $\sigma$ in \eqref{e-gue180308Im} 
and establish Theorem \ref{t-gue180428zy}.
In Section \ref{s-gue200926yyd}, we establish Theorem  \ref{t-180428zy}.
In Section \ref{s:6}, 
we establish a version for almost complex manifolds of 
Theorem \ref{t-gue181211}.
  
{\bf Notation}.  We denote by $\mathbb N=\set{0,1,2,\ldots}$ the set of
natural numbers and set  
$\mathbb N^*=\mathbb N\setminus\set{0}$, $\R_{+}=[0, +\infty[$.
We use standard notations about distributions and Sobolev spaces 
on manifolds, as in \cite{HM14,MM}.
In this paper we will systematically use the correspondence between
operators $A$ and their kernels $A(\cdot,\cdot)=A(x,y)$ via the
Schwartz kernel theorem \cite[Theorems 5.2.1, 5.2.6]{Hor03}, 
\cite[Theorem B.2.7]{MM}. For two distributions $u,v$ we write
$u\equiv v$ if $u-v$ is a smooth function.
For two operators $A$, $B$,
we write $A\equiv B$ if their Schwartz kernels
satisfy $A(\cdot,\cdot)\equiv B(\cdot,\cdot)$, equivalently, if 
$A-B$ is a smoothing operator. 

In the whole paper we will denote by $G$ 
a compact connected Lie group, 
by $\kg$ its Lie algebra, and by $d\mu$ the Haar measure on $G$
with $\int_{G}d\mu (h)=1$.
If $E$ is a complex representation of $G$,
we denote by $E^{G}$ the $G$-trivial component of $E$. 

For an operator $A$ we denote by $\spec A$ the spectrum of $A$.
For a real vector space/bundle $V$ we denote by 
$\C V= V\otimes_{\R}\C$
the associated complexified vector space/bundle. \\

{\bf Acknowledgments}. 
C.-Y. H. is partially supported by Taiwan Ministry of Science and Technology
projects 108-2115-M-001-012- MY5 and 109-2923-M-001-010-MY4.
X.\ M.\ is partially supported by
NNSFC  No.11829102 and
funded through the Institutional Strategy of
the University of Cologne within the German Excellence Initiative.
G.\ M.\   is partially supported by CRC TRR 191.

\section{Preliminaries}\label{s:prelim}

In this Section, we explain some basic facts on CR and Sasakian manifolds, 
 CR reduction and Szeg\H{o} kernel.

\subsection{CR manifolds and CR functions}\label{s-gue180305} 

Let $(X, T^{1,0}X)$ be a compact, connected and orientable CR manifold 
of dimension $2n+1$, $n\geq 1$, where $T^{1,0}X$ is 
a CR structure of $X$, 
that is, $T^{1,0}X$ is a complex vector sub-bundle of rank $n$ of
the complexified tangent bundle 
$\C TX$, satisfying 
\begin{align}\label{eq:2.3}
	T^{1,0}X\cap T^{0,1}X=\{0\}, \, \,
[\mathcal V,\mathcal V]\subset\mathcal V,
\text{ with } T^{0,1}X=\overline{T^{1,0}X}, 
\mathcal V=\cC^\infty(X, T^{1,0}X). 
	\end{align} 
Denote by $T^{*1,0}X$ and $T^{*0,1}X$ the dual bundles of
$T^{1,0}X$ and $T^{0,1}X$, respectively. 
Define the vector bundle of $(0,q)$-forms by 
\begin{align}\label{eq:2.4}
	T^{*0,q}X := \Lambda^q\,T^{*0,1}X. 
\end{align} 
The Levi distribution (or holomorphic
tangent space) $HX$ of the CR manifold $X$ is the real part of 
$T^{1,0}X \oplus T^{0,1}X$,
i.e., the unique sub-bundle $HX$ of $TX$ such that 
\begin{align}\label{eq:2.5}
	\C HX=T^{1,0}X \oplus T^{0,1}X.
	\end{align} 
Let $J:HX\To HX$ be the complex structure given by 
$J(u+\ol u)=iu-i\ol u$, for every $u\in T^{1,0}X$. 
If we extend $J$ complex linearly to $\C HX$ we have
\begin{align}\label{eq:2.6}
	T^{1,0}X \, = \, \left\{ V \in \C HX \,;\, \, JV \, 
=  \,  iV  \right\}.
	\end{align} 
Thus the CR structure $T^{1,0}X$ is determined by
the Levi distribution and we shall also write 
$(X,HX,J)$ to denote the CR manifold $(X,T^{1,0}X)$.

The annihilator $(HX)^0\subset T^*X$
of $HX$ is called the characteristic bundle of the CR manifold. 
Since $X$ is orientable, the characteristic bundle $(HX)^0$
is a trivial real line sub-bundle. We fix a 
global frame of $(HX)^0$, that is,
a real non-vanishing $1$-form
$\omega_0\in\cC^\infty(X,T^*X)$ 
such that $(HX)^0=\R\omega_0$, called characteristic $1$-form.
We have
\begin{align}\label{eq:2.7}
	\langle\,\omega_0(x),u\,\rangle=0, 
\quad	\text{ for any } u\in H_xX,\: x\in X. 
	\end{align} 
Then by \eqref{eq:2.3}, the restriction of $d\omega_0$
on $HX$ is a $(1,1)$-form.
The Levi form of $X$ at $x\in X$ 
is the Hermitian quadratic form on $T^{1,0}_xX$ given 
by    
\begin{equation}\label{eq:2.12}
\cL_x(u,\ol v)=
-\frac{1}{2i}\langle\,d\omega_0(x),u\wedge\ol v\,\rangle
=-\frac{1}{2i} d\omega_{0}(u, \ol v) ,
 \quad \text{ for } u, v\in T^{1,0}_xX.
\end{equation}  
A CR manifold $X$ is said to be  strictly pseudoconvex if 
for every $x\in X$ the Levi form
$\cL_x$ is positive definite (negative definite).
By \eqref{eq:2.12} we see that the definition does not depend 
on the choice of the characteristic $1$-form $\omega_0$.
By a change of sign of $\omega_0$ we can and shall
assume in the sequel that \emph{the Levi form is
positive definite}.
If $X$ is strictly pseudoconvex then $\omega_0$
is a contact form and the Levi distribution $HX$
is a contact structure.

Let $T\in \cC^{\infty}(X, TX)$ be a vector field, called
characteristic vector field, such that
\begin{equation}\label{eq:2.8}
\C TX = T^{1,0}X \oplus T^{0,1}X \oplus \C T .
\end{equation}
and
\begin{equation}\label{eq:2.13}
	i_T\,\omega_0=1. 
	\end{equation}  
Let $g^{\C TX}$ be a Hermitian metric on $\C TX$
such that the decomposition \eqref{eq:2.8} is orthogonal.
For $u,v\in \C TX$ we denote by $\langle u,v\rangle_g$
the inner product given by $g^{\C TX}$ and for 
$u \in \C TX$, we write 
$|u|^2_g := \langle u,  u \rangle_g$. 

The determinant of the Levi form $\cL_x$ at $x\in X$
with respect to $g^{\C TX}$ is defined by
\begin{equation}\label{e:detlev}
\det\cL_x=\lambda_1(x)\ldots\lambda_n(x)\,,
\end{equation}
where $\lambda_1(x), \ldots, \lambda_n(x),$ are the eigenvalues
of $\cL_x$ as Hermitian form on $T^{1,0}_xX$
with respect to the inner product $\langle\,\cdot\,,\cdot\,\rangle_g$ on
$T^{1,0}_xX$.

The Hermitian metric $g^{\C TX}$
on $\C TX$ induces, by duality, a Hermitian metric on 
$\C T^*X$ and also on the bundles of $(0,q)$ forms 
$T^{*0,q}X$, $q=1,2,\ldots, n$. 
We shall also denote the inner product given by these metrics by 
$\langle\,\cdot\,,\cdot\,\rangle_g$. 
The metric $g^{\C TX}$ induces a Riemannian metric $g^{TX}$ on
$TX$ and $g^{TX}$ induces in turn a Riemannian
volume form $dv=dv(x)$ on $X$
and a distance function $d(\cdot,\cdot)$ on $X$. 

The natural global $L^2$ inner product $(\,\cdot\,,\,\cdot\,)$ on 
$\Omega^{0,q}(X)$ 
induced by $dv(x)$ and $\langle\,\cdot\,,\,\cdot\,\rangle_g$ is given by
\begin{equation}\label{e:l2}
(\,u,v\,):=\int_X\langle\,u(x),v(x)\,\rangle_g\, dv(x)\,,
\quad u,v\in\Omega^{0,q}(X)\,.
\end{equation}
We denote by $(L^2_{(0,q)}(X),(\,\cdot\,,\,\cdot\,))$ 
the completion of $\Omega^{0,q}(X)$ with respect to 
$(\,\cdot\,,\,\cdot\,)$ and denote $\norm{\cdot}$ 
the corresponding $L^2$ norm. 
We set $L^2(X):=L^2_{(0,0)}(X)$. 

Let $\ddbar_b:\Omega^{0,q}(X)\To\Omega^{0,q+1}(X)$ 
be the tangential Cauchy-Riemann operators on $X$ which is the 
composition of the exterior differential $d$ and the projection 
$\pi^{0,q+1}:\Lambda^{q+1}(\mathbb CT^{*}X)\to T^{*0,q+1}X$.
We extend $\ddbar_b$ 
to $L^2$ spaces by taking the weak maximal extension: 
\begin{equation}\label{eq:2.14}
\Dom\ddbar_b=\set{u\in L^2_{(0,q)}(X);\, 
\ddbar_bu\in L^2_{(0,q+1)}(X)}\,,\:\:
\Dom\ddbar_b\ni u\longmapsto\ddbar_bu\in L^2_{(0,q+1)}(X).
\end{equation}
The space of $L^2$ CR functions on $X$ is given by
\begin{align}\label{eq:2.15}
H^0_b(X):=\set{u\in L^2(X)=L^2_{(0,0)}(X);\, \ddbar_bu=0}.
\end{align}
The \emph{Szeg\H{o} projection} is the orthogonal projection
\begin{equation}\label{e-gue180512q}
S: (L^2(X), (\,\cdot\,,\,\cdot\,))\To H^0_b(X).
\end{equation} 
The \emph{Szeg\H{o} kernel} $S(x,y)\in \cD'(X\times X)$ 
is the distribution kernel of $S$. 

\subsection{Sasakian manifolds}\label{s2.30} 
We recall here some facts about Sasakian manifolds, cf.\ \cite{BoGa08}. 
Recently, the subject of Sasakian geometry generated a great deal 
of interest due to the study of existence of Sasaki-Einstein metrics, 
and more generally, Sasakian metrics of constant scalar curvature, 
see for example~\cite{CS18}. 

Let $(X,HX,J)$ be an orientable strictly pseudoconvex CR manifold 
of dimension $2n+1$ and let $\omega_0$ be a characteristic $1$-form 
(global frame of the characteristic bundle $(HX)^0$).
The fact that $X$ is strictly pseudoconvex is equivalent to the
condition that $g^{HX}= d\omega_{0}(\cdot, J\cdot)$
defines a $J$-invariant metric on $HX$.
Notice that in this case 
$\omega_{0}\wedge (d \omega_{0})^{n}\neq 0$, 
so $\omega_{0}$ is a contact form on $X$.
The Reeb vector field $R$ associated to the contact form
$\omega_{0}$ is the vector field on $X$ defined by 
\begin{equation}\label{e:reeb}
i_R\,\omega_{0}=1,\quad i_R\,d\omega_{0}=0.
\end{equation}  
We define a Riemannian metric $g_{\omega_{0}}$ 
on $X$ by 
$g_{\omega_{0}}(\cdot,\cdot)=d\omega_{0}(\cdot,J\cdot)+
\omega_{0}(\cdot)\omega_{0}(\cdot)$. 
Associated to the data $(X,\omega_{0},R,J,g_{\omega_{0}})$, 
there is a canonical connection $\nabla$ on $TX$, 
called the Tanaka-Webster connection
(see Tanaka \cite{T} and Webster \cite{W}), 
that is the unique affine connection on $TX$ such that 
\begin{itemize} 
 \item $\nabla g_{\omega_{0}}=0$, $\nabla J=0$, 
 $\nabla\omega_{0}=0$. 
 \item For any $u$, $v$ in the Levi distribution $HX$, 
 the torsion $T_{\nabla}$ of $\nabla$ satisfies  
 $T_{\nabla}(u,v)=d\omega_{0}(u,v) R$ 
 and $T_{\nabla}(R,Ju)=JT_{\nabla}(R,u)$. 
\end{itemize} 
\begin{definition}\label{Sas} 
A strictly pseudoconvex manifold $(X,\omega_{0},R,J,g_{\omega_{0}})$
is called a Sasakian manifold 
if the torsion of its Tanaka-Webster connection in the direction 
of the Reeb vector field vanishes: $T_{\nabla}(R,\cdot)=0$. 
\end{definition} 
It follows from \cite[Proposition 3.1]{BH05} that
this definition is equivalent to the more traditional definition
that the metric cone $(C(X)=\R_+\times X,dr^2+r^2g_{\omega_0})$
is a K\"ahler manifold \cite[Definition 6.1.15]{BoGa08}
or equivalently, $X$ is a normal contact metric
manifold \cite[Definition 6.1.13]{BoGa08}.

Sasakian manifolds can be classified in three categories 
based on the properties of the Reeb foliation $\mathcal{F}_R$ consisting
of the orbits of the Reeb field (see \cite[Definition 6.1.25]{BoGa08}). 
If the orbits of the Reeb field are all closed, then 
the Reeb field $R$ generates a locally free, isometric $S^1$-action 
on $(X,g_{\omega_0})$ and the Reeb foliation is called quasi-regular.
If this $S^1$-action is free, then the Reeb foliation
is said to be regular. 
If $\mathcal{F}_R$ is not quasi-regular, it is said to be irregular.
In this case, $R$ generates a transversal CR $\R$-action 
on $X$. 

If $\mathcal{F}_R$ is quasi-regular, then by the structure theorem
\cite[Theorem 7.1.3]{BoGa08} the quotient
space $M\coloneq X/\mathcal{F}_R=X/S^1$ is a K\"ahler
orbifold and the quotient map $\pi:X\to M$ an orbifold Riemannian 
submersion.
Moreover, $X$ is the total space of a principal $S^1$ bundle 
over $M$ with connection $1$-form $\omega_0$; there exists
an integral K\"ahler form $\omega$ on $M$ such that the curvature 
$d\omega_0$ of $\omega_0$ is the pullback by the quotient map
of $\omega$: $d\omega_0=\pi^*\omega$.  
If $\mathcal{F}_R$ is regular, then $(M,\omega)$
is a K\"ahler manifold and the assertions above are the
content of the Boothby-Wang theorem \cite{BW:58}. 

\begin{remark}\label{eq:t2.1a} For a Sasakian manifold
$(X,\omega_{0},R,J,g_{\omega_{0}})$, from Definition \ref{Sas},
\begin{align}\label{eq:2.10a}
d \omega_{0}(R,u)=0 \quad \text{for any } u\in HX.
\end{align}
By \eqref{e:reeb} and \eqref{eq:2.10a}, 
$[R, u]\in \cC^{\infty}(X, HX)$ for any $u\in \cC^{\infty}(X, HX)$,
i.e., the flow associated with $R$ preserves $HX$.
Moreover, by \cite[Corollary 6.5.11, Definition 6.5.13]{BoGa08}
it follows that $\mL_{R}J=0$, where $\mL$ denotes the Lie derivative.
Combining with \eqref{e:reeb} and \eqref{eq:2.10a}, 
we get $\mL_{R}g_{\omega_{0}}=0$, i.e., the  
Reeb vector field is a Killing vector field 
on $(X, g_{\omega_{0}})$. 
\end{remark}
Related to Theorem \ref{t-gue181211} and
motivated  by the structure of quasi-regular Sasakian manifolds
let us consider now the case of a circle bundle associated
to a Hermitian holomorphic line bundle.
Let $(L,h^L)$ be a Hermitian holomorphic line bundle over a connected 
compact complex manifold $(M,J)$.
Let $h^{L^*}$ be the Hermitian metric 
on $L^*$ induced by $h^L$. 
Let 
\begin{equation}\label{eq:GraTube}
X:=\set{v\in L^*;\, \abs{v}^2_{h^{L^*}}=1}
\end{equation}
be the circle bundle of $L^*$ (Grauert tube); 
it is isomorphic to the $S^1$ principal bundle associated to $L$.
Since $X$ is a hypersurface in the complex manifold $L^*$,
it a has a CR structure inherited from the complex structure of $L^*$
by setting $T^{1,0}X= TX\cap T^{1,0} L^{*}$.

In this situation,  $S^{1}$ acts on $X$ by fiberwise multiplication,
denoted $(x,e^{i\theta})\mapsto xe^{i\theta}$.
A point $x\in X$ is a pair $x=(p,\lambda)$, where $\lambda$ is a linear
functional on $L_p$, the $S^1$ action is 
$xe^{i\theta}=(p,\lambda)e^{i\theta}=(p,e^{i\theta}\lambda)$.

On $X$ we have a globally defined vector field $\partial_\theta$,
the generator of the $S^1$ action. The span of $\partial_\theta$
defines a rank one subbundle $T^VX\cong TS^1\subset TX$, 
the vertical subbundle of the fibration $\pi:X\to M$. 
Moreover \eqref{eq:2.8} holds for $T=\partial_\theta$.

For $m\in\Z$ the space $\cC^\infty(X,L^m)$ of 
smooth sections of $L^m$ can be 
identified to the space $m$-equivariant smooth functions
\[\cC^\infty (X)_m=\{ f\in \cC^\infty(X,\C): 
f(xe ^{i\theta})=e ^{im\theta} f(x),\, \, {\rm for}\,\,
e ^{i\theta}\in S^1,\, x\in X \}.
\]
by 
\begin{equation}\label{e:mequiv1}
\cC^\infty(X,L^m)\ni s\mapsto f\in\cC^\infty (X)_m,\quad 
f(x)=f(p,\lambda)=\lambda^{\otimes m}(s(p)),
\end{equation} 
where $\lambda^{m}=\lambda^{\otimes m}$ for $m\geq0$
and $\lambda^{m}=(\lambda^{-1})^{\otimes(-m)}$ for $m<0$.
Through the identification \eqref{e:mequiv1} 
holomorphic sections correspond to CR functions:
\begin{equation}\label{e:mequiv2}
H^0(X,L^m)\cong H^0_{b,m}(X):=\{f\in\cC^\infty (X)_m:
\overline\partial_b f=0 \}.
\end{equation} 
We construct now a Riemannian metric on $X$. 
Let $g^{TM}$ be a $J$-invariant metric on $M$.
The Chern connection $\nabla ^L$ on  $L$ induces a connection 
on the $S^1$-principal bundle $\pi:X\to M$, 
and let $T^H X \subset TX$ be the corresponding horizontal bundle.
Let $g^{TX}= \pi^*g^{TM}\oplus \frac{d\theta^2}{2\pi}$
be the metric on 
$TX= T^HX\oplus T S^1$, with $d\theta^2$ 
the standard metric on $S^1= \R/2\pi\Z$.

Pertaining to $g^{TX}$ we construct the $L^2$ inner product 
$(\cdot,\cdot)_X$ given by \eqref{e:l2} on $X$. 
The metric $g^{TM}$ induces a Riemannian volume
form $dv_M$ on $M$, which together with the fiber metric $h^{L^{m}}$
gives rise to an $L^2$ inner product $(\cdot,\cdot)_m$ on
$\cC^\infty(X,L^m)$. Then the isomorphism \eqref{e:mequiv1}
becomes an isometry 
$(\cC^\infty(M,L^m),(\cdot,\cdot)_m)\cong
(\cC^\infty (X)_m,(\cdot,\cdot)_X)$
and accordingly an isometry 
$L^2(M,L^m)\cong L^2(X)_m$, where the latter space
is the completion of $(\cC^\infty (X)_m,(\cdot,\cdot)_X)$.
Moreover, \eqref{e:mequiv1} induces an isometry
\begin{equation}\label{e:mequiv3}
(H^0(M,L^m),(\cdot,\cdot)_m)\cong(H^0_{b,m}(X),(\cdot,\cdot)_X).
\end{equation} 
The $S^1$-action gives rise to a Fourier decomposition
$L^2(X)\cong\widehat{\bigoplus}_{m\in\Z}\,L^2(X)_m$
and this induces the following decomposition at the level
of CR functions:
\begin{equation}\label{eq:2.17}
H^0_{b}(X)\cong\widehat{\bigoplus}_{m\in\N} H^0_{b,m}(X)
\cong\widehat{\bigoplus}_{m\in\N} H^0(M,L^m).
\end{equation}
Let $\omega_0$ be the connection $1$-form on $X$
associated to the Chern connection $\nabla^L$. 
Then $\omega_0(\partial_\theta)=1$, thus \eqref{eq:2.8} and
\eqref{eq:2.13} are fullfiled and $T=\partial_\theta$
is a characteristic vector field on $X$ and $\omega_0$
is a characteristic $1$-form for the CR structure on $X$. Moreover,
\begin{equation}\label{eq:curvx}
d\omega_0=\pi^*(iR^L),
\end{equation}
where $R^L$ is the curvature of $\nabla^L$. On account of
\eqref{eq:2.12} $X$ is strictly pseudoconvex at $x\in X$ if and only if 
$(L,h^L)$ is positive at $\pi(x)\in M$.
In particular, if $(L,h^L)$ is positive on $M$, $X$ is a strictly pseudoconvex
CR manifolds, $\omega_0$ is a contact form and $\partial_\theta$ is the
associated Reeb vector field; $X$ is a regular Sasakian manifold.
Note also that in this case $H^0(X,L^m)=0$ for $m<0$
by the Kodaira vanishing
theorem, so the decomposition \eqref{eq:2.17} becomes
\begin{align}\label{eq:2.17a}
H^0_{b}(X)\cong\widehat{\bigoplus}_{m\in\N} H^0_{b,m}(X)
\cong\widehat{\bigoplus}_{m\in\N} H^0(M,L^m).
\end{align}
Note further that $X$ is (weakly) pseudoconvex, that is, the 
Levi form is positive semidefinite on $X$ 
if and only if the curvature $iR^L$ is semi-positive on $M$.
Moreover, $X$ has finite type if and only if $R^L$ vanishes
to finite order at any point of $M$, 
cf.\ \cite{HS20}, \cite[Proposition 11]{MS18}.

\subsection{CR reduction}\label{s2.3} 
We refer to \cite{Al89,BoGa08} for the fundamentals 
of contact reduction and examples.
Let $(X, HX, J)$ be a compact connected and orientable 
CR manifold of dimension $2n+1$, $n\geq1$,
and let $\omega_0$ be a characteristic $1$-form. 

Let $G$ be a $d$-dimensional connected compact 
Lie group with Lie algebra $\mathfrak{g}$.
We assume that $G$ acts smoothly on $X$ and 
that the $G$-action preserves $J$ and $\omega_0$.

For any $\xi \in \kg$, we denote 
$\xi_X(x)=\frac{\partial}{\partial t}\exp(-t\xi)x\big|_{t=0}$ 
the vector field on $X$ induced by $\xi$.
	For $x\in X$, set
	\begin{equation}\label{e:g}
\underline{\kg}_{x}={\rm Span\,}\big\{\xi_X(x);\, \xi\in\kg\big \}.
\end{equation} 

\begin{definition}\label{d-gue170410}
The moment map associated to the characteristic $1$-form $\omega_0$ 
is the map $\mu:X \to \mathfrak{g}^*$ defined by
\begin{equation}\label{E:cmpm}
\langle \mu(x), \xi \rangle = \omega_0(\xi_X(x))\,,\quad
x \in X, \:\:\xi \in \mathfrak{g}\,.
\end{equation}
\end{definition}

\noindent
The moment map is $G$-equivariant, i.e.,
for $x\in X$, $h\in G$, we have
\begin{align}\label{eq:2.20}
	\mu(h.x)= \Ad_{h}^{*}\mu(x).
\end{align}
Relation \eqref{eq:2.20} implies that $G$ acts on $\mu^{-1}(0)$.
In fact,  for any $\xi\in \kg$, we have 
\begin{multline}\label{eq:2.21}
\langle \mu(h.x), \xi \rangle = \omega_0(\xi_X(h.x)))
=\omega_0( dh (\Ad_{h^{-1}}\xi)_X))_{h.x} \\
= (h^{*}\omega_{0})((\Ad_{h^{-1}}\xi)_X)_{x}
= \omega_{0}((\Ad_{h^{-1}}\xi)_X)_{x}
=\langle \mu(x), \Ad_{h^{-1}}\xi \rangle 
=\langle\Ad_{h}^{*}\mu(x), \xi \rangle .
\end{multline}

\begin{lemma}\label{eq:t2.4} If $G$ acts freely on $\mu^{-1}(0)$ and 
	the Levi form  is positive on $\mu^{-1}(0)$, then $0$ is a 
	regular value of $\mu$.
\end{lemma}
\begin{proof} Observe first that since $\omega_{0}$ is $G$-invariant, 
\eqref{eq:2.7} yields that $HX$ is an $G$-equivariant sub-bundle of 
	$TX$, thus
	\begin{align}\label{eq:2.22}
[\xi_{X}, U]\in \cC^{\infty}(X, HX)
\text{ for any } U\in \cC^{\infty}(X, HX), \xi\in \kg. 
\end{align}	
Now \eqref{eq:2.7} and \eqref{E:cmpm} entail
	\begin{align}\label{eq:2.23}
\xi_{X}(x)\in  HX \text{ if }  x\in \mu^{-1}(0). 
\end{align}	
Let $U_{x}\in H_{x}X$ that
we extend to a section $U$ of $HX$ near $x$.
Then \eqref{eq:2.7} and \eqref{eq:2.22} yield
	\begin{align}\label{eq:2.24}
d\omega_{0} (U,\xi_{X})_{x}= U(\omega_{0}(\xi_{X}))
- \xi_{X}(\omega_{0}(U)) - \omega_{0}([U, \xi_{X}])
= U_{x}\langle \mu, \xi \rangle 
= \langle U(\mu)_{x}, \xi \rangle .
	\end{align}	
If $d\mu_{y}: HX\to \kg^{*}$
were not surjective for some $y\in \mu^{-1}(0)$, 
there would exist $\xi\in \kg$ such that 
$\langle Y(\mu)_{y}, \xi \rangle=0$ for any $Y\in H_{y}X$.
This is a contradiction since $d\omega_{0}$ is nondegenerate on 
$H_{y}X$ and $0\neq \xi_{X,y}\in H_{y}X$ by \eqref{eq:2.23}.
Thus $d\mu_{y}: HX\to \kg^{*}$ is surjective for $y\in \mu^{-1}(0)$.
	\end{proof}
	
Set as in \eqref{eq:2.25} $X_{G}= Y/G$ with $Y = \mu^{-1}(0)$.
Let $\iota: Y\to X$ be the natural injection
and let $\pi: Y\to X_{G}=Y/G$ be the natural projection.

\begin{theorem}\label{eq:t2.5}
If $G$ acts freely on $Y=\mu^{-1}(0)$ and the Levi form  is positive 
on $\mu^{-1}(0)$, then the reduced space $X_{G}= Y/G$ is
a strictly pseudoconvex manifold with contact form
$\omega_{0,G}$ satisfying $\iota^*\omega_0=\pi^*\omega_{0,G}$.	
Moreover, we can choose the characteristic vector field $T$
(cf.\ \eqref{eq:2.8}, \eqref{eq:2.13}) such that  
$T|_{Y}\in \cC^{\infty}(Y, TY)$
and $T$ is $G$-invariant.	
		\end{theorem}
\begin{proof}    By Lemma \ref{eq:t2.4}, 
$\mu^{-1}(0)$ is a smooth manifold. Since $G$ acts freely on $Y$, 	
	$X_{G}$ is a compact manifold.
	The positivity of the Levi form on $\mu^{-1}(0)$  means that
	\begin{align}\label{eq:2.27}
		g^{HX}= d\omega_{0}(\cdot, J\cdot)
	\end{align}		
is a $J$-invariant and $G$-equivariant metric on $HX$ 
on a neighborhood of $Y=\mu^{-1}(0)$.

Since $G$ acts freely on $Y$, the vector spaces $\underline{\kg}_{x}$ 
defined in \eqref{e:g} form a vector bundle $\underline{\kg}$
near $\mu^{-1}(0)$. We denote $\underline{\kg}_{Y}= 
\underline{\kg}|_{Y}$. 
Then $\underline{\kg}_{Y} \subset TY\cap HX$ by \eqref{eq:2.23}.

For $x\in \mu^{-1}(0)$, by \eqref{eq:2.23}, \eqref{eq:2.24}, 
and the fact that $d\omega_{0}(\cdot, J\cdot)$ is a metric on $H_{x}X$ 
we have that
$d\mu|_{TY}=0$ and $d\mu|_{J\underline{\kg}_{x}}\to \kg^{*}$
is surjective.
Since $\dim Y+\dim\underline{\kg}=\dim TX$, we have
	\begin{align}\label{eq:2.28}
J\underline{\kg}|_{Y} \oplus TY = TX|_{Y}.
	\end{align}		
From \eqref{eq:2.28} and $J\underline{\kg}|_{Y}\subset  HX$,
we know $\omega_{0}(TY)\neq 0$. Thus $TY\cap HX$ is a codimension
$1$ sub-bundle of $TY$, and 
\begin{align}\label{eq:2.31}	
\frac{TY}{TY\cap HX}= \frac{TX}{HX}\Big|_{Y}.
\end{align}	
From \eqref{eq:2.31}, we can choose the vector field $T$
in \eqref{eq:2.8} such that  $T|_{Y}\in \cC^{\infty}(Y, TY)$
and $T$ is $G$-invariant.
	
Let $T^{H}Y$ be the orthogonal complement of $\underline{\kg}_{Y}$
in $TY\cap HX$ with respect to $g^{HX}$. 
By \eqref{eq:2.24}, \eqref{eq:2.27}
and $TY\cap HX$ is a codimension $1$ sub-bundle of $TY$,  
we have the $G$-equivariant
orthogonal decomposition on $Y$,
	\begin{align}\label{eq:2.29}
TY\cap HX = T^{H}Y\oplus 	\underline{\kg}_{Y},
\quad HX|_{Y} =  T^{H}Y\oplus 	\underline{\kg}_{Y}
\oplus J\underline{\kg}|_{Y} .
	\end{align}	
Thus 	from \eqref{eq:2.29} and the metric $g^{HX}$ on  $HX|_{Y}$
is $J$-invariant, we get
\begin{align}\label{eq:2.30}	
JT^{H}Y = T^{H}Y  = (TY\cap HX)\cap J  (TY\cap HX). 
	\end{align}		
By \eqref{E:cmpm} $\iota^{*} \omega_{0}$ is 
a $G$-invariant horizontal $1$-form on $Y$, thus 
there exists a unique $1$-form $\omega_{0,G}\in \Omega^{1}(X_{G})$
such that
	\begin{align}\label{eq:2.32}	
\iota^{*} \omega_{0}= \pi^{*} \omega_{0,G}.
	\end{align}	
We now define the Levi distribution on $X_G$ by
\begin{align}\label{eq:2.33}	
HX_{G}\coloneq \ker  \omega_{0,G}.
	\end{align}	
From \eqref{eq:2.31}, \eqref{eq:2.29},
$d\pi: \R T\oplus T^{H}Y\to TX_{G}$ is bijective, and we get 
the isomorphism $\R T\oplus T^{H}Y\simeq \pi^{*} TX_{G}$.
Thus	$d\pi$ maps $T^{H}Y$ onto $HX_{G}$
and this gives an isomorphism 
$T^{H}Y\simeq \pi^{*} 	HX_{G}$.
Thus for $U\in H_{y}X_{G}$, we take $x\in \pi^{-1}(y)$
and $U^{H}\in T^{H}Y$ the lift of $U$, 
then by  \eqref{eq:2.30}, we define $J_{G}\in \End(HX_{G})$ by 
\begin{align}\label{eq:2.34}	
(J_{G}U)^{H}  = J U^{H}.
	\end{align}	
From \eqref{eq:2.32}, we have 
\begin{align}\label{eq:2.35}	
\iota^{*} d\omega_{0}= \pi^{*} d\omega_{0,G}.
	\end{align}
Thus from \eqref{eq:2.30}, $\iota^{*}d\omega_{0}(\cdot,J\cdot)$
is positive and $G$-invariant on $T^{H}Y$ implies
that $d\omega_{0,G}(\cdot,J_{G}\cdot)$ is positive 
and $J_{G}$-invariant on $HX_{G}$.
We verify now that 
\begin{equation}\label{e-gue200212ycd}
T^{1,0}X_{G}=\set{u - \sqrt{-1}J_{G}u;\, u\in HX_{G}}.
\end{equation}
defines a CR structure on $X$.  
For $U,V\in \cC^{\infty}(X_{G}, HX_{G})$, 
from \eqref{eq:2.34}, 
\begin{align}\label{eq:2.36}	
	(U - \sqrt{-1}J_{G}U)^{H}= U^{H}- \sqrt{-1} J U^{H}
\in \cC^{\infty}(Y, T^{1,0}X\cap \C TY), 
		\end{align}
		thus by \eqref{eq:2.3}, 
$	[U^{H}- \sqrt{-1} J U^{H}, V^{H}- \sqrt{-1} J V^{H}]
	\in \cC^{\infty}(Y, T^{1,0}X\cap \C TY)$.
	By  \eqref{eq:2.29},  
	$T^{1,0}X\cap \C TY= \{v - \sqrt{-1}J v; v\in T^{H}Y\}$.
	Thus there exists $W\in \cC^{\infty}(X_{G}, HX_{G})$
	such that
	\begin{align}\label{eq:2.37}	
	\Big[U^{H}- \sqrt{-1} J U^{H}, V^{H}- \sqrt{-1} J V^{H}\Big]
=W^{H}- \sqrt{-1} J W^{H} .
		\end{align}
From \eqref{eq:2.36}, \eqref{eq:2.37},  we obtain
\begin{equation}\label{eq:2.38}
\begin{split}	
 \Big[U - \sqrt{-1}J_{G}U, V - \sqrt{-1}J_{G}V\Big]
&= d\pi \Big[U^{H}- \sqrt{-1} J U^{H}, V^{H}- \sqrt{-1} J V^{H}\Big]\\
&= W- \sqrt{-1} J_{G} W.
\end{split}
\end{equation}		
i.e., $[\cC^{\infty}(X_{G},T^{1,0}X_{G}), 
\cC^{\infty}(X_{G},T^{1,0}X_{G})]\subset 
\cC^{\infty}(X_{G},T^{1,0}X_{G})$. 
Let us finally note that \eqref{eq:2.31} shows that 
we can choose the characteristic vector field $T$
(cf.\ \eqref{eq:2.8}, \eqref{eq:2.13}) such that 
$T|_{Y}\in \cC^{\infty}(Y, TY)$ and $T$ is $G$-invariant.
The proof of Theorem \ref{eq:t2.5} is completed.
\end{proof}
	
In the rest of this paper we will work under Assumption 
\ref{a-gue170410I}.

\begin{lemma}\label{l-gue181212}
Under Assumption 
\ref{a-gue170410I} there is a $G$-invariant Hermitian metric 
$g=g^{\C TX}$ on $\C TX$ so that 
\begin{itemize}
\item[(i)] $T^{1,0}X$ is orthogonal to $T^{0,1}X$,
\item[(ii)] $\underline{\kg }$ is orthogonal to $HY\cap JHY$ 
at every point of $Y$,
\item[(iii)] $\langle\,T,T\,\rangle_g=1$,
\item[(iv)] $T$ is orthogonal to $T^{1,0}X\oplus T^{0,1}X$,
\end{itemize}
where on $Y$, $HY:=HX\cap TY$. 
\end{lemma}

\begin{proof} This follows from the proof of Theorem \ref{eq:t2.5}.	
Let $U$ be a $G$-invariant  neighborhood of $Y$ so that the Levi form
is positive definite on $U$.
Then the metric $g^{HX}= d\omega_{0}(\cdot, J\cdot)$ 
is a $J$-invariant and $G$-equivariant metric on $HX$ and 
we have the orthogonal decomposition \eqref{eq:2.29} on $Y$.
Now we extend the metric $g^{HX}$ from $U$ to $X$
as a $J$-invariant and $G$-equivariant metric on $HX$ by a
partition of unity argument. 
Thus we can take $g^{TX}$ on $TX= \R T\oplus HX$
	as the direct sum 	metric on  $(HX, g^{HX})$ and    
	$(\R T, \langle T,T\rangle=1)$.
\end{proof}
\begin{convention}\label{conv_met}
From now on we fix a $G$-invariant Hermitian metric $g=g^{\C TX}$
on $\C TX$ so that (i)-(iv) in Lemma~\ref{l-gue181212} hold. 
This metric induces natural Hermitian metrics 
 $\langle\,\cdot\,,\,\cdot\,\rangle_{X_{G}}$ 
 on $\C TX_{G}$ and $\C T^*X_{G}$. 
As in \eqref{e:l2} we define the $L^2$ inner products
and spaces induced by $g$ on $X$ and $X_G$
by $(L^2_{(0,q)}(X),(\,\cdot\,,\,\cdot\,))$ and 
$(L^2_{(0,1)}(X_{G}),(\,\cdot\,,\,\cdot\,)_{X_{G}})$.
\end{convention}

\subsection{Closed range in $L^{2}$ for $\ddbar_b$ and 
Szeg\H{o} projections}\label{s2.4}  

The property of closed range in $L^2$ for $\ddbar_b$ in \eqref{eq:2.14}
plays an important role in CR geometry.
It follows from the works of Boutet de Monvel \cite{BdM1:74}, 
Boutet de Monvel-Sj\"ostrand \cite{BouSj76},
Harvey-Lawson \cite{HaLa75}, Burns \cite{Bur:79} and 
Kohn \cite{Koh86}
that the conditions below are equivalent for a compact strictly 
pseudoconvex
CR manifold $X$, $\dim_{\mathbb R}X\geqslant3$:
\begin{enumerate}
\renewcommand{\labelenumi}{\rm(\alph{enumi})}
\item $X$ is embeddable in the Euclidean space $\C ^N$, 
for $N$ sufficiently large;
\item $X$ bounds a strictly pseudoconvex complex manifold;
\item The tangential Cauchy-Riemann operator 
$\ddbar_b: \Dom  \ddbar_b\subset L^2(X)\To L^2_{(0,1)}(X)$ 
on functions has closed range.
\end{enumerate}
If $X$ is a compact strictly pseudoconvex CR manifold of dimension 
greater than or equal to five, 
then $X$ satisfies condition (a), by the embedding theorem
of Boutet de Monvel \cite{BdM1:74}. 
However, there are examples of non-embeddable compact 
strictly pseudoconvex CR manifolds of dimension three given by
Grauert, Andreotti-Siu and Rossi \cite{AS:70,Gra:94,Ros:65}, see
also Col\c{t}oiu-Tib\u{a}r \cite{CT09}. 
In fact this happens for arbitrarily small perturbations of the standard 
CR structure on the unit sphere in $\C^2$.
For these examples the closed range in $L^2$ property fails.

Assume that condition (b) is satisfied and let $M$ be
a strictly pseudoconvex complex manifold such that $\partial M=X$.
If $u$ is continuous on $\overline{M}$ and holomorphic on $M$, then
$u|_X$ satisfies the tangential Cauchy-Riemann equations 
$\ddbar_b(u|_X)=0$.
Conversely, by \cite{KR65} any smooth function $u$ on $X$ satisfying
$\ddbar_bu=0$ admits a smooth extension $\widetilde{u}$ 
to $\overline{M}$ which is holomorphic in $M$. 
In this sense the space $H^0_b(X)\cap\cC^\infty(X)$
is the space of boundary values of holomorphic functions 
$\mO(M)\cap\cC^\infty(\overline{M})$. Note also that the Hardy space
$H^0_b(S^1)$ consists of boundary values of holomorphic functions
on the unit disc in $\C$ cf.\ \cite[Theorem 17.10]{Rud87}.  
This is the unifying feature of our definition of $H^0_b(X)$
for strictly pseudoconvex $X$ and $X=S^1$.
 
There are important classes of embeddable compact strictly 
pseudoconvex 
three dimensional CR manifolds (for which $\overline\partial_b$ 
has thus closed 
range in $L^2$) carrying interesting geometric structures such as
\begin{itemize}
\item transverse CR $S^1$-actions \cite{BlDu91,Eps92,Lem92},
\item conformal structures  \cite{Biq02}, 
\item Sasakian structures 
(transverse CR $\mathbb R$-actions) \cite{MY07},
\end{itemize}

If $X$ is a compact strictly pseudoconvex CR manifold and $\ddbar_b$ 
has closed range in $L^2$, 
Boutet de Monvel-Sj\"ostrand~\cite{BouSj76} showed that $S(x,y)$
is a Fourier integral operator with complex phase. In particular,
$S(x,y)$ is smooth outside the diagonal of $X\times X$
and there is a precise description of the singularity on the diagonal $x=y$,
where $S(x,x)$ has a certain asymptotic expansion. 
Hsiao~\cite[Theorem 1.2]{Hsiao08} generalized 
Boutet de Monvel-Sj\"ostrand's result to $(0,q)$ forms when 
the Levi form is non-degenerate and Kohn Laplacian for $(0,q)$ forms has 
closed range in $L^2$.  If the Levi form is degenerate 
(for example $X$ is weakly pseudoconvex),
Hsiao and Marinescu~\cite[Theorem 1.14]{HM14} showed that the 
Szeg\H{o} projector $S$ is a complex Fourier integral 
operator on the subset where the Levi form is positive definite
if $\ddbar_b$ has closed range in $L^2$. 
 
Let
\begin{equation}\label{e-suVIII}
\ol{\pr}^{*}_{b}:\Dom  \ol{\pr}^{*}_{b}\subset L^2_{(0,1)}(X)
\To L^2(X)
\end{equation}
be the Hilbert space adjoint of $\ddbar_{b}$ in the $L^2$ space
with respect to $(\,\cdot\,,\,\cdot\, )$.
Let $\Box_{b}$ denote the (Gaffney extension) of the Kohn Laplacian 
on functions given by
\begin{equation}\label{e-suIXzq}
\begin{split}
\Dom\Box_{b}&=\Big\{u\in L^2(X);\, 
u\in\Dom  \ddbar_{b},\, 
\ddbar_{b}u\in\Dom  \ol{\pr}^{*}_{b}\Big\}\,,\\
\Box_{b}u&=\ol{\pr}^{*}_{b}\ddbar_{b}u
\:\:\text{for $u\in\Dom\Box^{(0)}_{b}$}\,.
\end{split}
\end{equation}
By a result of Gaffney, $\Box_{b}$ is a positive self-adjoint operator 
(see \cite[Proposition\,3.1.2]{MM}). 
In particular, the spectrum $\spec\Box_b$ of $\Box_{b}$
is contained in $[0,\infty)$.
For a Borel set $B\subset\mathbb{R}$ we denote by 
$E(B)$ the spectral projection of $\Box_{b}$ corresponding to
the set $B$, where $E$ is the spectral measure of $\Box_{b}$.
For $\lambda\geq0$, we set
\begin{equation} \label{e-suX}
H^0_{b,\leq\lambda}(X):=\Ran E\bigr((-\infty,\lambda]\bigr)
\subset L^2(X)\,,
\end{equation}
and let
\begin{equation}\label{e-suXI-I}
S_{\leq\lambda}:L^2(X)\To H^0_{b,\leq\lambda}(X),
\end{equation}
be the orthogonal projection with respect to the product 
$(\,\cdot\,,\,\cdot\,)$ and let
\begin{equation}\label{e-suXI-II}
S_{\leq\lambda}(x,y)\in\cD'(X\times X)
\end{equation}
denote the distribution kernel of $S_{\leq\lambda}$. 
For $\lambda=0$, we write $S:=S_{\leq0}$, $S(x,y):=S_{\leq0}(x,y)$.

Without the assumption that the range of $\ddbar_b$ is closed,
$\Ker\ddbar_b$ could be trivial and therefore
it is natural to consider the spectral projection
$S_{\leq\lambda}$ for $\lambda>0$.
\begin{theorem}[{\cite[Theorem 1.5]{HM14}}]\label{t-gue140305VIc}
For any $\lambda>0$ 
the spectral projector $S_{\leq\lambda}$ is a complex Fourier integral 
operator on the subset where the Levi form is positive definite.
\end{theorem}
Theorems \ref{t-gue180505} and \ref{t-gue161110g}
below are more detailed statements of this result.
Since we don't assume that $\ddbar_b$ has closed range in $L^2$ on $X$, 
Theorem~\ref{t-gue140305VIc} plays an important role in this work. 
We only assume that the Levi form is non-degenerate on $Y$ and 
we will show in Theorems~\ref{t-gue181017} and 
\ref{t-gue181021a} that the $G$-invariant tangential 
Cauchy-Riemann $\ddbar_{b,G}$ has closed range in $L^2(X)^G$ 
and we have
\begin{equation}\label{e-gue181021}
S_G(x,y)=\int_GS_{\leq\lambda_0}(x,h\circ y)d\mu(h)
\end{equation}
for some $\lambda_0>0$, where $d\mu(h)$ is the Haar measure on $G$
with $\int_Gd\mu(h)=1$. From \eqref{e-gue181021}, 
we can apply Theorem~\ref{t-gue140305VIc} to study 
$S_G$ without a closed range in $L^2$ assumption on 
$\ddbar_b$ on $X$. 



\section{$G$-invariant Szeg\H{o} kernel asymptotics}\label{s-gue180430}

In this section, we will establish asymptotic expansion 
for the $G$-invariant Szeg\H{o} kernel. 
From now on, we work under Assumption~\ref{a-gue170410I}.
We do not assume that 
$\ddbar_{b,X_{G}}$ has closed range in $L^2$. 
We first estimate the Szeg\H{o} kernel outside $Y$. 
From now on, we will use the same notations as 
in Sections ~\ref{s-gue170410}, ~\ref{s:prelim}. 

\subsection{Subelliptic estimates for $G$-invariant smooth functions 
away $Y$}\label{s-gue161109I} 
In this Section the manifold $X$ is supposed to
have arbitrary dimension $\geq3$. 
Let 
\[L_1,\ldots,L_{N}\in\cC^\infty(X,HX),\ \ N\in\N^*,\] 
such that for any $x\in X$, $\set{L_1(x),\ldots,L_N(x)}$ span $H_xX$.
Let $s\in\N^*$. For $u\in\cC^\infty(X)$, we define 
\begin{equation} \label{e-gue180430a}
\vvvert u\vvvert_{s}:=
\sum^s_{\nu=1}\,
\sum_{1\leq j_1,\ldots,j_\nu\leq N}
\norm{L_{j_1}L_{j_2}\ldots L_{j_\nu}u}+\norm{u}.
\end{equation}
Note first that the method of proof of 
\cite[Theorem 8.3.5]{CS01}
yield the following. We don't need here the condition $Y(1)$. 
\begin{theorem}
\label{t-gue180430}
 There exists $C>0$ such that for all $u\in\cC^\infty(X)$, 
\begin{equation}\label{e-gue180430aI}
\vvvert u\vvvert_1^2\leq C\left((\,\Box_bu,u\,)+\abs{(\,Tu,u\,)}
+\norm{u}^2\right).
\end{equation}
\end{theorem}

Fix $x_0\notin Y$. By definition of $Y$, 
we can find a vector field  $V\in\cC^\infty(X,\underline{\kg})$
such that 
$\omega_0(V)\neq0$ in an open neighborhood $D$ of $x_0$ with 
$\ol D\cap Y=\emptyset$. Then, 
\begin{equation}\label{e-gue180430aII}
L:=V-\omega_0(V)T\in HX.
\end{equation}
In the rest of this subsection we fix the neighborhood $D$ as above and let 
\begin{equation}\label{e:chi}
\chi, \Td\chi, \chi_1\in\cC^\infty_0(D),\:\: \text{$\Td\chi=1$ near 
$\supp\chi$ and $\chi_1=1$ near $\supp\Td\chi$.}
\end{equation}
Let $u\in\cC^\infty(X)^G$. From \eqref{e-gue180430aII} and since 
that $V(u)=0$ and $\omega_0(V)\neq 0$ on $D$, we have 
\begin{equation}\label{e-gue180430m}
\begin{split}
&Tu=\frac{-1}{\omega_0(V)}L(u)\:\: \text{on $D$},\\
&T(\chi u)=(T\chi)u+\chi Tu=(T\chi)u+\chi \frac{-1}{\omega_0(V)}L(u).
\end{split}
\end{equation}
From \eqref{e-gue180430m}, we deduce that  there exists 
$C>0$ such that for all $u\in\cC^\infty(X)^G$, 
\begin{equation}\label{e-gue180430aIII}
\norm{T(\chi u)}\leq C\left(\vvvert\chi u\vvvert_{1}
+\norm{\chi_1u}\right).
\end{equation} 

For $k\in\N^*$, $U\subset X$ an open set, let $\mathscr D^k(U)$
be the set of differential operators which can be written as 
a linear combination of operators as $W_1\circ\cdots\circ W_j$ 
with $W_1,\ldots,W_j\in\cC^\infty(U,\mathbb CHX)$, $1\leq j\leq k$. 

\begin{lemma}\label{l-gue180501a}
With the notations above, fix 
 $V_1,\ldots, V_s\in\cC^\infty(X,\C HX)$, 
$s\in\N^*$, then there exist 
$V_{1,1}\in\cC^\infty(X,\mathbb CHX)$, $Q_1\in\mathscr D^{s-1}(X)$, 
$Q_2\in\mathscr D^s(X)$, $a, b\in\cC^\infty(X)$, such that
\begin{equation}\label{e-gue180501az}
TV_1=V_1T+V_{1,1}+a(x)T,\  \ \mbox{if $s=1$},
\end{equation}
\begin{equation}\label{e-gue180501a}
TV_1\ldots V_s=V_1\ldots V_sT+Q_1T+Q_2+b(x)T,\ \ 
\mbox{if $s\geq2$}.
\end{equation}
\end{lemma}

\begin{proof}
We first prove \eqref{e-gue180501az}. Note that
\begin{equation}\label{e-gue180430mI}
TV_1=V_1T+[T,V_1].
\end{equation}
We have $[T, V_1]=\Td V_{1,1}+a(x)T$, where 
$\Td V_{1,1}\in\cC^\infty(X,\C HX)$ and $a(x)\in\cC^\infty(X)$. 
From this observation and \eqref{e-gue180430mI}, 
we get \eqref{e-gue180501az}. 

We  now prove \eqref{e-gue180501a}. Let $s=2$. By the argument after
\eqref{e-gue180430mI}, we have 
\begin{equation}\label{e-gue180430mIIz}
\begin{split}
TV_1V_2&=V_1TV_2+[T, V_1]V_2\\
&=V_1TV_2+( \Td V_{1,1}+a(x)T)V_2.
\end{split}
\end{equation}
From \eqref{e-gue180501az}  and \eqref{e-gue180430mIIz},  we get 
\eqref{e-gue180501a} for $s=2$.

Assume that the claim \eqref{e-gue180501a} holds for $s=s_0$ 
for some $s_0\geq2$. We are going to prove that the claim 
\eqref{e-gue180501a} holds for $s=s_0+1$. 
 From the argument after \eqref{e-gue180430mI}, we have 
\begin{equation}\label{e-gue180430mIII}
\begin{split}
TV_1\ldots V_{s_0+1}&=V_1TV_2\ldots V_{s_0+1}+[T, V_1]V_2
\ldots V_{s_0+1}\\
&=V_1TV_2\ldots V_{s_0+1}
+( \Td V_{1,1} +a(x)T )V_2\ldots V_{s_0+1}.
\end{split}
\end{equation}
From \eqref{e-gue180430mIII} and the induction assumption we get 
the claim \eqref{e-gue180501a} for $s=s_0+1$. The lemma  follows. 
\end{proof}

\begin{theorem}\label{t-gue180501}
Fix $s\in\N^*$. Let  $V_1,\ldots, 
V_s\in\cC^\infty(D,\C TX)$. Then there exists $Q\in\mathscr D^s(D)$ 
such that we have for every $u\in\cC^\infty(X)^G$, 
\begin{equation}\label{e-gue180501mp}
V_1\ldots V_su=Qu\quad\text{on $D$}.
\end{equation}
 
\end{theorem}

\begin{proof}
From \eqref{e-gue180430m}, we see that \eqref{e-gue180501mp} 
holds for $s=1$. Assume that \eqref{e-gue180501mp} holds for $s=s_0$.
We are going to prove that \eqref{e-gue180501mp} holds for $s=s_0+1$. 
By induction assumption we only need  to assume that 
$V_1=T$ and $V_j\in\cC^\infty(D,\C HX)$, $j=2,3,\ldots,s_0+1$. 
From \eqref{e-gue180501az}  and \eqref{e-gue180501a},
there exist $\Td V_{1,1}\in\cC^\infty(D,\mathbb CHX)$,
$Q_1\in\mathscr D^{s-1}(D)$, 
$Q_2\in\mathscr D^{s}(D)$ such that 
\begin{equation}\label{e-gue180430bI}
\begin{split}
&TV_2=V_2T+\Td V_{1,1}+a(x)T,\  \ \mbox{if $s_0=1$},\\
TV_2\ldots V_{s_0+1}&=V_2\ldots V_{s_0+1}T+Q_1T+Q_2
+b(x)T,
\quad\quad\mbox{if $s_0\geq 2$}, 
\end{split}
\end{equation}
where   $a, b\in\cC^\infty(D)$. From \eqref{e-gue180430m} 
and \eqref{e-gue180430bI}, we get \eqref{e-gue180501mp}.
\end{proof}

From Theorem~\ref{t-gue180501} and \eqref{e:chi} we deduce: 

\begin{corollary}\label{c-gue180501ab}
Let $s\in\N^*$ and  $V_1,\ldots, 
V_s\in\cC^\infty(D,\C TX)$. 
Then,  there exists $C>0$ such that for all $u\in\cC^\infty(X)^G$ we have
\begin{equation}\label{e-gue180501r}
\begin{split}
&\norm{(V_1\ldots V_s)\chi u}\leq C\left(\vvvert\chi u\vvvert_ s
+\vvvert\chi_1 u\vvvert_{s-1}\right),\\
&\norm{\chi(V_1\ldots V_s)u}\leq C\left(\vvvert\chi u\vvvert_ s
+\vvvert\chi_1 u\vvvert_{s-1}\right).
\end{split}
\end{equation}
\end{corollary}

For $s\in\mathbb Z$, let $\norm{\cdot}_s$ denote the standard
Sobolev norm of order $s$ on $X$. From Corollary~\ref{c-gue180501ab}
and \eqref{e:chi}, we deduce: 

\begin{corollary}\label{c-gue180501}
For every $s\in\N$
there exists $C_s>0$ such that for any $u\in\cC^\infty(X)^G$,
\[\norm{\chi u}_s\leq C_s\norm{\chi_1u}_{s}.\]
\end{corollary}
\begin{theorem}\label{t-gue180501I}
For every $s\in\N$, there exists 
$C_s>0$ such that for any $u\in\cC^\infty(X)^G$ we have
\begin{equation}\label{e-gue180501p}
\vvvert\chi u\vvvert_{s+1}^2\leq 
C_s\left(\vvvert\chi\Box_bu\vvvert_{s}^2+
\vvvert\chi_1u\vvvert_{s}^2\right).
\end{equation}
 \end{theorem}

\begin{proof}
We prove \eqref{e-gue180501p} by induction over $s$. 
From \eqref{e-gue180430aI},  there exists $C>0$ 
such that for any $u\in\cC^\infty(X)^G$ we have 
\begin{equation}\label{e-gue180501pI}
\vvvert\chi u\vvvert_{1}^2\leq C\left((\,\Box_b(\chi u),\chi u\,)+
\abs{(\,T(\chi u),\chi u\,)}+\norm{\chi u}^2\right).
\end{equation}
Now by \eqref{e:chi}, 
\begin{equation}\label{e-gue180501pII}
(\,\Box_b(\chi u),\chi u\,)=(\,\chi\Box_bu,\chi u\,)
+(\,\chi [\Box_b, \chi] u,\chi_1u\,).
\end{equation}
From \eqref{e-gue180430m}, \eqref{e-gue180501pI}, 
\eqref{e-gue180501pII}, and some elementary computation, we get
\eqref{e-gue180501p} for $s=0$. 
We now assume that 
\eqref{e-gue180501p} holds for every $s<k$ and $k\geq1$. 
We will prove that \eqref{e-gue180501p} holds for $s=k$.
Let  $Z_1,\ldots, Z_{k+1}\in\cC^\infty(X,\C HX)$.
By \eqref{e-gue180430aI},  there exist $C_0, C>0$ such that 
for any $u\in\cC^\infty(X)^G$, we have
\begin{multline}\label{e-gue180501g}
\norm{Z_1\ldots Z_ {k+1}(\chi u)}^2\leq 
C_0\vvvert Z_2\ldots Z_ { k+1}(\chi u)\vvvert_{1}^2\\
\leq C\big((\,\Box_b Z_2\ldots Z_ {k+1}(\chi u),
Z_2\ldots Z_ {k+1}(\chi u)\,)\\
\quad+\norm{Z_2\ldots Z_ {k+1}(\chi u)}^2+
|(\,TZ_2\ldots Z_ {k+1}(\chi u),Z_2
\ldots Z_ {k+1}(\chi u)\,)|\big).
\end{multline}
We have 
\begin{equation}\label{e-gue180504}
\begin{split}
&(\,\Box_b(Z_2\ldots Z_ {k+1})\chi u,
(Z_2\ldots Z_ {k+1})\chi u\,)\\
&=(Z_2\ldots Z_ {k+1}\chi\Box_bu+[\Box_b, Z_2\ldots Z_ {k+1}]
\chi u
+Z_2\ldots Z_ {k+1}[\Box_b,\chi]u,
Z_2\ldots Z_ {k+1}\chi u\,)\\
&=(Z_2\ldots Z_ {k+1}\chi\Box_bu,
Z_2\ldots Z_ {k+1}\chi u\,)
+(\,\chi_1  Z_3\ldots Z_ {k+1}[\Box_b,\chi]u,
Z^*_2   Z_2\ldots Z_ {k+1}\chi u\,)\\
&\quad +(\,[\Box_b, Z_2\ldots Z_ {k+1}]
\chi u, Z_2\ldots Z_ {k+1}\chi u\,),
\end{split}
\end{equation}
where $Z^*_2$ denotes the adjoint of $Z_2$ and 
$Z^*_2=-Z_2+\mbox{zero order term}$. 
From \eqref{e-gue180501r} and \eqref{e-gue180504}
we see that there exists $C>0$ such that 
\begin{equation}\label{e-gue180504I}
\big|(\,\Box_b Z_2\ldots Z_ {k+1}
(\chi u),Z_2\ldots Z_ {k+1}(\chi u))\big|
\leq C\left(\vvvert\chi\Box_bu\vvvert^2_{k}+
\frac{1}{\varepsilon}\vvvert\chi_1u\vvvert^2_{k}+
\varepsilon\vvvert\chi u\vvvert^2_{k+1}\right),
\end{equation}
for every $\varepsilon>0$. Similarly,  from \eqref{e-gue180501r}, 
there exists $\widehat C>0$ such that
\begin{equation}\label{e-gue180504II}
\big|(\,TZ_2\ldots Z_ {k+1}(\chi u),
Z_2\ldots Z_ {k+1}(\chi u))\big|
\leq \widehat C\left(\frac{1}{\varepsilon}\vvvert\chi_1u\vvvert^2_{k}+
\varepsilon\vvvert\chi u\vvvert^2_{k+1}\right),
\end{equation}
for every $\varepsilon>0$. From \eqref{e-gue180501g}, 
\eqref{e-gue180504I} and
\eqref{e-gue180504II},
we conclude that \eqref{e-gue180501p} holds for $s=k$.
The theorem follows. 
\end{proof}

From Corollary~\ref{c-gue180501} and Theorem~\ref{t-gue180501I}
we get: 

\begin{theorem}\label{t-gue181223}
For every $s\in\mathbb N$,  there is $C_s>0$ such that
for any $u\in\cC^\infty(X)^G$, 
\begin{equation}\label{e-gue181223}
\norm{\chi u}_{s+1}^2\leq 
C_s\Bigr(\norm{\chi_1\Box_bu}_{s}^2+
\norm{\chi_1u}_{s}^2\Bigr), 
\end{equation}
where $\chi, \chi_1\in\cC^\infty_0(D)$ are as in \eqref{e:chi}. 
\end{theorem}

\subsection{Closed range property for the $G$-invariant Kohn Laplacian}
\label{s-gue181012}
In this section we will work in the setting of
Assumption \ref{a-gue170410I} (i).
%
Under these hypotheses we prove subelliptic estimates, regularity and
the closed range in $L^2$ property for the $G$-invariant Kohn Laplacian.

Let $\ol{\pr}^*_b: {\rm Dom\,}\ol{\pr}^*_b\subset L^2_{(0,q+1)}(X)
\To L^2_{(0,q)}(X)$ 
be the Hilbert space adjoint of $\ddbar_b$ in \eqref{eq:2.14} 
with respect to $(\,\cdot\,,\,\cdot\,)$ in \eqref{e:l2}. 
The operators $\ddbar_b$, $\ol{\pr}^*_b$ commute with 
$G$-action, thus we can define
$\ddbar_{b,G}$ and $\ol{\pr}^*_{b,G}$  with 
\begin{align}\label{eq:ma3.23}
{\rm Dom\,}\ddbar_{b,G}
:={\rm Dom\,}\ddbar_b\cap L^2_{(0,q)}(X)^G,
\quad {\rm Dom\,}\ol{\pr}^*_{b,G}
:={\rm Dom\,}\ol{\pr}^*_b\cap L^2_{(0,q+1)}(X)^G,
\end{align}
and
$\ol{\pr}^*_{b,G}: {\rm Dom\,}\ol{\pr}^*_{b,G}\To L^2_{(0,q)}(X)^G$
is the Hilbert space adjoint of $\ddbar_{b,G}$. 
Let $\Box^{(q)}_{b,G}$ denote the (Gaffney extension) of 
the $G$-invariant Kohn Laplacian given by
\begin{equation}\label{e-gue181012}
\begin{split}
\Dom \Box^{(q)}_{b,G}=&\Big\{u\in L^2_{(0,q)}(X)^G;\, 
u\in\Dom  \ddbar_{b,G}\cap\Dom  \ol{\pr}^*_{b,G}, \\
& \hspace{3cm}\ddbar_{b,G}u\in\Dom  \ol{\pr}^{*}_{b,G}, \, 
\ol{\pr}^*_{b,G}u\in\Dom  \ddbar_{b,G}\Big\}\,,\\
\Box^{(q)}_{b,G}u=&\ddbar_{b,G}\ol{\pr}^*_{b,G}u
+\ol{\pr}^{*}_{b,G}\ddbar_{b,G}u
\:\:\text{for $u\in \Dom  \Box^{(q)}_{b,G}$}\,.
 \end{split}
\end{equation}

\begin{lemma}\label{l-gue200221yyd}
Let $u\in\Dom\ddbar_{b,G}\cap L^2_{(0,q)}(X)^G$. 
Then $\ddbar_{b,G}u\in\Dom  \ddbar_{b,G}$ and 
$\ddbar^2_{b,G}u=0$. 
\end{lemma}

\begin{proof}
By Friedrichs' lemma~\cite[Appendix D]{CS01}, there is a sequence 
$\set{u_j}^{\infty }_{j=1}\subset\Omega^{0,q}(X)^G$ such that 
$u_j\To u$ in $L^2_{(0,q)}(X)^G$ as $j\To\infty $ and 
$\ddbar_{b,G}u_j\To\ddbar_{b,G}u$ in $L^2_{(0,q+1)}(X)^G$ 
as $j\To\infty $. 
Let $v\in\Omega^{0,q+2}(X)^G$. We have 
\begin{align}\label{eq:h3.26}
(\,\ddbar_{b,G}u\,,\,\ol{\pr}^*_{b,G}v\,)
=\lim_{j\To\infty }(\,\ddbar_{b,G}u_j\,,\,\ol{\pr}^*_{b,G}v\,)
=\lim_{j\To\infty }(\,\ddbar^2_{b,G}u_j\,,\,v\,)=0.
\end{align}
Hence, $\ddbar_{b,G}u\in\Dom  \ddbar_{b,G}$ and 
$\ddbar^2_{b,G}u=0$. 
\end{proof}

\begin{lemma}\label{l-gue200214yyd}
The operator 
$\Box^{(q)}_{b,G}: \Dom  \Box^{(q)}_{b,G}\subset 
L^2_{(0,q)}(X)^G\To L^2_{(0,q)}(X)^G$
is a closed. 
\end{lemma}

\begin{proof}
Let $\{(f_k, \Box^{(q)}_{b,G}f_k)\in L^2_{(0,q)}(X)^G\times 
L^2_{(0,q)}(X)^G;\, f_k\in\Dom  \Box^{(q)}_{b,G}\}_{k=1}^\infty$
with 
\begin{align}\label{e-gue201013yyd}
\lim_{k\To\infty }f_k=f, \quad
\lim_{k\To\infty }\Box^{(q)}_{b,G}f_k=h \text{  in } L^2_{(0,q)}(X)^G.
\end{align}
By definition, to check that $\Box^{(q)}_{b,G}$ is a closed operator, 
we need to show that $f\in\Dom  \Box^{(q)}_{b,G}$ and 
$\Box^{(q)}_{b,G}f=h$. 
Since $f_k\in\Dom  \Box^{(q)}_{b,G}$, for each $k$, we have 
\begin{equation}\label{e-gue201013yydI}
\norm{\ol{\pr}^*_{b,G}(f_j-f_k)}^2
\leq  \norm{\ol{\pr}^*_{b,G}(f_j-f_k)}^2
+\norm{\ddbar_{b,G}(f_j-f_k)}^2
=(\,\Box^{(q)}_{b,G}(f_j-f_k)\,,\,f_j-f_k\,).
\end{equation}
From \eqref{e-gue201013yyd} and \eqref{e-gue201013yydI}, 
$\{\ol{\pr}^*_{b,G}f_k\}^{\infty}_{k=1}$ is a Cauchy sequence 
in $L^2$,  hence $\lim_{k\To\infty }\ol{\pr}^*_{b,G}f_k=h_{1}$ in
$L^2_{(0,q-1)}(X)^G$, for some $h_{1}\in L^2_{(0,q-1)}(X)^G$. 

Let $v\in\Dom  \ddbar_{b,G}\cap L^2_{(0,q-1)}(X)^G$. We have 
\begin{equation}\label{e-gue200214ycd}
(\,f\,,\,\ddbar_{b,G}v\,)=\lim_{k\To\infty }(\,f_k\,,\,
\ddbar_{b,G}v\,)=\lim_{k\To\infty }(\,\ol{\pr}^*_{b,G}f_k\,,\,v\,)
=(\,u\,,\,v\,).
\end{equation}
Hence, $f\in\Dom  \ol{\pr}^*_{b,G}$ and 
$\ol{\pr}^*_{b,G}f=\lim_{k\To\infty }\ol{\pr}^*_{b,G}f_k$. 
Similarly, we can repeat the procedure above and 
show that $f\in\Dom  \ddbar_{b,G}$ and 
$\ddbar_{b,G}f=\lim_{k\To\infty }\ddbar_{b,G}f_k$. 

Now, we show that $\ddbar_{b,G}f\in\Dom  \ol{\pr}^*_{b,G}$. 
From Lemma~\ref{l-gue200221yyd}, we know that 
\begin{equation}\label{e-gue200218yyd}
\mbox{$(\,\ol{\pr}^*_{b,G}\,\ddbar_{b,G}u\,,\,
\ddbar_{b,G}\,\ol{\pr}^*_{b,G}v\,)=0$, for every 
$u, v\in\Dom  \Box^{(q)}_{b,G}$.}
\end{equation}
From \eqref{e-gue200218yyd}, we have 
\begin{equation}\label{e-gue201014yyd}
\norm{\ol{\pr}^*_{b,G}\,
\ddbar_{b,G}(f_j-f_k)}^2+\norm{\ddbar_{b,G}\,
\ol{\pr}^*_{b,G}(f_j-f_k)}^2
=\norm{\Box^{(q)}_{b,G}(f_j-f_k)}^2.
\end{equation}
From \eqref{e-gue201013yyd}, \eqref{e-gue201014yyd}, 
$\set{\ol{\pr}^*_{b,G}\ddbar_{b,G}f_j}^{+\infty}_{j=1}$
is a Cauchy sequence in $L^2$, hence 
$\lim_{k\To\infty }\ol{\pr}^*_{b,G}\ddbar_{b,G}f_k=h_{2}$ in 
$L^2_{(0,q)}(X)^G$, for some $h_{2}\in L^2_{(0,q)}(X)^G$ and thus, 
$\ddbar_{b,G}f\in\Dom  \ol{\pr}^*_{b,G}$ and 
$\ol{\pr}^*_{b,G}\,\ddbar_{b,G}f
=\lim_{k\To\infty }\ol{\pr}^*_{b,G}\,\ddbar_{b,G}f_k$. 
Similarly, we can repeat the process above and show that 
\begin{align}\label{eq:h3.33}
\ol{\pr}^*_{b,G}f\in\Dom  \ddbar_{b,G},\ \ \ddbar_{b,G}\,
\ol{\pr}^*_{b,G}f=\lim_{k\To\infty }
\ddbar_{b,G}\,\ol{\pr}^*_{b,G}f_k.
\end{align}
Hence, $f\in\Dom  \Box^{(q)}_{b,G}$ and 
$\Box^{(q)}_{b,G}f=\lim_{k\To\infty }\Box^{(q)}_{b,G}f_k=h$. 
The lemma follows. 
\end{proof}

Since the dimension of $X$ is greater or equal to five, 
we can repeat Kohn's method \cite{Koh64}
(see also the proof of~\cite[Theorem 8.3.5]{CS01}) 
and deduce the following subelliptic estimates: 

\begin{theorem}\label{t-gue181012}
Under Assumption \ref{a-gue170410I} (i), let 
$\eta, \eta_1\in\cC^\infty(X)$ 
such that $\eta=1$ near $Y$, $\eta_1=1$ near 
$\supp\eta$ and the Levi form is positive near 
$\supp\eta_1$.
Then for every $s\in\N$ 
 there is $C_s>0$ such that for any $u\in\Omega^{0,1}(X)$, 
\begin{equation}\label{e-gue181012m}
\norm{\eta u}^2_{s+1}\leq C_s\Bigr(\norm{\eta_1\Box^{(1)}_bu}^2_s
+\norm{\eta_1u}^2_s\Bigr).
\end{equation}
\end{theorem}

Repeating the proof of Theorem~\ref{t-gue181223} with minor changes
we get: 

\begin{theorem}\label{t-gue181012I}
Let $\gamma, \gamma_1\in\cC^\infty(X)$ with $\gamma_1=1$ near 
$\supp\gamma$ and $\supp\gamma_1\cap Y=\emptyset$.
For every $s\in\N$, there exists 
$C_s>0$ such that for any $u\in\Omega^{0,1}(X)^G$, 
\begin{equation}\label{e-gue181012mI}
\norm{\gamma u}_{s+1}^2
\leq C_s\Bigr(\norm{\gamma_1\Box^{(1)}_{b,G}u}_{s}^2+
\norm{\gamma_1u}_{s}^2\Bigr).
\end{equation}
\end{theorem}

From Theorems ~\ref{t-gue181012}, ~\ref{t-gue181012I} 
and by using a partition of unity we obtain: 

\begin{theorem}\label{t-gue181012II}
Assume that Assumption \ref{a-gue170410I} (i) holds. 
Then for every $s\in\N$
and every $\gamma, \gamma_1\in\cC^\infty(X)$ with
$\gamma_1=1$ near 
$\supp\gamma$,  there is $C_s>0$ such that for any 
$u\in\Omega^{0,1}(X)^G$,
\begin{equation}\label{e-gue181012mzq}
\norm{\gamma u}^2_{s+1}
\leq C_s\Bigr(\norm{\gamma_1\Box^{(1)}_{b,G}u}^2_s
+\norm{\gamma_1u}^2_s\Bigr).
\end{equation}
\end{theorem}

For every $s\in\mathbb Z$, let $H^s_{(0,q)}(X)^G$ be 
the completion of $\Omega^{0,q}(X)^G$ with respect to 
$\norm{\cdot}_s$. From Theorem~\ref{t-gue181012II}, 
we can repeat the technique of elliptic regularization 
(see the proof of Theorem 8.4.2 in~\cite{CS01}) and conclude: 

\begin{theorem}\label{t-gue181013}
Under Assumption \ref{a-gue170410I} (i) let 
$u\in\Dom  \Box^{(1)}_{b,G}$ and 
let $\mathcal{U}$ be an open set of $X$.  
Let $\Box^{(1)}_{b,G}u=v\in L^2_{(0,1)}(X)^G$. 
If $v|_{\mathcal{U}}$ is smooth, then $u|_{\mathcal{U}}$ is smooth. 

Let $\tau, \tau_1\in\cC^\infty(X)^G$ with $\tau_1=1$ 
near $\supp\tau$.  If $\tau_1v\in H^{s}_{(0,1)}(X)^G$, 
for some $s\in\N$, then 
$\tau u\in H^{s+1}_{(0,1)}(X)^G$ and  there is 
$C_{s}>0$ independent of $u$, $v$, such that 
\begin{equation}\label{e-gue181016mp}
\norm{\tau u}_{s+1}\leq C_{s}
\Bigr(\norm{\tau_1\Box^{(1)}_{b,G}u}_{s}
+\norm{\tau_1u}_{s}\Bigr).
\end{equation}
\end{theorem}

From Theorem~\ref{t-gue181013} and a standard argument 
in functional analysis, we get: 

\begin{theorem}\label{t-gue181013I}
If Assumption \ref{a-gue170410I} (i) holds, then the operator 
$\Box^{(1)}_{b,G}:\Dom  \Box^{(1)}_{b,G}
\subset L^2_{(0,1)}(X)^G\To L^2_{(0,1)}(X)^G$
has closed range.
\end{theorem}

Let $N^{(1)}_G: L^2_{(0,1)}(X)^G\To\Dom\Box^{(1)}_{b,G}$ 
be the partial inverse of $\Box^{(1)}_{b,G}$ and let 
 $S^{(1)}_G: L^2_{(0,1)}(X)^G\To\Ker\Box^{(1)}_{b,G}$ 
 be the Szeg\H{o} projection, i.e., the orthogonal projection onto 
 $\Ker\Box^{(1)}_{b,G}$ with respect to $(\cdot\,,\,\cdot\,)$. 
 We have 
 \begin{equation}\label{e-gue181013mp}
 \begin{split}
 &\Box^{(1)}_{b,G}N^{(1)}_G+S^{(1)}_G=I\ \ \mbox{on 
 $L^2_{(0,1)}(X)^G$},\\
 &N^{(1)}_G\Box^{(1)}_{b,G}+S^{(1)}_G=I\ \
 \text{on $\Dom  \Box^{(1)}_{b,G}$}.
\end{split}
\end{equation}
From Theorem~\ref{t-gue181013} we deduce: 
\begin{theorem}\label{t-gue181013mp}
Assume that Assumption \ref{a-gue170410I} (i) holds. 
Then for every $s\in\Z$ we can extend $N^{(1)}_G$ to 
a continuous operator  
$N^{(1)}_G: H^s_{(0,1)}(X)^G\To H^{s+1}_{(0,1)}(X)$.
Moreover, $\Ker\Box^{(1)}_{b,G}$ is a finite dimensional 
subspace of $\Omega^{0,1}(X)^G$. 
\end{theorem}
Let 
\begin{equation}\label{eq:ma3.37}
\mP^{(q)}_G: L^2_{(0,q)}(X)\To L^2_{(0,q)}(X)^G,
\end{equation} 
be the orthogonal projection with respect to $(\,\cdot\,,\,\cdot\,)$.
It is not difficult to see that for every $s\in\mathbb Z$
we can extend  $\mP^{(q)}_G$ to  $H^s_{(0,q)}(X)$
and $\mP^{(q)}_G: H^s_{(0,q)}(X)\To H^{s}_{(0,q)}(X)^G$ 
is continuous. 
We extend $N^{(1)}_G$ and 
$S^{(1)}_G$ to $H^s_{(0,1)}(X)$ by  
\begin{equation}\label{eq:3.32}
N^{(1)}_Gu:=N^{(1)}_G\mP^{(1)}_Gu,\ \ 
S^{(1)}_Gu:=S^{(1)}_G\mP^{(1)}_Gu,\ \ 
\mbox{for $u\in H^s_{(0,1)}(X)$},\ \ 
s\in\mathbb Z.
\end{equation}
From Theorem~\ref{t-gue181013mp} we see that 
$S^{(1)}_G$ is smoothing and
\begin{equation}\label{e-gue181013ab}
\mbox{$N^{(1)}_G: H^s_{(0,1)}(X)\To H^{s+1}_{(0,1)}(X)^G$
is continuous for every $s\in\mathbb Z$}.
\end{equation}
 
\begin{theorem}\label{t-gue181016}
Under Assumption \ref{a-gue170410I} (i), let 
$\tau, \tau_1\in\cC^\infty(X)^G$
with $\supp\tau\cap\supp\tau_1=\emptyset$. 
Then $\tau N^{(1)}_G\tau_1$ is smoothing.
\end{theorem}

\begin{proof}
Let $\Td\tau, \hat\tau\in\cC^\infty(X)^G$ with $\Td\tau=1$ near 
$\supp\hat\tau$, $\hat\tau=1$ near $\supp\tau$ 
and $\supp\hat\tau\cap\supp\tau_1=\emptyset$. 
Let $v\in L^2_{(0,1)}(X)^G$ and put 
$\Td\tau N^{(1)}_G\tau_1v=u\in H^1_{(0,1)}(X)^G$.
From \eqref{e-gue181013mp}, we have 
\begin{equation}\label{e-gue181016mpI}
\begin{split}
\Box^{(1)}_{b,G}u&=\Box^{(1)}_{b,G}\Td\tau N^{(1)}_G\tau_1v\\
&=\Td\tau\,\Box^{(1)}_{b,G} N^{(1)}_G\tau_1v+\big[\Box^{(1)}_{b,G}, 
\Td\tau\big]N^{(1)}_G\tau_1v\\
&=\Td\tau(I-S^{(1)}_G)\tau_1v
+\big[\Box^{(1)}_{b,G}, \Td\tau\big]N^{(1)}_G\tau_1v\\
&=-\Td\tau S^{(1)}_G\tau_1v
+\big[\Box^{(1)}_{b,G}, \Td\tau\big]N^{(1)}_G\tau_1v.
\end{split}
\end{equation}
Since $S^{(1)}_G$ is smoothing, 
$-\Td\tau S^{(1)}_G\tau_1v\in\cC^\infty(X)^G$. 
Since $\Td\tau=1$ near $\supp\hat\tau$, 
$\hat\tau[\Box^{(1)}_{b,G}, \Td\tau]N^{(1)}_G\tau_1v=0$. 
From this observation, we deduce that 
\begin{equation}\label{e-gue181016mpII}
\hat\tau\Box^{(1)}_{b,G}u\in\cC^\infty(X)^G. 
\end{equation}
Fix $s\in\N$.   
From Theorem~\ref{t-gue181013mp}, \eqref{e-gue181016mp} and 
\eqref{e-gue181016mpI}, there exist $C_s, \Td C_s>0$ such that 
for any $v\in L^2_{(0,1)}(X)^G$, we have 
\begin{equation}\label{e-gue181016mpIII}
\begin{split}
\norm{\tau N^{(1)}_G\tau_1v}_{s+1}
&=\norm{\tau u}_{s+1}
\leq C_s\Bigr(\norm{\hat\tau\Box^{(1)}_{b,G}u}_s
+\norm{\hat\tau u}_s\Bigr)\\
&\leq\Td C_s\Bigr(\norm{-\hat\tau S^{(1)}_G\tau v}_s
+\norm{\hat\tau N^{(1)}_G\tau_1v}_s\Bigr).
\end{split}
\end{equation}
Take $s=1$ in \eqref{e-gue181016mpIII} , 
from Theorem~\ref{t-gue181013mp} and note that $S^{(1)}_G$
is smoothing, we conclude that 
$\norm{\tau N^{(1)}_G\tau_1v}_{2}\leq C\norm{v}$. We have proved
that for any  $\gamma, \gamma_1\in\cC^\infty(X)^G$
with $\supp\gamma\cap\supp\gamma_1=\emptyset$, 
we have
\begin{equation}\label{e-gue181016a}
\mbox{$\gamma N^{(1)}_G\gamma_1: L^2_{(0,1)}(X)^G
\To H^{2}_{(0,1)}(X)^G$ is continuous}.
\end{equation}
Take $s=2$ in \eqref{e-gue181016mpIII} , from 
Theorem~\ref{t-gue181013mp}, \eqref{e-gue181016a}
and note that $S^{(1)}_G$ is smoothing, we conclude that
 there exists $\hat C>0$ such that for any 
$v\in L^2_{(0,1)}(X)^G$,
$\norm{\tau N^{(1)}_G\tau_1v}_{3}\leq\hat C\norm{v}$.
Continuing in this way, we conclude that 
for any  $\gamma, \gamma_1\in\cC^\infty(X)^G$ with 
$\supp\gamma\cap\supp\gamma_1=\emptyset$, 
we have
\begin{equation}\label{e-gue181016aII}
\mbox{$\gamma N^{(1)}_G\gamma_1: 
L^2_{(0,1)}(X)^G\To H^{s}_{(0,1)}(X)^G$ is continuous, for every 
$s\in\mathbb N^{*}•$}.
\end{equation}
By taking adjoint in \eqref{e-gue181016aII}, we deduce that for any  
$\gamma, \gamma_1\in\cC^\infty(X)^G$ with 
$\supp\gamma\cap\supp\gamma_1=\emptyset$, 
we have
\begin{equation}\label{e-gue181016aIII}
\mbox{$\gamma N^{(1)}_G\gamma_1: 
H^{-s}_{(0,1)}(X)^G\To L^2_{(0,1)}(X)^G$ is continuous, 
for every $s\in\mathbb N^{*}$}.
\end{equation}

Now, let $v\in H^{-s_0}_{(0,1)}(X)^G$, $s_0\in\mathbb N^*$, and 
put $\Td\tau N^{(1)}_G\tau_1v=u\in H^{-s_0+1}_{(0,1)}(X)^G$. 
Let $v_j\in\Omega^{0,1}(X)^G$, $j=1,2,\ldots$\,, 
with $v_j\To v$ in $H^{-s_0}_{(0,1)}(X)^G$ as $j\To\infty $. 
Taking $s=0$ in \eqref{e-gue181016mpIII}, we deduce from
Theorem~\ref{t-gue181013mp}, \eqref{e-gue181016aIII} 
and the fact that $S^{(1)}_G$ is smoothing that 
 there exists $C>0$ such that for any $h\in L^2_{(0,1)}(X)^G$, 
\begin{equation}\label{e-gue200213yyd}
\norm{\tau N^{(1)}_G\tau_1h}_{1}\leq C\norm{h}_{-s_0}.
\end{equation}
From \eqref{e-gue200213yyd}, we have 
\begin{equation}\label{e-gue200213yydI}
\lim_{j,k\To\infty }\norm{\tau N^{(1)}_G\tau_1(v_j-v_k)}_{1}
\leq C\norm{v_j-v_k}_{-s_0}=0.
\end{equation}
From \eqref{e-gue200213yydI}, we see that 
$\{\tau N^{(1)}_G\tau_1v_j\}_{j\geq1}$ is a Cauchy sequence 
in $H^{1}_{(0,1)}(X)^G$. 
Since $\tau N^{(1)}_G\tau_1v_j\To\tau N^{(1)}_G\tau_1v$ 
in $H^{-s_0+1}_{(0,1)}(X)^G$ as $j\To\infty $ 
(see Theorem~\ref{t-gue181013mp}), 
we deduce that 
$\tau N^{(1)}_G\tau_1v\in H^1_{(0,1)}(X)^G$ and 
$\norm{\tau N^{(1)}_G\tau_1v}_{1}\leq C\norm{v}_{-s_0}$.
We have proved that for any  $\gamma, \gamma_1\in\cC^\infty(X)^G$ 
with 
\[\supp\gamma\cap\supp\gamma_1=\emptyset,\] 
we have
\begin{equation}\label{e-gue181016y}
\mbox{$\gamma N^{(1)}_G\gamma_1: 
H^{-s_0}_{(0,1)}(X)^G\To H^{1}_{(0,1)}(X)^G$ is continuous}.
\end{equation}
Again, take $s=1$ in \eqref{e-gue181016mpIII} , from 
Theorem~\ref{t-gue181013mp}, \eqref{e-gue181016y} 
and note that $S^{(1)}_G$ is smoothing,  we conclude that 
 there exists $C>0$ such that for any $h\in L^2_{(0,1)}(X)^G$, 
\begin{equation}\label{e-gue200213yydII}
\norm{\tau N^{(1)}_G\tau_1h}_{2}\leq C\norm{h}_{-s_0}.
\end{equation}
From \eqref{e-gue200213yydII}, we can repeat the argument above 
and deduce that 
$\tau N^{(1)}_G\tau_1v\in H^2_{(0,1)}(X)^G$ and 
$\norm{\tau N^{(1)}_G\tau_1v}_{2}\leq C\norm{v}_{-s_0}$.
Continuing in this way, we conclude that 
for any  $\gamma, \gamma_1\in\cC^\infty(X)^G$ with 
$\supp\gamma\cap\supp\gamma_1=\emptyset$,
we have
\begin{equation}\label{e-gue181016yI}
\mbox{$\gamma N^{(1)}_G\gamma_1: 
H^{-s_0}_{(0,1)}(X)^G\To H^{s}_{(0,1)}(X)^G$ is continuous, 
for every $s\in\mathbb N$}.
\end{equation}
The theorem follows. 
\end{proof}

We return to the case of functions. As before, let 
$S_G: L^2(X)\To\Ker\Box^{(0)}_{b,G}=H^0_b(X)^G$ 
be the $G$-invariant Szeg\H{o} projection. 
In view of Theorem~\ref{t-gue181013I}, we can repeat the proof 
of \cite[Proposition 6.15, p.\,56]{Hsiao08} and deduce the following. 

\begin{theorem}\label{t-gue181017}
Under Assumption \ref{a-gue170410I} (i) the operator $\Box^{(0)}_{b,G}: 
\Dom  \Box^{(0)}_{b,G}\subset L^2(X)^G\To L^2(X)^G$ 
has closed range and 
\begin{equation}\label{e-gue181017}
\mbox{$S_G=\mP^{(0)}_G
-\ol{\pr}^*_{b,G}N^{(1)}_G\ddbar_{b,G}\mP^{(0)}_G$ on $L^2(X)$.}
\end{equation}
\end{theorem}

\begin{theorem}\label{t-gue181017I}
Let $\tau, \tau_1\in\cC^\infty(X)^G$ with 
$\supp\tau\cap\supp\tau_1=\emptyset$. 
If Assumption \ref{a-gue170410I} (i) holds,
then $\tau S_G\tau_1$ is smoothing. 
\end{theorem}

\begin{proof}
By \eqref{e-gue181017} we have 
\begin{equation}\label{e-gue181017I}
\tau S_G\tau_1=-\tau\ol{\pr}^*_{b,G}N^{(1)}_G\ddbar_{b,G}\tau_1
=-\tau\ol{\pr}^*_{b,G}\Td\tau N^{(1)}_G\Td\tau_1\ddbar_{b,G}\tau_1,
\end{equation}
where $\Td\tau, \Td\tau_1\in\cC^\infty(X)^G$ with $\Td\tau=1$ 
near $\supp\tau$, $\Td\tau_1=1$ near $\supp\tau_1$, 
$\supp\Td\tau\cap\supp\Td\tau_1=\emptyset$. 
In view of Theorem~\ref{t-gue181016}, we see that 
$\Td\tau N^{(1)}_G\Td\tau_1$ is smoothing. From this observation 
and \eqref{e-gue181017I}, the theorem follows. 
\end{proof}

Let $d\mu=d\mu(h)$ be the Haar measure on $G$ with 
$\int_Gd\mu(h)=1$. We also need 
\begin{theorem}\label{t-gue181021a}
Under Assumption \ref{a-gue170410I} (i) there exists $c_0>0$
such that for any $\lambda\in(0,c_0)$ such that 
\[S_G(x,y)=\int_GS_{\leq\lambda}(x,h.y)d\mu(h)\quad 
\text{on $X\times X$},\]
where $S_{\leq\lambda}$ is the spectral projection given 
by \eqref{e-suXI-I}. 
\end{theorem}
\begin{proof}
Since $\Box^{(0)}_{b,G}: \Dom\Box^{(0)}_{b,G}\subset L^2(X)^G
\To L^2(X)^G$ has closed range, we can define the partial inverse
\[N^{(0)}_G: L^2(X)^G\To\Dom  \Box^{(0)}_{b,G}\]
of $\Box^{(0)}_{b,G}$ as follows: Let $u\in L^2(X)^G$. 
Since $\Ran \Box^{(0)}_{b,G}$ is closed in $L^2(X)^G$,
we have the orthogonal decomposition:
\[u=v_0+v_1,\ \ v_0\in\Ran\Box^{(0)}_{b,G},\ \ 
v_1\perp\Ran \Box^{(0)}_{b,G}\,.\]
There is a unique $\beta\in\Dom  \Box^{(0)}_{b,G}$, 
$\beta\perp\Ker\Box^{(0)}_{b,G}$, such that 
$\Box^{(0)}_{b,G}\beta=v_0$. 
Define $N^{(0)}_Gu:=\beta$. Then, 
$N^{(0)}_G: L^2(X)^G\To\Dom  \Box^{(0)}_{b,G}$
is a linear operator. We claim that 
\begin{equation}\label{e-gue200221yyda}
\mbox{$N^{(0)}_G: L^2(X)^G\To\Dom  \Box^{(0)}_{b,G}$ 
is a closed operator.}
\end{equation}
Let $\big\{(f_k, N^{(0)}_Gf_k)\in L^2(X)^G\times L^2(X)^G;\, 
k\in\N\big\}$, with 
$\lim_{k\To\infty }f_k=f$ in $L^2(X)^G$ and 
$\lim_{k\To\infty }N^{(0)}_Gf_k=h$ in $L^2(X)^G$. 
For every $k\in\N$, write 
\begin{equation}\label{e-gue200222yyd}
f_k=\Box^{(0)}_{b,G}g_k+\xi_k,\ \ g_k\in\Dom  \Box^{(0)}_{b,G},
\ \ g_k\perp\Ker\Box^{(0)}_{b,G},\ \ 
\xi_k\perp\Ran \Box^{(0)}_{b,G}.
\end{equation}
From \eqref{e-gue200222yyd} and the definition of $N^{(0)}_G$, 
we see that 
\begin{equation}\label{e-gue200222yydII}
N^{(0)}_Gf_k=g_k, \quad\lim_{k\To\infty }g_k=h
\text{  in }   L^2(X)^G.
\end{equation}
Since $\set{f_k}^{\infty }_{k=1}$ is a Cauchy sequence, 
$\{\Box^{(0)}_{b,G}g_k\}^{\infty }_{k=1}$ and 
$\set{\xi_k}^{\infty }_{k=1}$ are 
Cauchy sequences. Hence, 
\begin{equation}\label{e-gue200222yydI}
\lim_{k\To\infty }\Box^{(0)}_{b,G}g_k=\eta, \quad
\lim_{k\To\infty }\xi_k=\xi  \text{  in }  L^2(X)^G,
\end{equation}
for some $\eta, \xi\in L^2(X)^G$, 
$\xi\perp\Ran\,\Box^{(0)}_{b,G}$, and we have the orthogonal
decomposition
\begin{equation}\label{e-gue200218yyda}
f=\eta+\xi.
\end{equation}
From \eqref{e-gue200222yydII} and \eqref{e-gue200222yydI}, we see 
that  $(g_k,\Box^{(0)}_{b,G}g_k)\To(h, \eta)$ in
$L^2(X)^G\times L^2(X)^G$ as $k\To\infty $. Since 
$\Box^{(0)}_{b,G}$ is a closed operator, we conclude that 
$h\in\Dom  \Box^{(0)}_{b,G}$, 
$h\perp\Ker\Box^{(0)}_{b,G}$, and 
$\Box^{(0)}_{b,G}h=\eta$. 
From this observation and \eqref{e-gue200218yyda}, we get the 
orthogonal decomposition
\[f=\Box^{(0)}_{b,G}h+\xi\]
and hence $N^{(0)}_Gf=h$. The claim \eqref{e-gue200221yyda} 
follows. 

Since $N^{(0)}_G$ is a closed operator defined on the Banach space 
$L^2(X)^G$, by the closed graph theorem, $N^{(0)}_G$ is a continuous
operator. 
Hence, there is $c_0>0$ such that 
\begin{equation}\label{e-gue200218yydb}
\|N^{(0)}_G\beta\|\leq\frac{1}{c_0}\|\beta\|,\ \ 
\mbox{for every $\beta\in L^2(X)^G$}. 
\end{equation}
Let $\beta=\Box^{(0)}_{b,G}u$ in \eqref{e-gue200218yydb}, 
where $u\in\Dom  \Box^{(0)}_{b,G}$, 
$u\perp\Ker\Box^{(0)}_{b,G}$, 
we get 
\begin{equation}\label{e-gue181021my}
\|\Box^{(0)}_{b,G}u\|\geq c_0\norm{u},\ \ \mbox{for every 
$u\in\Dom  \Box^{(0)}_{b,G}$, $u\perp\Ker
\Box^{(0)}_{b,G}$}. 
\end{equation}
Fix $0<\lambda<c_0$. We claim that with the notation 
\eqref{eq:ma3.37}, we have
\begin{equation}\label{e-gue181021myI}
S_G=S_{\leq\lambda}\circ\mP^{(0)}_G\ \ \mbox{on $L^2(X)$}. 
\end{equation}
If the claim is not true, we can find a 
$u\in H^0_{b,\leq\lambda}(X)\cap L^2(X)^G$ with 
$u\perp\Ker\Box^{(0)}_{b,G}$, $\norm{u}=1$, 
where $H^0_{b,\leq\lambda}(X)$ is given by 
\eqref{e-suX}. Since $u\in H^0_{b,\leq\lambda}(X)$, 
we have $\norm{\Box^{(0)}_{b,G}u}
\leq\lambda\norm{u}<c_0\norm{u}$. From this observation and 
\eqref{e-gue181021my} we get a contradiction. 
The claim \eqref{e-gue181021myI} follows and this yields
the theorem. 
\end{proof}

\subsection{$G$-invariant Szeg\H{o} kernel asymptotics away $Y$}
\label{s-gue180505}

Let 
$S_G(x,y)\in \cD'(X\times X)$ be the distribution kernel of $S_G$. 
For any subset $A$ of $X$, put $GA:=\set{g\circ x;\, x\in A, g\in G}$. 
\begin{theorem}\label{t-gue181021}
Under Assumption \ref{a-gue170410I}, let 
$\tau\in\cC^\infty(X)$  with $\supp\tau\cap Y=\emptyset$.
Then $\tau S_G$ and $S_G\tau$ 
are smoothing operators. 
\end{theorem}
\begin{proof} First we work under Assumption \ref{a-gue170410I} (i) as
in the previous Section.
Let $v\in L^2(X)$. Take $v_j\in\cC^\infty(X)$, $j\in\N^{*}$, 
such that $\norm{v_j-v}\To0$ as $j\To\infty $. 
We have $\norm{S_Gv_j-S_Gv}\To0$ as $j\To\infty $. 
By  \eqref{e-gue181013ab} and \eqref{e-gue181017}, 
we see that 
\begin{equation}\label{eq:ma3.60}
	S_Gv_j\in\cC^\infty(X), \:\:\text{for every $j\in\N^{*}$}. 
\end{equation}
For every $j$, put $u_j:=S_Gv_j$. 
By Corollary~\ref{c-gue180501} and \eqref{e-gue180501p}, 
it is straightforward to see that for every $s\in\mathbb N$ 
there exists $C_s>0$ such that for $\chi$ in \eqref{e:chi}, 
\begin{equation}\label{e-gue180504b}
\norm{\chi u_j}_s\leq C_s,\ \ \mbox{for every $j\in\N^{*}$}.
\end{equation}
From \eqref{e-gue180504b} we deduce that $\chi S_Gv\in H^s(X)$ 
for every $s\in\mathbb N$ and 
\begin{equation}\label{e-gue180504bI}
\mbox{$\chi S_G: L^2(X)\To H^s(X)$ is continuous, for every 
$s\in\mathbb N$.}
\end{equation}
By using a partition of unity we conclude that for any 
$\tau, \tau_1\in\cC^\infty(X)^G$ with 
$\supp\tau\cap Y=\emptyset$, $\supp\tau_1\cap Y=\emptyset$, 
we have 
\begin{equation}\label{e-gue181019}
\mbox{$\tau S_G: L^2(X)\To H^s(X)$ is continuous, 
for every $s\in\mathbb N$}
\end{equation}
and hence
\begin{equation}\label{e-gue181019I}
\mbox{$S_G\tau_1: H^{-s}(X)\To L^2(X)$ is continuous, for every
$s\in\mathbb N$.}
\end{equation}
From \eqref{e-gue181019} and \eqref{e-gue181019I}, we conclude that 
\begin{equation}\label{e-gue181019II}
\mbox{$\tau S_G\tau_1=(\tau S_G)\circ(S_G\tau_1): 
H^{-s}(X)\To H^s(X)$ is continuous, for every $s\in\mathbb N$,}
\end{equation}
and hence $\tau S_G\tau_1$ is smoothing. 

Since $G$ acts on $Y$, it is not difficult to see that there is 
a small neighborhood $W$ of $Y$ such that 
$\supp\tau\cap GW=\emptyset$. 
Let $\gamma_0, \gamma_1\in\cC^\infty(X)^G\cap\cC^\infty_0(GW)$ 
with $\gamma_0=1$ near $\supp\gamma_1$, $\gamma_1=1$ near $Y$. 
For $i=0,1$, put $\tau_i:=1-\gamma_i$; then 
$\tau_i\in\cC^\infty(X)^G$ with 
$\supp\tau_i\cap Y=\emptyset$. Moreover,
it is straightforward to check that 
\begin{equation}\label{e-gue181019m}
\tau_0=1 \text{ on }\supp\tau,\:\:
\supp\tau_0\cap\supp(1-\tau_1)=\emptyset.
\end{equation}
From \eqref{e-gue181019m} follows
\begin{equation}\label{e-gue181019mII}
\tau S_G=\tau\tau_0 S_G=\tau\tau_0 S_G\tau_1
+\tau\tau_0 S_G(1-\tau_1). 
\end{equation}
By \eqref{e-gue181019II} $\tau_0 S_G\tau_1$
is smoothing.
In view of Theorem~\ref{t-gue181017I} and \eqref{e-gue181019m}, 
we see that 
\begin{equation}\label{eq:ma3.67}
\tau\tau_0 S_G(1-\tau_1) \:\:\text{is smoothing}. 
\end{equation}
From \eqref{eq:ma3.67}
and \eqref{e-gue181019mII}, 
we get that $\tau S_G$ is smoothing 
and hence $S_G\tau$ is smoothing. 

Let us now work under the Assumption \ref{a-gue170410I} (ii).
By \cite[Theorem 1.14]{HM14} the Szeg\H{o} projector $S$ is a 
Fourier integral operator with complex phase on $X$.
We have $S_G=S\circ\mP^{(0)}_G$ and hence
\begin{equation}\label{eq:ma3.70}
S_G(x,y)=\int_G S(x,h.y)d\mu(h).
\end{equation}
and hence $S_G$ is a Fourier integral 
with complex phase on $X$.
Thus $S_G$ maps smooth functions to smooth functions and
is smoothing outside the diagonal. Then \eqref{eq:ma3.60}
and \eqref{eq:ma3.67} still hold so by repeating
the proof above we conclude also in this case.
\end{proof}
\subsection{$G$-invariant Szeg\H{o} kernel asymptotics 
near $Y$}\label{s-gue180505zq}
In this Section $X$ has arbitrary dimension $\geq3$.
We first recall the definition of the
H\"ormander symbol spaces. Let $D\subset X$ 
be a local coordinate patch with local coordinates 
$x=(x_1,\ldots,x_{2n+1})$. 

\begin{definition}\label{d-gue140221a}
For $m\in\R$, $S^m_{1,0}(D\times D\times\mathbb{R}_+)$ 
is the space of all $a(x,y,t)\in\cC^\infty(D\times D\times\mathbb{R}_+)$ 
such that, for all compact $K\Subset D\times D$ and all 
$\alpha, \beta\in\mathbb N^{2n+1}$, $\gamma\in\N$, 
there exists $C_{\alpha,\beta,\gamma}>0$ such that 
\[\abs{\pr^\alpha_x\pr^\beta_y\pr^\gamma_t a(x,y,t)}\leq 
C_{\alpha,\beta,\gamma}(1+\abs{t})^{m-\gamma},\ \ 
 (x,y,t)\in K\times\R_+,\ \ t\geq1.\]
Put 
\[S^{-\infty}(D\times D\times\mathbb{R}_+):=
\bigcap_{m\in\R}S^m_{1,0}(D\times D\times\mathbb{R}_+).\]
Let $a_j\in S^{m_j}_{1,0}(D\times D\times\mathbb{R}_+)$, 
$j=0,1,2,\ldots$ with $m_j\searrow-\infty$, as $j\To\infty$. 
Then there exists $a\in S^{m_0}_{1,0}(D\times D\times\mathbb{R}_+)$ 
unique modulo $S^{-\infty}$, such that 
\[a-\sum^{k-1}_{j=0}a_j\in S^{m_k}_{1,0}(D\times 
D\times\mathbb{R}_+)\quad 
\text{for all $k\in\N^*$.}\]
If $a$ and $a_j$ have the properties above, we write 
$a\sim\sum^{\infty}_{j=0}a_j$ in 
$S^{m_0}_{1,0}(D\times D\times\mathbb{R}_+)$. 
We write
\begin{equation}  \label{e-gue140205III}
s(x, y, t)\in S^{m}_{{\rm cl\,}}(D\times D\times\mathbb{R}_+\big)
\end{equation}
if $s(x, y, t)\in S^{m}_{1,0}(D\times D\times\mathbb{R}_+)$ and 
\begin{equation}\label{e-fal}\begin{split}
&s(x, y, t)\sim\sum^\infty_{j=0}s_j(x, y)t^{m-j}
\text{ in }S^{m}_{1, 0}(D\times D\times\mathbb{R}_+)\,,\\
&s_j(x, y)\in\cC^\infty(D\times D),\ \text{for any $j\in\N$}.
\end{split}
\end{equation}
Let $W_1\subset\R^{N_1}$, $W_2\subset\R^{N_2}$ be open sets. 
We can also define $S^m_{1,0}(W_1\times W_2\times\mathbb{R}_+)$, 
$S^m_{{\rm cl\,}}(W_1\times W_2\times\mathbb{R}_+)$ 
and asymptotic sum in the similar way.
\end{definition}

By Theorem \ref{t-gue140305VIc}
(cf.\ \cite[Theorem 1.5]{HM14})
the spectral projector $S_{\leq\lambda}$ is for every $\lambda>0$
a complex Fourier integral 
operator on the subset where the Levi form is positive definite.
We give here a detailed description of the
spectral kernel.

\begin{theorem}[{\cite[Theorem 4.7]{HM14}}]\label{t-gue180505}
Fix $\lambda>0$. 
Let $D\subset X$ be a local coordinate patch with local coordinates 
$x=(x_1,\ldots,x_{2n+1})$. 
Assume that the Levi form of $X$ is positive definite at every point of $D$. 
Then, 
\[S_{\leq\lambda}(x, y)\equiv
\int^{\infty}_{0}e^{i\varphi(x, y)t}s(x, y, t)dt\ \
\text{on $D\times D$},\]
with a symbol 
$s(x, y, t)\in S^{n}_{ {\rm cl\,}}(D\times 
D\times\mathbb{R}_+)$
such that
the coefficient $s_0$ of the expansion \eqref{e-fal} is given by
\begin{equation}  \label{e-gue161110r}\begin{split}
&s_0(x,x)=\frac{1}{2}\pi^{-n-1}\abs{\det\cL_x},\ \  x\in D,
\end{split}\end{equation}
where $\det\cL_x$ is the determinant of the Levi form, 
see \eqref{e:detlev}, and the phase function $\varphi$ satisfies 
\begin{equation}\label{e-gue140205IV}
\begin{split}
&\varphi\in\cC^\infty(D\times D),\ \ {\rm Im\,}\varphi(x, y)\geq0,\\
&\varphi(x, x)=0,\ \ \varphi(x, y)\neq0\ \ \mbox{if}\ \ x\neq y,\\
&d_x\varphi(x, y)\big|_{x=y}=-d_y\varphi(x, y)\big|_{x=y}
=-\lambda(x)\omega_0(x),\ \ \lambda(x)>0,\\
&\varphi(x, y)=-\ol\varphi(y, x).
\end{split}
\end{equation}

Moreover, let $X':=\set{x\in X;\, \mbox{the Levi form is non-degenerate 
at $x$}}$. Then, $S_{\leq\lambda}$ is smoothing away the diagonal on 
$X'\times X'$. 
\end{theorem}

The following result describes the phase function in local coordinates 
(see chapter 8 of part I in \cite{Hsiao08}).
For a given point $p\in D$, let $\{W_j\}_{j=1}^{n}$
be an orthonormal frame of $T^{1, 0}X$ in a neighborhood of $p$
such that
the Levi form is diagonal at $p$, i.e.\ 
$\cL_{p}(W_{j},\overline{W}_{s})=\delta_{j,s}\mu_{j}$, 
$j,s=1,\ldots,n$,  where $\delta_{j,s}=1$ if $j=s$, 
$\delta_{j,s}=0$ if $j\neq s$. 
We take local coordinates $x=(x_1,\ldots,x_{2n+1})$
defined on some neighborhood of $p$ such that 
$\omega_0(p)=dx_{2n+1}$, $x(p)=0$ 
and 
\begin{equation}\label{e-gue161219a}
W_j=\frac{\pr}{\pr z_j}-i\mu_j\ol z_j\frac{\pr}{\pr x_{2n+1}}-
c_jx_{2n+1}\frac{\pr}{\pr x_{2n+1}}
+\sum^{2n}_{k=1}a_{j,k}(x)\frac{\pr}{\pr x_k}
+O(\abs{x}^2),\ j=1,\ldots,n\,,\end{equation}
where $z_j=x_{2j-1}+ix_{2j}$, 
$c_j\in\C$, $a_{j,k}(x)$ is $\cC^\infty$, $a_{j,k}(x)=O(\abs{x})$,
for every $j=1,\ldots,n$, $k=1,\ldots,2n$. 
Set
$y=(y_1,\ldots,y_{2n+1})$, $w_j=y_{2j-1}+iy_{2j}$, $j=1,\ldots,n$.
\begin{theorem} \label{t-gue161110g}
With the same notations and assumptions used in
Theorem~\ref{t-gue180505} we have 
for the phase function $\varphi$ in some neighbourhood of $(0,0)$,
\begin{equation} \label{e-gue140205VI}
{\rm Im\,}\varphi(x,y)\geq c\sum^{2n}_{j=1}\abs{x_j-y_j}^2,\ \ c>0,
\end{equation}
\begin{equation} \label{e-gue140205VII}
\begin{split}
&\varphi(x, y)=-x_{2n+1}+y_{2n+1}
+i\sum^{n}_{j=1}\abs{\mu_j}\abs{z_j-w_j}^2 \\
&\quad+\sum^{n}_{j=1}\Bigr(i\mu_j(\ol z_jw_j-z_j\ol w_j)
+c_j(-z_jx_{2n+1}+w_jy_{2n+1})\\
&\quad+\ol c_j(-\ol z_jx_{2n+1}+\ol w_jy_{2n+1})\Bigr)
+(x_{2n+1}-y_{2n+1})f(x, y) +O(\abs{(x, y)}^3),
\end{split}
\end{equation}
where $f$ is smooth and satisfies $f(0,0)=0$, $f(x, y)=\ol f(y, x)$. 
Moreover, we can take the phase $\varphi$ so that 
\begin{equation}\label{e-gue180512}
\mbox{$\ddbar_{b,x}\varphi(x,y)$ vanishes to infinite order at $x=y$}.
\end{equation}
Furthermore, for any $\varphi_1(x,y)\in\cC^\infty(D\times D)$, 
if $\varphi_1$ satisfies \eqref{e-gue140205IV},  \eqref{e-gue140205VI}, 
\eqref{e-gue140205VII} and \eqref{e-gue180512}, then there is
a function $h(x,y)\in\cC^\infty(D\times D)$ with $h(x,x)\neq0$, 
for every $x\in D$, such that 
$\varphi(x,y)-h(x,y)\varphi_1(x,y)$ vanishes to infinite order at $x=y$. 
\end{theorem} 
For the next result we recall that the map
$\mR_x$ and the function $V_{{\rm eff\,}}$
were defined in \eqref{e:rx} and \eqref{e-gue180308m}.
We denote by $\lambda_1(x),\ldots,\lambda_d(x)$ the eigenvalues 
of $\mR_x$ with respect to the $G$-invariant 
Hermitian metric $g$ and we define the determinant of $\mR_x$ by
\begin{equation}\label{e:detrx}
\det\mR_x=\lambda_1(x)\ldots\lambda_d(x).
\end{equation}
\begin{theorem}\label{t-gue180505I}
Under the assumptions of Theorem \ref{t-180428zy},
let $p\in Y$, let $U$ be an open neighborhood of $p$ 
and let $x=(x_1,\ldots,x_{2n+1})$ be local coordinates defined in $U$. 
Then, 
\begin{equation}\label{e-gue180305y}
S_G(x, y)\equiv\int^{\infty}_{0}e^{i\Phi(x, y)t}a(x, y, t)dt\ \ 
\mbox{on $U\times U$}
\end{equation}
with a symbol 
$a(x, y, t)\in S^{n-\frac{d}{2}}_{{\rm cl\,}}
(U\times U\times\mathbb{R}_+)$
such that the coefficient $a_0$ in its expansion \eqref{e-fal} satisfies
\begin{equation}  \label{e-gue180305yI}\begin{split}
&a_0(x,x)=2^{d-1}\frac{1}{V_{{\rm eff\,}}(x)}
\pi^{-n-1+\frac{d}{2}}\abs{\det \mR_x}^{-\frac{1}{2}}
\abs{\det\cL_x},\ \  x\in U\cap Y,
\end{split}\end{equation}
and the phase function $\Phi$ satisfies
\begin{equation}\label{e-gue180305yII}
\begin{split}
&\Phi(x,y)\in\cC^\infty(U\times U),\ \ {\rm Im\,}\Phi(x,y)\geq0,\\
&d_x\Phi(x,x)=-d_y\Phi(x,x)=-\lambda(x)\omega_0(x),\ \
\lambda(x)>0,\ \ x\in U\cap Y.\\
\end{split}
\end{equation}
Moreover, there exists $C\geq 1$ such that for all $(x,y)\in U\times U$, 
\begin{equation}\label{e-gue180305yIII}
\begin{split}
&\abs{\Phi(x,y)}+{\rm Im\,}\Phi(x,y)\leq C 
\left( \inf\set{d^2(g\circ x,y);\, g\in G}+d^2(x,Y)+
d^2(y,Y) \right), \\
&\abs{\Phi(x,y)}+{\rm Im\,}\Phi(x,y)\geq\frac{1}{C} 
\left( \inf\set{d^2(g\circ x,y);\, g\in G}+
d^2(x,Y)+d^2(y, Y) \right),\\
&Cd^2(x,Y)\geq 
{\rm Im\,}\Phi(x,x)\geq\frac{1}{C}d^2(x,Y),\ \  x\in U, 
\end{split}
\end{equation}
and $\Phi(x,y)$ satisfies \eqref{e-gue170126}, \eqref{e-gue200223yyd}
and \eqref{e-gue200223yydI} below. 
\end{theorem}
\begin{proof}
If Assumption \ref{a-gue170410I} (i) holds, then
by Theorem~\ref{t-gue181021} we can localize the study of 
the $G$-invariant Szeg\H{o} 
kernel $S_G$ to $Y$ and from Theorems~\ref{t-gue181021a} 
and ~\ref{t-gue180505}, we repeat 
the proof of \cite[Theorem 1.5]{HsiaoHuang17} and conclude.
If Assumption \ref{a-gue170410I} (ii) holds, 
we know by \cite[Theorem 1.14]{HM14} 
that the Szeg\H{o} projector $S$ is a Fourier integral 
operator with complex phase on $X$. 
Then by \eqref{eq:ma3.70} 
also $S_G$ is a complex Fourier integral 
operator with complex phase on $X$.
Repeating the argument from \cite[Theorem 1.5]{HsiaoHuang17}
we conclude.
\end{proof}
As applications of Theorems~\ref{t-gue181017I}, 
\ref{t-gue181021} and~\ref{t-gue180505I}, 
we establish the following regularity property for $S_G$. 
\begin{theorem}\label{t-gue200211yyd}
Under the assumptions of Theorem \ref{t-180428zy} we have 
\[S_G:\cC^\infty(X)\To H^0_b(X)^G\cap\cC^\infty(X).\]
In particular, $H^0_b(X)^G\cap\cC^\infty(X)$ is dense in 
$H^0_b(X)^G$. 
\end{theorem} 
\begin{proof}
Let $U$ be an open coordinate patch of $X$ and let 
$u\in\cC^\infty_0(U)$. 
If $U\cap Y=\emptyset$, we see in view of Theorem~\ref{t-gue181021} 
that $S_Gu\in\cC^\infty(X)$. Assume now $U\cap Y\neq\emptyset$. 
By Theorem~\ref{t-gue180505I} $S_G$ is 
a Fourier integral operator with complex phase on $U$ and hence 
$S_Gu\in\cC^\infty(U)$. 
We only need to show that $S_Gu$ is smooth outside $U$. 
Let $x_0\notin U$. If $x_0\notin Y$, by Theorem~\ref{t-gue181021} 
again we deduce that 
$S_Gu$ is smooth near $x_0$. Now, we suppose that $x_0\in Y$. 

\noindent
Case I: $Gx_0\cap U=\emptyset$. 
We can find $\tau, \tau_1\in\cC^\infty(X)^G$ 
with $\supp\tau\cap\supp\tau_1=\emptyset$, 
$\tau\equiv1$ near $x_0$, $\tau_1\equiv1$ near $\supp u$. 
We have $\tau S_Gu=\tau S_G\tau_1u$. 
In view of Theorem~\ref{t-gue181017I} under 
Assumption \ref{a-gue170410I} (i) or by the fact that $S_G$
is a Fourier integral operator under Assumption \ref{a-gue170410I} (ii), 
we see that  $\tau S_G\tau_1$ is smoothing and we deduce that 
$\tau S_G\tau_1u\in\cC^\infty(X)$. In particular, $S_Gu$
is smooth near $x_0$. 

\noindent
Case II: $Gx_0\cap U\neq\emptyset$. 
There is a $\hat g\in G$ such that $\hat g\circ x_0\in U$. 
Since $S_Gu\in\cC^\infty(U)$, $S_Gu$ is smooth near 
$\hat g\circ x_0$. Since $S_Gu$ is $G$-invariant, $S_Gu$
is smooth near $x_0$. 

We have thus proved that $S_Gu\in\cC^\infty(X)$.
By using a partition of unity we conclude. 
\end{proof}

Let $e_0$ be the identity element in $G$. 
Fix $p\in Y$. It was shown in~\cite[Theorem 3.6]{HsiaoHuang17} that 
there exist local coordinates $v=(v_1,\ldots,v_d)$ on $G$ defined in  
a neighborhood $V$ of $e_0$ with $v(e_0)=(0,\ldots,0)$ (until further
notice, we will identify the element $h\in V$ with $v(h)$),
local coordinates $x=(x_1,\ldots,x_{2n+1})$ of $X$ defined in 
a neighborhood $U=U_1\times U_2$ of $p$ with $0\leftrightarrow p$, 
where $U_1\subset\R^d$ is an neighborhood of $0\in\R^d$,  
$U_2\subset\R^{2n+1-d}$ is an open neighborhood of $0\in\R^{2n+1-d}$
and a smooth function 
$\gamma=(\gamma_1,\ldots,\gamma_d)\in\cC^\infty(U_2,U_1)$
with $\gamma(0)=0\in\R^d$  such that for
$(v_1,\ldots,v_d)\in V$, $(x_{d+1},\ldots,x_{2n+1})\in U_2$,
\begin{equation}\label{e-gue161202m}
\begin{split}
&(v_1,\ldots,v_d)\circ (\gamma(x_{d+1},
\ldots,x_{2n+1}),x_{d+1},\ldots,x_{2n+1})\\
&=(v_1+\gamma_1(x_{d+1},\ldots,x_{2n+1}),\ldots,v_d
+\gamma_d(x_{d+1},\ldots,x_{2n+1}),x_{d+1},\ldots,x_{2n+1}),\\
&Y\cap U=\set{x_{d+1}=\ldots=x_{2d}=0},\quad
\underline{\kg }={\rm span\,}
\set{\frac{\pr}{\pr x_1},\ldots,\frac{\pr}{\pr x_d}},\\
&J\Big(\frac{\pr}{\pr x_j}\Big)
=\frac{\pr}{\pr x_{d+j}}+a_j(x)\frac{\pr}{\pr x_{2n+1}},\quad
\text{on $Y\cap U$, for $1\leq j\leq d$}, 
\end{split}
\end{equation}
where $a_j(x)$ are smooth functions on $Y\cap U$, independent of 
$x_1,\ldots,x_{2d}$, $x_{2n+1}$, $a_j(0)=0$, and 
$T^{1,0}_pX={\rm span\,}\set{Z_1,\ldots,Z_n}$ with 
\begin{equation}\label{e-gue161202Im}
\begin{split}
&Z_j=\frac{1}{2}\left(\frac{\pr}{\pr x_j}
-i\frac{\pr}{\pr x_{d+j}}\right)(p),\ \
\text{for $j=1,\ldots,d$},\\
&Z_j=\frac{1}{2}\left(\frac{\pr}{\pr x_{2j-1}}
-i\frac{\pr}{\pr x_{2j}}\right)(p),\ \ 
\text{for $j=d+1,\ldots,n$},\\
&\langle\,Z_j,Z_k\,\rangle_g=\delta_{j,k},\ \ 
\cL_p(Z_j, \ol Z_k)=\mu_j\delta_{j,k},\ \ \text{for $j,k=1,2,\ldots,n$},
\end{split}
\end{equation}
\begin{equation}\label{e-gue161219m}
\begin{split}
\omega_0(x)&=(1+O(\abs{x}))dx_{2n+1}
+\sum^d_{j=1}4\mu_jx_{d+j}dx_j+
\sum^n_{j=d+1}2\mu_jx_{2j}dx_{2j-1}\\
&\quad-\sum^n_{j=d+1}2\mu_jx_{2j-1}dx_{2j}
+\sum^{2n}_{j=d+1}b_jx_{2n+1}dx_j+O(\abs{x}^2),
\end{split}
\end{equation}
where  $b_{d+1},\ldots,b_{2n}\in\R$.  
Put 
\begin{align}\label{eq:ma3.83}
x''=(x_{d+1},\ldots,x_{2n+1}),\:\:
\widehat x''=(x_{d+1},\ldots,x_{2d}),\:\:
\mathring{x}''=(x_{d+1},\ldots,x_{2n}).
\end{align}
 
\begin{theorem}[{\cite[Theorem 1.11]{HsiaoHuang17}}]
	\label{t-gue170126} 
The phase function 
$\Phi(x,y)$ appearing in the expression
of the Szeg\H{o} kernel \eqref{e-gue180305y}, \eqref{e-gue180305yII}
is independent of $(x_1,\ldots,x_d)$ and $(y_1,\ldots,y_d)$. 
Hence, 
\begin{align}\label{eq:ma3.84}
\Phi(x,y)=\Phi((0,x''),(0,y''))=:\Phi(x'',y'').
\end{align}
Moreover, there exists $c>0$ such that 
\begin{equation}\label{e-gue170126}
{\rm Im\,}\Phi(x'',y'')\geq c\Bigr(\abs{\widehat x''}^2
+\abs{\widehat y''}^2+
\abs{\mathring{x}''-\mathring{y}''}^2\Bigr),\ \  \text{for $(0,x''),(0,y'')
\in U$},
\end{equation}
and there exists a smooth function 
$g(x,y)\in\cC^\infty(U\times U,T^{*0,1}X)$ such that 
\begin{equation}\label{e-gue200223yyd}
\mbox{$\ddbar_{b,x}\Phi(x'',y'')-g(x,y)\Phi(x'',y'')$ vanishes to
infinite order on 
$\diag\big((Y\cap U)\times(Y\cap U)\big)$,}
\end{equation}
and with the same $\mu_j, b_{d+1},\ldots,b_{2n}\in\R$ 
as in \eqref{e-gue161219m} we have 
\begin{equation}\label{e-gue200223yydI}
\begin{split}
\Phi(x'', y'')&=-x_{2n+1}+y_{2n+1}+
2i\sum^d_{j=1}\abs{\mu_j}y^2_{d+j}
+2i\sum^d_{j=1}\abs{\mu_j}x^2_{d+j}\\
&+i\sum^{n}_{j=d+1}\abs{\mu_j}\abs{z_j-w_j}^2 +
\sum^{n}_{j=d+1}i\mu_j(\ol z_jw_j-z_j\ol w_j)\\
&+\sum^d_{j=1}(-b_{d+j}x_{d+j}x_{2n+1}+b_{d+j}y_{d+j}y_{2n+1})\\
&+\sum^n_{j=d+1}\frac{1}{2}(b_{2j-1}-ib_{2j})(-z_jx_{2n+1}
+w_jy_{2n+1})\\
&+\sum^n_{j=d+1}\frac{1}{2}(b_{2j-1}+ib_{2j})(-\ol z_jx_{2n+1}
+\ol w_jy_{2n+1})\\
&+(-x_{2n+1}+y_{2n+1})f(x'',y'')+O(\abs{(x'', y'')}^3),
\end{split}
\end{equation}
where $z_j=x_{2j-1}+ix_{2j}$, $w_j=y_{2j-1}+iy_{2j}$, 
for $j=d+1,\ldots,n$, 
and $f(x'',y'')\in\cC^\infty(U\times U)$ with $f(0,0)=0$. 
\end{theorem}
 
\begin{remark}\label{r-gue200223yydI} 
The phase function $\Phi(x'',y'')$ is not unique. For example, 
we can replace $\Phi(x'',y'')$ by $\Phi(x'',y'')r(x'',y'')$, where 
$r(x'',y'')\in\cC^\infty(U\times U)$, 
$r(0,0)=1$. In~\cite[Theorem 5.2]{HsiaoHuang17}, the first author
and Huang characterized the phase $\Phi$. Since 
$\pr_{y_{2n+1}}\Phi(0,0)\neq0$, the Malgrange preparation
theorem~\cite[Theorem 7.57]{Hor03} implies that 
\[\Phi(x'',y'')=g(x'',y'')(y_{2n+1}+\widehat\Phi(x'',\mathring{y}''))\]
in some neighborhood of $(0,0)$, where 
$g(x'',y''), \widehat\Phi(x'',\mathring{y}'')\in\cC^\infty(U\times U)$.
We can replace $\Phi(x'',y'')$ by 
$y_{2n+1}+\widehat\Phi(x'',\mathring{y}'')$.
From now on, we assume that 
\begin{equation}\label{e-gue180425}
\Phi(x'', y'')=y_{2n+1}+\widehat\Phi(x'',\mathring{y}''),\ \ 
\widehat\Phi(x'',\mathring{y}'')\in\cC^\infty(U\times U).
\end{equation}
It is straightforward to check that the phase $\Phi(x'',y'')$ 
satisfies \eqref{e-gue180305yII}, \eqref{e-gue180305yIII}, 
\eqref{e-gue170126}, and
\begin{equation}\label{e-gue180425I}
\mbox{$\ddbar_{b,x}\Phi(x'',y'')$ vanishes to infinite order at 
$\diag\big((Y\cap U)\times(Y\cap U)\big)$}
\end{equation}
and with the same notations as in \eqref{e-gue200223yydI}
we have
\begin{equation}\label{e-gue170126I}
\begin{split}
\Phi(x'', y'')&=-x_{2n+1}+y_{2n+1}+
2i\sum^d_{j=1}\abs{\mu_j}y^2_{d+j}
+2i\sum^d_{j=1}\abs{\mu_j}x^2_{d+j}\\
&+i\sum^{n}_{j=d+1}\abs{\mu_j}\abs{z_j-w_j}^2 +
\sum^{n}_{j=d+1}i\mu_j(\ol z_jw_j-z_j\ol w_j)\\
&+\sum^d_{j=1}(-b_{d+j}x_{d+j}x_{2n+1}+b_{d+j}y_{d+j}x_{2n+1})\\
&+\sum^n_{j=d+1}\frac{1}{2}(b_{2j-1}-ib_{2j})(-z_jx_{2n+1}
+w_jx_{2n+1})\\
&+\sum^n_{j=d+1}\frac{1}{2}(b_{2j-1}+ib_{2j})(-\ol z_jx_{2n+1}
+\ol w_jx_{2n+1})
+O(\abs{(x'', \mathring{y}'')}^3).
\end{split}
\end{equation}
\end{remark}

\section{The distribution kernels of the maps
$\sigma$ and $\sigma^*\sigma$}\label{s-gue180308}

In this Section we will study the map $\sigma$ defined in
\eqref{e-gue180308Im} and prove Theorem \ref{t-gue180428zy}.
We assume throughout that $\dim X_G\geq3$ and 
$\ddbar_{b,X_{G}}$ has closed range in $L^2$. 
The case when $\dim X_G=1$ will be treated separately.

Let $\iota^*:\cC^\infty(X)\To \cC^\infty(Y)$ be the pull-back 
of the  inclusion $\iota:Y\To X$. Let 
$\iota_G: \cC^\infty(Y)^G\To \cC^\infty(X_{G})$
be the natural identification. 
Let $S_{X_{G}}:L^2(X_{G})\To H^{0}_{b}(X_{G})$ 
be the orthogonal projection
with respect to $(\,\cdot\,,\,\cdot\,)_{X_{G}}$ 
(cf.\ Convention \ref{conv_met}).
By Theorem \ref{t-gue200211yyd} and \eqref{e-gue200212yyd} 
(here we use that $\ddbar_{b,X_{G}}$ has closed range in $L^2$),
 we can extend $\sigma$ to $\cC^\infty(X)$ by 
\begin{equation}\label{e-gue180308II}
\begin{split}
\sigma: \cC^\infty(X)&\To H^{0}_{b}(X_{G})
\cap\cC^\infty(X_{G})\subset\cC^\infty(X_{G}),\\
\sigma&=S_{X_{G}}\circ E\circ \iota_G
\circ f_{G}\circ \iota^*\circ S_G\,.
\end{split}
\end{equation}
Let $\sigma^{\,*}: \cC^\infty(X_{G})\To \cD'(X)$ be
the formal adjoint of $\sigma$. 
We will show in Theorem~\ref{t-gue170304ry} 
that $\sigma^{\,*}$ actually maps $\cC^\infty(X_{G})$ into
$H^0_{b}(X)^G\cap\cC^\infty(X)^G$.

In this Section, we will study the distribution kernels of 
$\sigma$ and $\sigma^{\,*}\sigma$. 
We explain briefly the role of the operator $E$ in \eqref{e-gue180308II}. 
To prove our main result, we need to show that 
$\sigma^{\,*}\sigma$ is "microlocally close" to 
$S_G$. In other words, we want $\sigma^{\,*}\sigma$ 
to be a complex Fourier integral operator with the same phase, the same 
order and the same leading 
symbol as $S_G$. To achieve this we need to take $E$
to be a classical elliptic pseudodifferential operator with principal 
symbol $p_E(x,\xi)=\abs{\xi}^{-d/4}$, see also 
Remark~\ref{r-gue181220mp}. 

This Section is organized as follows. In Section \ref{s-gue180308I}
we develop the calculus of Fourier integral operators of
$G$-Szeg\H{o} type. In Section \ref{s-gue180512}
we study the distribution kernels of $\sigma$
and $\sigma^*\sigma$ and prove Theorem \ref{t-gue180428zy}.

\subsection{Calculus of complex Fourier integral operators}
\label{s-gue180308I}

Let $p\in Y$ and $x=(x_1,\ldots,x_{2n+1})$ be the local coordinates
as in the discussion before Theorem~\ref{t-gue170126} defined
in an open neighborhood $U$ of $p$. From now on, we change $x_{2n+1}$ 
to $x_{2n+1}-\sum^d_{j=1}a_j(x)x_{d+j}$, 
where $a_j(x)$ are as in \eqref{e-gue161202m}. 
With this new local coordinates $x=(x_1,\ldots,x_{2n+1})$
on $Y\cap U$ we have 
\begin{equation}\label{e:ma4.2}
J\Big(\frac{\pr}{\pr x_j}\Big)=\frac{\pr}{\pr x_{d+j}},\ \ 
\text{ for }j=1,2,\ldots,d.
\end{equation}
Moreover, it is clear that 
$\Phi(x,y)+\sum^d_{j=1}a_j(x)x_{d+j}-\sum^{d}_{j=1}a_j(y)y_{d+j}$
satisfies \eqref{e-gue170126I}. Note that $a_j(x)$ is a smooth function 
on $Y\cap U$, independent of $x_1,\ldots,x_{2d}$, $x_{2n+1}$ 
and $a_j(0)=0$, $j=1,\ldots,d$. We may assume that 
$U=\Omega_1\times\Omega_2\times\Omega_3$,
where $\Omega_1\subset\R^d$, $\Omega_2\subset\R^d$ 
are open  neighborhoods of $0\in\R^d$,
$\Omega_3\subset\R^{2n+1-2d}$
is an open  neighborhood of $0\in\R^{2n+1-2d}$. 
From now on, we identify 
$\Omega_2$ with 
\[\set{(0,\ldots,0,x_{d+1},\ldots,x_{2d},0,\ldots,0)\in U;\, 
(x_{d+1},\ldots,x_{2d})\in\Omega_2},\] 
$\Omega_3$ with $\set{(0,\ldots,0,x_{2d+1},\ldots,x_{2n+1})\in U;\, 
(x_{2d+1},\ldots,x_{2n+1})\in\Omega_3}$,  $\Omega_2\times\Omega_3$
with 
\[\set{(0,\ldots,0,x_{d+1},\ldots,x_{2n+1})\in U;\,
(x_{d+1},\ldots,x_{2n+1})\in\Omega_2\times\Omega_3}.\] 
For $x=(x_1,\ldots,x_{2n+1})$, we set
\begin{equation}\label{e:ma4.3}
\begin{split}
x''&=(x_{d+1},\ldots,x_{2n+1}),\quad \mathring{x}''=(x_{d+1},
\ldots,x_{2n},0),\\
\widehat{x}''&=(x_{d+1},\ldots,x_{2d}),\quad 
\underline{x}''=(x_{2d+1},\ldots,x_{2n+1}).
\end{split}
\end{equation}
From now on, we identify 
\begin{equation}\label{e:ma4.4}
\begin{split}
&\text{$x''$ with $(0,\ldots,0,x_{d+1},\ldots,x_{2n+1})\in U$},\\ 
&\text{$\widehat x''$ with 
$(0,\ldots,0,x_{d+1},\ldots,x_{2d},0,\ldots,0)\in U$},\\ 
&\text{$\underline{x}''$ with 
$(0,\ldots,0,x_{2d+1},\ldots,x_{2n+1})\in U$}.
\end{split}
\end{equation}
Since $G$ acts freely on $Y$, we take $\Omega_2$ and $\Omega_3$
small enough so that if $x, x_1\in\Omega_2\times\Omega_3$ 
and $x\neq x_1$, then 
\begin{equation}\label{e-gue170227c}
 \mbox{$g\circ x\neq x_1$, for all $g\in G$.}
\end{equation}
Recall that we take $\Phi$ so that \eqref{e-gue180425}, 
\eqref{e-gue180425I}, \eqref{e-gue170126I} hold. 
Put 
\begin{equation}\label{e-gue180505q}
\Phi^*(x,y):=-\ol\Phi(y,x).
\end{equation}
 From \eqref{e-gue180425I} and notice that for 
$j=1,\ldots,d$, $x\in Y$, we have $\frac{\pr}{\pr x_j}
+i\frac{\pr}{\pr x_{d+j}}\in T^{0,1}_xX$ and
$\frac{\pr}{\pr x_j}\Phi(x,y)=\frac{\pr}{\pr y_j}\Phi^*(x,y)=0$,
we conclude that  for $j=1,\ldots,d$,
\begin{equation}\label{e-gue180505p}
\begin{split}
&\frac{\pr}{\pr x_{d+j}}\Phi(x,y)\Big|_{x_{d+1}=\ldots=x_{2d}=0}\:\: 
\text{and $\frac{\pr}{\pr y_{d+j}}\Phi^*(x,y)\Big|_{y_{d+1}
=\ldots=y_{2d}=0}$}\\
&\qquad\qquad \text{vanish to infinite order at
$\diag\big((Y\cap U)\times(Y\cap U)\big)$}.
\end{split}
\end{equation}
Let $G_j(x,y):=\frac{\pr}{\pr y_{d+j}}\Phi^*(x,y)\big|_{y_{d+1}
=\ldots=y_{2d}=0}$\,, $H_j(x,y):=\frac{\pr}{\pr x_{d+j}}
\Phi(x,y)\big|_{x_{d+1}=\ldots=x_{2d}=0}$. Put
\begin{equation}\label{e-gue170226}
\begin{split}
&\Phi_1(x,y):=\Phi^*(x,y)-\sum^d_{j=1}y_{d+j}G_j(x,y),\\
&\Phi_2(x,y):=\Phi(x,y)-\sum^d_{j=1}x_{d+j}H_j(x,y).
\end{split}
\end{equation}
Then for $j=1,2,\ldots,d,$
\begin{equation}\label{e-gue170226Ip}
\begin{split}
\frac{\pr}{\pr y_{d+j}}\Phi_1(x,y)\Big|_{y_{d+1}=\ldots=y_{2d}=0}=
\frac{\pr}{\pr x_{d+j}}\Phi_2(x,y)\Big|_{x_{d+1}=\ldots=x_{2d}=0}=0,\
\end{split}
\end{equation}
and 
\begin{equation}\label{e-gue170226II}
\begin{split}
&\mbox{$\Phi^*(x,y)-\Phi_1(x,y)$ vanishes to infinite order on 
$\diag\big((Y\cap U)\times(Y\cap U)\big)$},\\
&\mbox{$\Phi(x,y)-\Phi_2(x,y)$ vanishes to infinite order on 
$\diag\big((Y\cap U)\times(Y\cap U)\big)$}.
\end{split}
\end{equation}

We also write $u=(u_1,\ldots,u_{2n+1})$ to denote the local coordinates 
of $U$. For any smooth function $h\in\cC^\infty(U)$, let
$\Td h\in\cC^\infty(U^{\C})$ be an almost analytic extension 
of $h$ (see~\cite[Section 1]{MS74}). 
Let $\vartheta$ be a local coordinate of $\R$. Let 
\begin{equation}\label{e-gue180506}
F(\Td x,\Td y,\Td{u''},\Td{\vartheta}\,):=\Td\Phi_1(\Td x,\Td{u''})
+\Td{\vartheta}\,\Td\Phi_2(\Td{u''},\Td y).
\end{equation}
We consider the following two systems for $\ j=1,2,\ldots,2n-2d+1$ 
and $\ j=1,2,\ldots,2n-d+1$, respectively, 
\begin{equation}\label{e-gue170226III}
\begin{split}
&\frac{\pr F}{\pr\Td{\vartheta}}(\Td x,\Td y,\Td{\underline{u}''},
\Td{\vartheta}\,)=
\Td\Phi_2(\Td{\underline{u}''},\Td y)=0,\\
&\frac{\pr F}{\pr\Td u_{2d+j}}(\Td x,\Td y,\Td{\underline{u}''},
\Td{\vartheta}\,)=
\frac{\pr\Td\Phi_1}{\pr\Td y_{2d+j}}
(\Td x,\Td{\underline{u}''})+
\Td{\vartheta}\frac{\pr\Td\Phi_2}{\pr\Td x_{2d+j}}
(\Td{\underline{u}''},\Td y)=0, 
\end{split}
\end{equation}
and 
\begin{equation}\label{e-gue170226IIIa}
\begin{split}
&\frac{\pr F}{\pr\Td{\vartheta}}(\Td x,\Td y,\Td{u''},\Td{\vartheta}\,)=
\Td\Phi_2(\Td{u''},\Td y)=0,\\
&\frac{\pr F}{\pr\Td u_{d+j}}(\Td x,\Td y,\Td{u''},\Td{\vartheta}\,)=
\frac{\pr\Td\Phi_1}{\pr\Td y_{d+j}}(\Td x,\Td{u''})+
\Td{\vartheta}\frac{\pr\Td\Phi_2}{\pr\Td x_{d+j}}(\Td{u''},\Td y)=0,
\end{split}
\end{equation}
where $\Td{\underline{u}''}=(0,\ldots,0,\Td u_{2d+1},\ldots,
\Td u_{2n+1})$, 
$\Td{u''}=(0,\ldots,0,\Td u_{d+1},\ldots,\Td u_{2n+1})$. 
 Here we  always 
use $\frac{\pr}{\pr\Td x_k}$ for first variable, $\frac{\pr}{\pr\Td y_k}$
for second variable.
From \eqref{e-gue180425} and \eqref{e-gue170226Ip}, 
we can take $\Td\Phi_1$ and $\Td\Phi_2$ so that
for every $j=1,2,\ldots,d$, 
\begin{equation}\label{e-gue170226I}
\begin{split}
&\frac{\pr\Td\Phi_1}{\pr\Td y_{d+j}}(\Td x,\Td{u''})=
\frac{\pr\Td\Phi_2}{\pr\Td x_{d+j}}(\Td{u''},\Td y)=0,\ \ 
\mbox{if $\Td u_{d+1}=\ldots=\Td u_{2d}=0$},
\end{split}
\end{equation}
and $\Td{\widehat\Phi}_1, \Td{\widehat\Phi}_2\in
\cC^\infty(U^\C\times U^\C)$ such that
\begin{equation}\label{e-gue170227}
\begin{split}
\Td\Phi_1(\Td x, \Td y)=-\Td x_{2n+1}
+\Td{\widehat\Phi}_1(\Td{\mathring{x}}'',\Td{y''}),\ \ 
\Td\Phi_2(\Td x, \Td y)=\Td y_{2n+1}
+\Td{\widehat\Phi}_2(\Td{x''},\Td{\mathring{y}}''),
\end{split}
\end{equation}
where $\Td{\mathring{x}}''=(0,\ldots,0,\Td x_{d+1},\ldots,
\Td x_{2n},0)$, $\Td{\mathring{y}}''=(0,\ldots,0,\Td y_{d+1},\ldots,
\Td y_{2n},0)$. 

From  \eqref{e-gue180305yII}, \eqref{e-gue180305yIII} 
and \eqref{e-gue170126I},
it is not difficult to see that 
\begin{equation}\label{e:ma4.16}
\begin{split}
&\Td\Phi_2(\underline{x}'',\underline{x}'')=0,\\
&\frac{\pr\Td\Phi_1}{\pr\Td y_{d+j}}
(\underline{x}'',\underline{x}'')+
\Td{\vartheta}\frac{\pr\Td\Phi_2}{\pr\Td x_{d+j}}
(\underline{x}'',\underline{x}'')\big|_{\Td{\vartheta}=1}=0,\ \ 
j=1,2,\ldots,2n-d,
\end{split}\end{equation}
and the Hessians 
\begin{equation}\label{e:ma4.17}
F_{\Td{\vartheta},\Td{\underline{u}''}}(0,0,0,1)=\begin{pmatrix}
  &\dfrac{\pr^2F}{\pr^2\Td{\vartheta}}  
  &\dfrac{\pr^2F}{\pr\Td{\vartheta}\pr\Td{\underline{u}''}} \\[15pt]
  &\dfrac{\pr^2F}{\pr\Td{\underline{u}''}\pr\Td{\vartheta}} 
  &\dfrac{\pr^2F}{\pr^2\Td{\underline{u}''}}
\end{pmatrix}\Big|_{(0,0,0,1)}
\end{equation}
and 
\begin{equation}\label{e:ma4.18}
F_{\Td{\vartheta},\Td{u''}}(0,0,0,1)=\begin{pmatrix}
  &\dfrac{\pr^2F}{\pr^2\Td{\vartheta}} 
  &\dfrac{\pr^2F}{\pr\Td{\vartheta}\pr\Td{u''}} \\[15pt]
  &\dfrac{\pr^2F}{\pr\Td{u''}\pr\Td{\vartheta}} 
  &\dfrac{\pr^2F}{\pr^2\Td{u''}}
\end{pmatrix}\Big|_{(0,0,0,1)}
\end{equation}
are non-singular. Moreover, from \eqref{e-gue170126I}, we calculate
\begin{equation}\label{e-gue170226a}
\begin{split}
&\det\left(\frac{1}{2\pi i}F_{\Td{\vartheta},
\Td{\underline{u}''}}\right)(0,0,0,1)
=\frac{2^{2n-2d-2}}{\pi^{2n-2d+2}}
(\abs{\mu_{d+1}}\ldots\abs{\mu_n})^2,\\
&\det\left(\frac{1}{2\pi i}F_{\Td{\vartheta},
\Td{\underline{u}''}}\right)(0,0,0,1)
=\frac{2^{2n-2}}{\pi^{2n-d+2}}(\abs{\mu_1}\ldots\abs{\mu_d})
(\abs{\mu_{d+1}}\ldots\abs{\mu_n})^2.
\end{split}
\end{equation}
Hence, near $(p,p)$ and $\Td{\vartheta}=1$, we can 
solve \eqref{e-gue170226III} 
and \eqref{e-gue170226IIIa} and the solutions are unique. 
Let
\begin{equation}\label{eq:ma4.20}
\begin{split}
&\Td{\underline{u}''}=\alpha(x,y)=(\alpha_{2d+1}(x,y),
\ldots,\alpha_{2n+1}(x,y))\in 
\cC^\infty(U\times U,\C^{2n-2d+1}),\\
&\Td{\vartheta}=\gamma(x,y)\in\cC^\infty(U\times U,\C),\end{split}
\end{equation}
and 
\begin{equation}\label{eq:ma4.21}
\begin{split}
&\Td{u''}=\beta(x,y)=(\beta_{d+1}(x,y),\ldots,\beta_{2n+1}(x,y))\in 
\cC^\infty(U\times U,\C^{2n-d+1}),\\
&\Td{\vartheta}=\delta(x,y)\in\cC^\infty(U\times U,\C)
\end{split}
\end{equation}
 be the solutions of \eqref{e-gue170226III} and 
 \eqref{e-gue170226IIIa}, respectively. From \eqref{e-gue170226I},
 it is easy to see that 
\begin{equation}\label{e-gue170226b}
\begin{split}
&\beta(x,y)=(\beta_{d+1}(x,y),\ldots,\beta_{2n+1}(x,y))
=(0,\ldots,0,\alpha_{2d+1}(x,y),\ldots,\alpha_{2n+1}(x,y)),\\
&\gamma(x,y)=\delta(x,y).
\end{split}
\end{equation}
From \eqref{e-gue170226b}, we see that the value 
of $\Td\Phi_1(x,\Td{\underline{u}''})+
\Td{\vartheta}\Td\Phi_2(\Td{\underline{u}''},y)$ 
at critical points $\Td{\underline{u}''}=\alpha(x,y)$, 
$\Td{\vartheta}=\gamma(x,y)$ 
is equal to the value of 
$\Td\Phi_1(x,\Td{u''})+\Td{\vartheta}\Td\Phi_2(\Td{u''},y)$ 
at critical points
$\Td{u''}=\beta(x,y)$, $\Td{\vartheta}=\delta(x,y)$. Put 
\begin{equation}\label{e-gue170226bI}
\begin{split}
\Phi_3(x,y)&:=\Td\Phi_1(x,\alpha(x,y))
+\gamma(x,y)\Td\Phi_2(\alpha(x,y),y)\\
&=\Td\Phi_1(x,\beta(x,y))+\delta(x,y)\Td\Phi_2(\beta(x,y),y).
\end{split}
\end{equation}
Then $\Phi_3(x,y)$ is a complex phase function, 
$\im\Phi_3(x,y)\geq0$. 
It is easy to check that 
\begin{equation}\label{eq:ma4.25}
\begin{split}
d_x\Phi_3(x,x)&=-d_y\Phi_3(x,x)=d_x\Phi(x,x)\\
&=-d_y\Phi(x,x)
=-\lambda(x)\omega_0(x),\ \  \lambda(x)>0,\ \ 
\text{for $x\in U\cap Y$}.
\end{split}
\end{equation}
From now on we take $U$ small enough so that the Levi form is 
positive on $U$ and 
\begin{equation}\label{eq:ma4.26}
d_x\Phi_3(x,y)\neq0,\:\: 
d_y\Phi_3(x,y)\neq0,\:\: \text{for every $(x,y)\in U\times U$}
\end{equation}
and 
$a_0(x,y)\neq0$, for every $(x,y)\in U\times U$, where 
$a_0(x,y)\in\cC^\infty(U\times U)$ is as in \eqref{e-gue180305yI}. 

Fix an open neighborhood $\mU\Subset U$ of $p$ with 
$\widehat\Omega_2\times\widehat\Omega_3\subset\mU$, 
where $\widehat\Omega_2\Subset\Omega_2\subset\R^d$ 
is an open neighborhood of $0\in\R^d$ and 
$\widehat\Omega_3\Subset\Omega_3\subset\R^{2n+1-2d}$ 
is an open neighborhood of $0\in\R^{2n+1-2d}$. 
\begin{theorem}\label{t-gue170226cw}
The phase functions $\Phi$ and $\Phi_3$ are equivalent on $\mU$,
that is, for any
$b_1\in  S^{n-\frac{d}{2}}_{{\rm cl\,}}
(\mU\times\mU\times\mathbb{R}_+)$
there exist
$\hat{b}_1,b_2\in  S^{n-\frac{d}{2}}_{{\rm cl\,}}
(\mU\times\mU\times\mathbb{R}_+)$
such that
\begin{equation}\label{eq:ma4.29}
\begin{split}
\int^\infty_0e^{i\Phi(x,y)t}b_1(x,y,t)dt\equiv
\int^\infty_0e^{i\Phi_3(x,y)t}b_2(x,y,t)dt\ \ 
\text{on $\mU\times\mU$},\\
\int^\infty_0e^{i\Phi(x,y)t}\hat b_1(x,y,t)dt\equiv
\int^\infty_0e^{i\Phi_3(x,y)t}b_1(x,y,t)dt\ \ 
\text{on $\mU\times\mU$}.
\end{split}
\end{equation}
\end{theorem}
\begin{proof}
We consider the kernel $(S_{G}\circ S_{G})(\cdot,\cdot)$ on $U$. 
Let $\mU\Subset U_1\Subset U$ be open neighborhoods of $p$. 
Let $\chi(x'')\in\cC^\infty_0(\Omega_2\times\Omega_3)$. 
From \eqref{e-gue170227c}, we can extend $\chi(x'')$ to 
\[W:=\set{g\circ x;\, g\in G,\ \  x\in\Omega_2\times\Omega_3}\]
by $\chi(g\circ x''):=\chi(x'')$, for every $g\in G$. 
Assume that  $\chi=1$ on some neighborhood of $\overline{U}_1$. 
Let $\chi_1\in\cC^\infty_0(U)$ with $\chi_1=1$ on some neighborhood 
of $\overline{U}_1$ and $\supp\chi_1\subset\set{x\in X;\, \chi(x)=1}$. 
We have 
\begin{equation}\label{e-gue170227b}
\chi_1S_G\circ S_G=\chi_1S_G\chi\circ S_G+\chi_1S_G(1-\chi)\circ S_G.
\end{equation}
Let's first consider $\chi_1S_G(1-\chi)\circ S_G$. We have 
\begin{equation}\label{e-gue170227bI}
\begin{split}
(\chi_1S_G(1-\chi))(x,u)
&=\chi_1(x)\int_GS_{\leq\lambda_0}(x,g\circ u)(1-\chi(u))d\mu(g)\\
&=\chi_1(x)\int_GS_{\leq\lambda_0}(x,u)(1-\chi(g^{-1}u))d\mu(g),
\end{split}
\end{equation}
where $\lambda_0>0$ is a small constant as in 
Theorem~\ref{t-gue181021a}. 
If $g^{-1}u\notin\set{x\in X;\, \chi(x)=1}$. 
Since $\supp\chi_1\subset\set{x\in X;\, \chi(x)=1}$ 
and $\chi(x)=\chi(g\circ x)$, for every $g\in G$, $x\in X$, 
we conclude that $u\notin\supp\chi_1$. 
From this observation and notice that $S_{\leq\lambda_0}$ 
is smoothing away the diagonal on $GU$
(see Theorem~\ref{t-gue180505}), we  deduce that $\chi_1S_G(1-\chi)$
is smoothing and hence 
\begin{equation}\label{e-gue170227bII}
\chi_1S_G(1-\chi)\circ S_G\equiv 0\ \ \mbox{on $X\times X$}. 
\end{equation}
From \eqref{e-gue170227b} and \eqref{e-gue170227bII}, we get 
\begin{equation}\label{e-gue170227bIII}
\chi_1S_{G}\circ S_{G}\equiv\chi_1S_{G}\chi\circ S_{G}\ \ 
\mbox{on $X\times X$}. 
\end{equation}
From Theorem~\ref{t-gue180505I} and using that $S^*_G=S_G$,
where $S^*_G$ is the adjoint of $S_G$ with respect to 
$(\,\cdot\,,\,\cdot\,)$, we obtain that on $U$, 
\begin{equation}\label{e-gue170227y}
\begin{split}
&(\chi_1S_{G}\chi\circ S_{G})(x,y)\\
&\equiv \int_{\Omega_2\times\Omega_3}
\int^\infty_0\int^\infty_0e^{i\Phi^*(x,u'')t
+i\Phi(u'',y)s}\chi_1(x)a^*(x'',u'',t)V_{{\rm eff\,}}(u'')
\chi(u'')a(u'',y'',s)dsdtdv(u'')\\
&\equiv \int_{\Omega_2\times\Omega_3}
\int^\infty_0\int^\infty_0e^{i\Phi_1(x,u'')t
+i\Phi_2(u'',y)s}\chi_1(x)b(x,u'',t)\chi(u'')c(u'',y,s)dsdtdv(u'')\\
&\mbox{(here we used \eqref{e-gue170226II})}\\
&\equiv \int_{\Omega_2\times\Omega_3}
\int^\infty_0\int^\infty_0e^{i(\Phi_1(x,u'')
+ \Phi_2(u'',y){\vartheta})t}\chi_1(x)b(x,u'',t)
\chi(u'')c(u'',y,t{\vartheta})td{\vartheta} dtdv(u''),
\end{split}
\end{equation}
where  $s={\vartheta} t$, $d\mu(g)dv(u'')=dv(x)$ on $U$, 
$a^*(x'',u'',t)=\ol a(u'',x'',t)$ and 
\begin{equation}  \label{e-gue180511}\begin{split}
&b(x, y, t), c(x,y,t)\in S^{n-\frac{d}{2}}_{ {\rm cl\,}}
(U\times U\times\mathbb{R}_+), \\
& \mbox{$b_0(x,x)\neq0$, $c_0(x,x)\neq0$, for any 
$x\in U\cap Y$ (cf. Notation \eqref{e-fal})}.
\end{split}\end{equation}

We apply the complex stationary phase formula of 
Melin-Sj\"ostrand~\cite{MS74} to carry out the $dv(u'')d{\vartheta}$ 
integration in \eqref{e-gue170227y}. This yields the existence
of a symbol $d(x,y,t)\in S^{n-\frac{d}{2}}_{{\rm cl\,}}
(U\times U\times\mathbb{R}_+)$ with the expansion
$d(x,y,t)\sim\sum^\infty_{j=0}t^{n-\frac{d}{2}-j}d_j(x,y)$ in 
$S^{n-\frac{d}{2}}_{1, 0}(U\times U\times\mathbb{R}_+)$
(cf.\ \eqref{e-fal}) with $d_0(x,x)\neq0$ for $x\in U\cap Y$ and
\begin{equation}\label{e-gue170227yI}
\begin{split}
&(\chi_1S_{G}\chi\circ S_{G})(x,y)\equiv
\int^\infty_0e^{i\Phi_3(x,y)t}d(x,y,t)dt\ \ \mbox{on $U\times U$}.
\end{split}
\end{equation}
From \eqref{e-gue170227bIII}, \eqref{e-gue170227yI}, we deduce that 
\begin{equation}\label{e-gue170227yII}
\int^\infty_0e^{i\Phi_3(x,y)t}d(x,y,t)dt\equiv
\int^\infty_0e^{i\Phi(x,y)t}\chi_1(x)a(x,y,t)dt\ \ \mbox{on $U\times U$}. 
\end{equation}
From \eqref{e-gue170227yII}, we can repeat the proof of 
Theorem 5.2 in~\cite{HsiaoHuang17} and deduce that 
$\Phi$ and $\Phi_3$ are equivalent on $\mU$. The theorem follows. 
\end{proof}

The following two theorems follow from 
Theorem~\ref{t-gue170226cw}, \eqref{e-gue170226II}, 
\eqref{e-gue170226bI}, the proof of \eqref{e-gue170227yII}, 
the complex stationary phase 
formula of Melin-Sj\"ostrand~\cite{MS74}
and some straightforward computation. We omit the details. 

\begin{theorem}\label{t-gue170301w}
Consider the Fourier integral operators
\[\begin{split}
&A(x,y)=\int^\infty_0e^{i\Phi(x,y)t}a(x,y,t)dt,\ \ 
B(x,y)=\int^\infty_0e^{i\Phi(x,y)t}b(x,y,t)dt,
\end{split}\]
with symbols 
$a(x,y,t)\in S^{k}_{{\rm cl\,}}(\mU\times\mU\times\mathbb{R}_+)$ and
$b(x,y,t)\in S^{\ell}_{{\rm cl\,}}(\mU\times\mU\times\mathbb{R}_+)$.
Consider 
$\chi(x'')\in\cC^\infty_0(\widehat\Omega_2\times\widehat\Omega_3)$. 
Then, we have 
\[\begin{split}
&\int A(x,u)\chi(u'')B(u,y)dv(u'')\equiv
\int^\infty_0e^{i\Phi(x,y)t}c(x,y,t)dt\ \ 
\mbox{on $\mU\times\mU$},
\end{split}\]
with $c(x,y,t)\in S^{k+\ell-(n-\frac{d}{2})}_{{\rm cl\,}}
(\mU\times\mU\times\mathbb{R}_+)$.
For $x\in Y\cap U$ we have
\begin{equation}\label{e-gue170301u}
c_0(x,x)=2^{-n-\frac{d}{2}+1}\pi^{n-\frac{d}{2}+1}
\abs{\det\cL_{x}}^{-1}
\abs{\det \mR_x}^{\frac{1}{2}}a_0(x,x)b_0(x,x)\chi(x''),
\end{equation}
where $\det \mR_x$ is the determinant $\mR_x$ cf.\ \eqref{e:detrx}.
Moreover, if there are $N_1, N_2\in\mathbb N^*$, $C>0$, such that 
for all $x_0\in Y\cap\mU$,
\[\abs{a_0(x,y)}\leq C\abs{(x,y)-(x_0,x_0)}^{N_1},\quad
\abs{b_0(x,y)}\leq C\abs{(x,y)-(x_0,x_0)}^{N_2},\]
then there exists $\widehat C>0$ such that for all $x_0\in Y\cap\mU$
\begin{equation}\label{e-gue170301uI}
\abs{c_0(x,y)}\leq\widehat C\abs{(x,y)-(x_0,x_0)}^{N_1+N_2}. 
\end{equation}
\end{theorem} 

\begin{theorem}\label{t-gue170301wI}
Consider the Fourier integral operators 
\[\begin{split}
&\mathcal{A}(x,\underline{y}'')=\int^\infty_0e^{i\Phi(x,\underline{y}'')t}
\alpha(x,\underline{y}'',t)dt,
\ \ \mathcal{B}(\underline{x}'',y)=\int^\infty_0e^{i\Phi(\underline{x}'',y)t}
\beta(\underline{x}'',y,t)dt,
\end{split}\]
with symbols 
$\alpha(x,\underline{y}'',t)\in S^{k}_{{\rm cl\,}}
(\mU\times\Omega_3\times\mathbb{R}_+)$
and $\beta(\underline{x}'',y,t)\in S^{\ell}_{{\rm cl\,}}
(\Omega_3\times\mU\times\mathbb{R}_+)$.
Let $\chi_1(\underline{x}'')\in\cC^\infty_0(\Omega_3)$. Then, we have 
\[\begin{split}
&\int\mathcal{A}(x,\Td u'')\chi_1(\Td u'')
\mathcal{B}(\Td u'',y)dv(\Td u'')\equiv\int^\infty_0e^{i\Phi(x,y)t}
\gamma(x,y,t)dt\ \ \mbox{on $\mU\times\mU$},
\end{split}\]
with $\gamma(x,y,t)\in S^{k+\ell-(n-d)}_{{\rm cl\,}}
(\mU\times\mU\times\mathbb{R}_+)$
where
\begin{equation}\label{e-gue170301ua}
\gamma_0(x,x)=2^{-n+1}\pi^{n-d+1}
\abs{\det\cL_{x}}^{-1}\abs{\det \mR_x}\alpha_0(x,\underline{x}'')
\beta_0(\underline{x}'',x)\chi_1(\underline{x}''),\ \  x\in Y\cap U.
\end{equation}
Moreover, if there are $N_1, N_2\in\mathbb N^*$, $C>0$,
such that for all $x_0\in Y\cap\mU$ we have
\[\abs{\alpha_0(x,\underline{y}'')}
\leq C\abs{(x,\underline{y}'')-(x_0,x_0)}^{N_1},\:\:  
\abs{\beta_0(x,\underline{y}'')}
\leq C\abs{(x,\underline{y}'')-(x_0,x_0)}^{N_2},
\] 
then there exists $\widehat C>0$ such that for all $x_0\in Y\cap\mU$,
\begin{equation}\label{e-gue170301uaI}
\abs{\gamma_0(x,y)}\leq\widehat C\abs{(x,y)-(x_0,x_0)}^{N_1+N_2}.
\end{equation}
\end{theorem} 

We introduce next  the following notion.

\begin{definition}\label{d-gue180514}
Let $H: \cC^\infty(X)\To \cC^\infty(X)$ be a continuous operator 
with distribution kernel $H(x,y)\in \cD'(X\times X)$ and
$k\in\mathbb R$, $\ell\in\N$.

(i) We say that $H$ is a \emph{complex Fourier integral operator 
of $G$-Szeg\H{o} 
type of leading order $(k,\ell)$}, if for every open set $D$ of $X$ 
with $D\cap Y=\emptyset$, 
\begin{equation}\label{e-gue180514}
\mbox{$\chi H$ and $H\chi$ are smoothing operators on $X$, 
for every $\chi\in\cC^\infty_0(D)$}
\end{equation}
and for every $p\in Y$ and any open coordinate neighborhood 
$(U,x=(x_1,\ldots,x_{2n+1}))$ of $p$, we have 
\begin{equation}\label{e-gue180514I}
H(x, y)\equiv\int^{\infty}_{0}e^{i\Phi(x, y)t}a(x, y, t)dt\ \ 
\mbox{on $U\times U$}
\end{equation}
with $\Phi(x,y)\in\cC^\infty(U\times U)$ as in 
Theorem~\ref{t-gue180505I}, and 
\begin{equation}  \label{e-gue180514II}
a(x, y, t)\in S^{k}_{{\rm cl\,}}(U\times U\times\mathbb{R}_+)
\end{equation}
and under the notation \eqref{e-fal}, 
\begin{equation}\label{e0gue180514III}
\frac{\pr^{\abs{\alpha}+\abs{\beta}}a_0(x,y)}{\pr x^\alpha\pr y^\beta}
\Big\vert_{x=y\in Y}=0, 
\quad\text{ for }\alpha, \beta\in\N^{2n+1}, \abs{\alpha}
+\abs{\beta}\leq \ell-1.
\end{equation}
(ii) We say that $H$ is a \emph{complex 
Fourier integral operator of G-Szeg\H{o} type of order $(k,\ell)$, 
$k\in\mathbb R$, $\ell\in\N$}, 
if \eqref{e-gue180514}, \eqref{e-gue180514I}, \eqref{e-gue180514II} 
hold and  
there is a 
$r(x,y,t)\in S^{-\infty}_{{\rm cl\,}}(U\times U\times\mathbb{R}_+)$ 
such that
\begin{equation}\label{e0gue180514a}
\frac{\pr^{\abs{\alpha}
+\abs{\beta}}\bigl(a(x,y,t)-r(x,y,t)\bigr)}{\pr x^\alpha\pr y^\beta}
\Big\vert_{x=y\in Y}=0, 
\quad\text{ for }\alpha, \beta\in\N^{2n+1}, \abs{\alpha}
+\abs{\beta}\leq \ell-1.
\end{equation}

 Let $G_{k,\ell}(X)$ denote the space of 
all complex Fourier integral operators of $G$-Szeg\H{o} 
type of leading order $(k,\ell)$ and let $\widehat G_{k,\ell}(X)$ 
denote the space of all complex Fourier integral operators of 
$G$-Szeg\H{o} type of order $(k,\ell)$.
\end{definition}

In Definition~\ref{d-gue180514}, the $G$ from the  terminology 
$G$-Szeg\H{o} type comes from our group $G$. 
Indeed, Definition~\ref{d-gue180514} depends on the set $Y$. 

Let us explain briefly the role of the spaces $G_{k,\ell}(X)$ 
and $\widehat G_{k,\ell}(X)$. Our goal is to study distribution kernel 
of $\sigma^{\,*}\sigma$. In Theorem~\ref{t-gue170305a} 
below, we will show that $C_0\sigma^{\,*}\sigma$ is
of the same type as $S_G$ and with the same leading term, where $C_0$
is a constant. In the terminology introduced in
Definition~\ref{d-gue180514}, 
$C_0\sigma^{\,*}\sigma-S_G\in G_{n-\frac{d}{2},1}(X)$. 
To prove our main result, we need to show that 
$C_0\sigma^{\,*}\sigma-S_G$ is "microlocally small", 
and it suffices to prove that elements $H\in G_{n-\frac{d}{2},1}(X)$
have good regularity properties (see \eqref{e-gue180514mpII}). 

Let $H=H_0+H_1$, where $H_0\in G_{n-\frac{d}{2},1}(X)$
is the leading term of $H$ and $H_1\in G_{n-\frac{d}{2}-1,0}(X)$
is the lower order term of $H$. By using calculus of
complex Fourier integral operators, we can show that when we compose
$H_1$ with itself, the order of the composition will decrease. 
More precisely, $H^N_1\in G_{n-\frac{d}{2}-N,0}(X)$, for every 
$N\in\mathbb N^*$.  Hence, for large $N$,  $H^N_1$ has good regularity 
properties and hence $H_1$ itself has good regularity properties. 

In order to handle $H_0$ we observe that when we compose 
$H_0$ with itself the order of the composition will not decrease, 
that is, $H^N_0$ is still in $G_{n-\frac{d}{2},1}(X)$ for every 
$N\in\mathbb N^*$.  To get good regularity properties, 
we need the space $\widehat G_{k,\ell}(X)$. 
Note that the space $\widehat G_{k,\ell}(X)$
is the subspace of $G_{k,\ell}(X)$
whose elements have \emph{full symbols} vanish to order $\ell$  
at $\diag(Y\times Y)$.
The key observation is that the leading symbol of $H^N_0$ 
vanishes to order $N$ at $\diag(Y\times Y)$. 
We write $H^N_0=R_{0,N}+R_{1,N}$, where 
$R_{1,N}\in G_{n-\frac{d}{2}-1,0}(X)$ is the lower order term of 
 $H^N_0$ and 
$R_{0,N}\in\widehat G_{n-\frac{d}{2},N}(X)$ is the  the leading term 
of  $H^N_0$. Since the \emph{full symbol} of $R_{0,N}$
vanishes to order 
$N$ at $\diag(Y\times Y)$, even the order of $R_{0,N}$ is high, 
$R_{0,N}$ still has good regularity properties if $N$ is large. Note that for
an element $A\in G_{k,\ell}(X)$, only the leading symbol of 
$A$ vanishes to order $\ell$ at $\diag(Y\times Y)$, 
even $\ell$ is large, we still do not have good regularity property 
for $A$ in general. That's why we need the space $\widehat G_{k,\ell}(X)$.

From Theorem~\ref{t-gue170301w}, we deduce the following. 

\begin{theorem}\label{t-gue180514}
Let $H_1\in G_{k_1,\ell_1}(X)$, $H_2\in G_{k_2,\ell_2}(X)$, 
where $k_1, \ell_1, k_2, \ell_2\in\mathbb{R}$. Then, 
\[H_1\circ H_2\in G_{k_1+k_2-(n-\frac{d}{2}), \ell_1+\ell_2}(X).\]
\end{theorem}
 Recall that $\norm{\cdot}_s$ denotes the standard 
Sobolev norm on $X$ of order $s$. 
\begin{theorem}\label{t-gue180514I}
Let $H\in G_{k,0}(X)$ with $k\leq n-\frac{d}{2}-1$. 
Then, there exists $C>0$ such that  for any 
$u\in\cC^\infty(X)$, 
\begin{equation}\label{e-gue180514aI}
\norm{Hu}\leq C\norm{u}.
\end{equation}
Moreover, for every $s\in\mathbb N^*$, there exist $N_s\in\mathbb N^*$ 
and $C_s>0$ such that  for any $u\in\cC^\infty(X)$, 
\begin{equation}\label{e-gue180514aII}
\norm{H^{N_s}u}_s\leq C_s\norm{u}.
\end{equation}
\end{theorem}
\begin{proof}
Fix $s\in\mathbb N^*$. 
By Theorem~\ref{t-gue180514}, 
for any $N_s\in\N^*$ we have
$H^{N_s}\in G_{N_sk-(N_s-1)(n-\frac{d}{2}),0}(X)$. 
Taking now $N_s\in\N^*$ such that $N_sk-(N_s-1)(n-\frac{d}{2})<-s-2$, 
it is easy to see that $H^{N_s}(x,y)\in\cC^s(X\times X)$ and 
\eqref{e-gue180514aII} follows then. 

We now prove \eqref{e-gue180514aI}. We claim that for every 
$\ell\in\mathbb N^{*}$, we have  for any $u\in\cC^\infty(X)$, 
\begin{equation}\label{e-gue180514m}
\norm{Hu}^2\leq\norm{(H^*H)^{2^\ell}u}^{2^{-\ell}}
\norm{u}^{2-2^{-\ell}},
\end{equation}
where $H^*$ is the adjoint of $H$. We have for any $u\in\cC^\infty(X)$,
\begin{equation}\label{e-gue181220}
\norm{Hu}^2=(\,Hu\,,\,Hu\,)=(\,H^*Hu\,,\,u)\leq\norm{H^*Hu}\norm{u}
\end{equation}
and
\begin{equation}\label{e-gue181220I}
\norm{H^*Hu}^2=(\,H^*Hu\,,\,H^*Hu\,)=(\,(H^*H)^2u\,,\,u)
\leq\norm{(H^*H)^2u}\norm{u}.
\end{equation}
We prove \eqref{e-gue180514m} by induction on $\ell$.
From \eqref{e-gue181220} and 
\eqref{e-gue181220I}, we get \eqref{e-gue180514m} for $\ell=1$. 
Suppose that \eqref{e-gue180514m} holds for $\ell\in\N^*$. 
We have for every $u\in\cC^\infty(X)$,
\begin{equation}\label{e-gue181220II}
\begin{split}
\norm{(H^*H)^{2^{\ell}}u}^2&=(\,(H^*H)^{2^{\ell}}u,
(H^*H)^{2^{\ell}}u\,)\\
&=(\,(H^*H)^{2^{\ell+1}}u\,,\,u)
\leq\norm{(H^*H)^{2^{\ell+1}}u}\norm{u},
\end{split}
\end{equation}
From the induction hypothesis 
and \eqref{e-gue181220II}, we get  \eqref{e-gue180514m} 
for $\ell+1$. 
 
It is obvious that $H^*\in G_{k,0}$ 
and hence $H^*H\in G_{2k-(n-\frac{d}{2}),0}$. 
From this observation and \eqref{e-gue180514aII}, 
we deduce that for $\ell$ large, there exists $C>0$ such that 
for any $u\in\cC^\infty(X)$,
\begin{equation}\label{e-gue180514mI}
\norm{(H^*H)^{2^\ell}u}\leq C\norm{u}\  .
\end{equation}
From \eqref{e-gue180514m} and \eqref{e-gue180514mI}, 
we get \eqref{e-gue180514aI}. 
\end{proof}

\begin{lemma}\label{l-gue180514m}
Let $H\in\widehat G_{k,2\ell}(X)$. If $k-\ell\leq -s-2$, 
for some $s\in\N$, then $H(x,y)\in\cC^s(X\times X)$. 
\end{lemma}

\begin{proof}
By using \eqref{e-gue180305yIII}, we can integrate 
by parts with respect to $t$ several times and deduce that 
$H(x,y)\in\cC^s(X\times X)$. The calculation is elementary 
and straightforward and we omit the details. 
\end{proof}

In the proof of our main result, we need the following.

\begin{theorem}\label{t-gue180514mp}
Let $H\in G_{n-\frac{d}{2},1}(X)$.  Then there exists $C>0$
such that   for any $u\in\cC^\infty(X)$, 
\begin{equation}\label{e-gue180514mpI}
\norm{Hu}\leq C\norm{u}. 
\end{equation}
Moreover, for every $s\in\mathbb N^*$, there exist $N_s\in\mathbb N^*$ 
and $C_s>0$ such that   for any $u\in\cC^\infty(X)$,
\begin{equation}\label{e-gue180514mpII}
\norm{H^{N_s}u}_s\leq C_s\norm{u}.
\end{equation}
\end{theorem}
\begin{proof}
From Theorem~\ref{t-gue170301w}, we see that for every 
$N\in\mathbb N^*$, we have 
\begin{equation}\label{e-gue180515}
\begin{split}
&H^N=H_{1,N}+H_{2,N},\\
&H_{1,N}\in\widehat G_{n-\frac{d}{2},N}(X),\ \ H_{2,N}
\in G_{n-\frac{d}{2}-1,0}(X).
\end{split}
\end{equation}
Fix $s\in\mathbb N^*$. 
Due to Lemma~\ref{l-gue180514m} there exists $N\gg1$ 
such that $H_{1,N}(x,y)\in\cC^s(X\times X)$ and for every 
$j\in\mathbb N^*$, there exists $C_j>0$ such that
for every $u\in\cC^\infty(X)$ we have
\begin{equation}\label{e-gue180515I}
\norm{H^j_{1,N}u}_s\leq C_j\norm{u}.
\end{equation}
Since $H_{2,N}\in G_{n-\frac{d}{2}-1,0}(X)$, 
Theorem~\ref{t-gue180514I} shows that
there exist $K_s\in\mathbb N^{*}$ and $\widehat C_s>0$ so that 
\begin{equation}\label{e-gue180515II}
\norm{H^{K_s}_{2,N}u}_s\leq\widehat C_s\norm{u},\ \ 
\text{ for }u\in\cC^\infty(X).
\end{equation}
We have 
\begin{equation}\label{e-gue180515III}
H^{NK_s}=(H_{1,N}+H_{2,N})^{K_s}
=\sum^{K_s}_{j=0}H^j_{1,N}H^{K_s-j}_{2,N}.
\end{equation}
From \eqref{e-gue180514aI}, 
\eqref{e-gue180515I}, \eqref{e-gue180515II} and  \eqref{e-gue180515III}, 
we get \eqref{e-gue180514mpII} with $N_s=NK_s$. 
From \eqref{e-gue180514mpII}, we can repeat the proof of 
\eqref{e-gue180514aI} and get \eqref{e-gue180514mpI}. 
\end{proof}

Let $H\in G_{n-\frac{d}{2},1}(X)$. From \eqref{e-gue180514mpI}, 
we can extend $I-H$ to a bounded operator in 
\[I-H: L^2(X)\To L^2(X)\]
by density.
In the proof of our main result, we need the following.
\begin{theorem}\label{t-gue180514mpq}
Let $H\in G_{n-\frac{d}{2},1}(X)$ and extend $I-H$ to 
a bounded operator in 
\[I-H: L^2(X)\To L^2(X)\]
by density. Then ${\Ker\,}(I-H)$ is a finite dimensional 
subspace of $\cC^\infty(X)$ and there exists $C>0$ such that 
\begin{equation}\label{e-gue180515q}
\norm{(I-H)u}\geq C\norm{u},\ \  
\text{ for any } u\in L^2(X),\ \ u\perp{\Ker\,}(I-H).
\end{equation}
\end{theorem}
\begin{proof}
Fix $s\in\mathbb N^*$. 
Theorem~\ref{t-gue180514mp} shows that we can extend 
$H^{N_s}$ to a bounded operator 
\begin{equation}\label{e-gue180515rI}
H^{N_s}: L^2(X)\To H^s(X),
\end{equation}
 by density. Now, let $u\in{\Ker\,}(I-H)$. Then, 
 \begin{equation}\label{e-gue180515t}
 (I-H^{N_s})u=(I+H+\ldots+H^{N_s-1})(I-H)u=0
 \end{equation}
 and hence $u=H^{N_s}u\in H^s(X)$. Since $s$ is arbitrary, 
 we deduce that $u\in\cC^\infty(X)$. 
 Moreover, from \eqref{e-gue180515t}, 
 we can apply Rellich's theorem and conclude that ${\Ker\,}(I-H)$ 
 is a finite dimensional subspace of $\cC^\infty(X)$. 
 Since the argument is standard, we omit the details. 
 
 We now prove \eqref{e-gue180515q}. Assume that \eqref{e-gue180515q}
 is not true.  For every $j\in\N^*$ we can find $u_j\in L^2(X)$ 
 with $u_j\perp{\Ker\,}(I-H)$ and 
$\norm{u_j}=1$ such that 
\begin{equation}\label{e-gue180515tI}
\norm{(I-H)u_j}\leq\frac{1}{j}\,\cdot  
\end{equation}
Put $v_j:=(I-H)u_j$. 
We have for any $j\in\N^*$,
\begin{equation}\label{e-gue180515tII}
(I-H^{N_s})u_j=(I+H+\ldots+H^{N_s-1})v_j \ . 
\end{equation}
From \eqref{e-gue180514mpI} and \eqref{e-gue180515tI} 
we see that there exists $C>0$ such that 
\begin{equation}\label{e-gue180515tIII}
\norm{(I+H+\ldots+H^{N_s-1})v_j}\leq\frac{C}{j},\ \  j\in\N^*. 
\end{equation}
From \eqref{e-gue180515rI} and since $\|u_j\|=1$ we conclude that 
there exists $\widehat C>0$ such that for any $j\in\N^*$,
\begin{equation}\label{e-gue180515y}
\norm{H^{N_s}u_j}_s\leq\widehat C. 
\end{equation}
By Rellich's theorem, we can find a subsequence 
$H^{N_s}u_{j_k}$, $1\leq j_1<j_2<\ldots$ , such that 
$H^{N_s}u_{j_k}$ converges to some $u$ in $L^2(X)$ 
 as $k\To\infty $.  From this observation, \eqref{e-gue180515tII} and 
 \eqref{e-gue180515tIII}, we deduce that $u_{j_k}$ converges to 
 $u$ in $L^2(X)$ with $\norm{u}=1$ as $k\To\infty $. 
 By \eqref{e-gue180515tI}, we get $u\in{\Ker\,}(I-H)$. 
 Since $u_{j_k}\perp{\Ker\,}(I-H)$, for every $k$, 
 we have $u\perp{\Ker\,}(I-H)$. We get a contradiction and 
 \eqref{e-gue180515q} follows. 
\end{proof}

\subsection{The distribution kernels of $\sigma$ and 
$\sigma^*\sigma$; proof of Theorem \ref{t-gue180428zy}}
\label{s-gue180512}

We are now ready to study the  distribution kernels of 
$\sigma$ and $\sigma\sigma^*$ 
  in \eqref{e-gue180308II}.
We will use the same notations as before. 
Let $\cL_{X_{G},q}$ be the Levi form on $X_{G}$  at $q\in X_{G}$ 
induced naturally from $\cL$. The Hermitian metric $g^{\C TX}$ 
on $\C TX$ restricts to a metric on $T^{1,0}X$ wich in turn 
induces a Hermitian metric  on $T^{1,0}X_{G}$. 
Let $\mu_{d+1},\ldots,\mu_n$
be the eigenvalues of 
$\cL_{X_{G},q}$ with respect to this Hermitian metric.
We set 
\begin{equation}\label{e:ma4.65}
\det\,\cL_{X_{G},q}=\mu_{d+1}\ldots\mu_n.
\end{equation}
Recall that $\pi: Y=\mu^{-1}(0)\To X_{G}$ is the natural projection. 
Let 
\begin{equation}\label{e-gue180512qa}
S_{X_{G}}:L^2(X_{G})\To{\Ker\,}\ddbar_{b,X_{G}}=H^0_b(X_{G}),
\end{equation}
be the Szeg\H{o} projection on $X_G$ (cf.\ \eqref{e-gue180512q}). 
Since $X_{G}$ is assumed to be strictly pseudoconvex and $\ddbar_{b,X_{G}}$
has closed range in $L^2$ on $X_{G}$, $S_{X_{G}}$
is smoothing away the diagonal 
(see~\cite[Theorem 1.2]{Hsiao08}, ~\cite[Theorem 4.7]{HM14}). 
Hence, for any $x, y\in Y$ with $\pi(x)\neq\pi(y)$, there are 
open neighborhoods $U$ of $\pi(x)$ in $X_{G}$ 
and $U_1$ of $\pi(y)$ in $X_{G}$ such that for all 
$\widehat\chi\in\cC^\infty_0(U)$, $\Td\chi\in\cC^\infty_0(U_1)$, 
we have 
\begin{equation}\label{e-gue170304}
\widehat\chi S_{X_{G}}\Td\chi\equiv0\ \ \mbox{on $X_{G}\times X_{G}$}. 
\end{equation}

We will use the same notations as in Section~\ref{s-gue180308I}.  
Fix $p\in Y$ and let $x=(x_1,\ldots,x_{2n+1})$ be the local coordinates 
and $\Omega_3\subset\R^{2n+1-2d}$ be an open set
as in the discussion  at the beginning of
Section~\ref{s-gue180308I}. 
From now on, we identify $\underline{x}''$ as local coordinates of $X_{G}$ 
near $q:=\pi(p)\in X_{G}$ and we identify $W:=\Omega_3$ with 
an neighborhood of $\pi(p)$ in $X_{G}$. In view of 
Theorems~\ref{t-gue180505},~\ref{t-gue161110g}, 
we have
\begin{equation}\label{e-gue170304I}
\begin{split}
&S_{X_{G}}(\underline{x}'',\underline{y}'')\equiv 
\int^{\infty }_0e^{i\varphi(\underline{x}'',\underline{y}'')t}
\beta(\underline{x}'',\underline{y}'',t)dt\ \
\mbox{on $W\times W$},
\end{split}
\end{equation}
where $\beta(\underline{x}'',\underline{y}'',t)\in S^{n-d}_{{\rm cl\,}}
(W\times W\times\mathbb{R}_+)$ with
\begin{equation}\label{e-gue170304II}
\beta_0(\underline{x}'',\underline{x}'') =\frac{1}{2}\pi^{-(n-d)-1}
\abs{\det\,\cL_{X_{G},\underline{x}''}},\ \ \underline{x}''\in W,
\end{equation}
and $\varphi(\underline{x}'',\underline{y}'')\in\cC^\infty(W\times W)$
with
\begin{equation}\label{e-gue170304III}
\begin{split} 
&d_{\underline{x}''}\varphi(\underline{x}'',\underline{x}'')
=-d_{\underline{y}''} \phi(\underline{x}'',\underline{x}'')
=-\lambda(\underline{x}'')\omega_{0,G}(\underline{x}''),\ \ 
\lambda(\underline{x}'')>0,\\
&{\rm Im\,}\varphi(\underline{x}'',\underline{y}'')\geq 
c\sum^{2n}_{j=2d+1}\abs{x_j-y_j}^2,\ \ \mbox{for some $c>0$}, \\
&\mbox{$\ddbar_{b,\underline{x}''}\varphi(\underline{x}'',\underline{y}'')$ 
vanishes to infinite order at $\underline{x}''=\underline{y}''$},\\
&\varphi(\underline{x}'', \underline{y}'')=-x_{2n+1}+y_{2n+1}+
i\sum^{n}_{j=d+1}\abs{\mu_j}\abs{z_j-w_j}^2 \\
&\quad+\sum^{n}_{j=d+1}i\mu_j(\ol z_jw_j-z_j\ol w_j)+
\sum^n_{j=d+1}\frac{1}{2}(b_{2j-1}-ib_{2j})
(-z_jx_{2n+1}+w_jy_{2n+1})\\
&\quad+\sum^n_{j=d+1}\frac{1}{2}(b_{2j-1}+
ib_{2j})(-\ol z_jx_{2n+1}+\ol w_jy_{2n+1})\\
&\quad+(x_{2n+1}-y_{2n+1})r(\underline{x}'', \underline{y}'')
+O(\abs{(\underline{x}'', \underline{y}'')}^3),
\end{split}
\end{equation}
where $r(\underline{x}'', \underline{y}'')\in\cC^\infty(W\times W)$, 
$r(0,0)=0$, 
$z_j=x_{2j-1}+ix_{2j}$, $j=d+1,\ldots,n$.

It is not difficult to see that the phase function 
$\Phi(\underline{x}'',\underline{y}'')$ in Theorem~\ref{t-gue180505I}
satisfies \eqref{e-gue170304III}. Hence,  there is a function 
$h\in\cC^\infty(W\times W)$ 
with $h(\underline{x}'',\underline{x}'')\neq0$ for any 
$\underline{x}''\in W$, such that 
$\varphi(\underline{x}'',\underline{y}'')-h(\underline{x}'',
\underline{y}'')\Phi(\underline{x}'',\underline{y}'')$ 
vanishes to infinite order at $\underline{x}''=\underline{y}''$ 
(see Theorem~\ref{t-gue161110g}). 
We can replace the phase $\varphi(\underline{x}'',\underline{y}'')$ 
by $\Phi(\underline{x}'',\underline{y}'')$ 
and we have 
\begin{equation}\label{e-gue170304ry}
S_{X_{G}}(\underline{x}'',\underline{y}'')\equiv
\int^{\infty }_0e^{i\Phi(\underline{x}'',\underline{y}'')t}
\beta(\underline{x}'',\underline{y}'',t)dt\ \ 
\text{on $W\times W$}.
\end{equation}
\begin{theorem}\label{t-gue170304ry}
If $y\notin Y$, 
then for any open neightborhood $D$ of $y$ with 
$\ol D\cap Y=\emptyset$, we have
\begin{equation}\label{e-gue170304ryI}
\sigma\equiv0\ \ \mbox{on $X_{G}\times D$.}
\end{equation}
Let $x_0, y_0\in Y$.  If $\pi(x_0)\neq\pi(y_0)$, 
then there are open neighborhoods $U_G$ of $\pi(x_0)$ in $X_{G}$ and 
$U_1$ of $y_0$ in $X$ such that 
\begin{equation}\label{e-gue170304ryII}
\sigma\equiv0\ \ \mbox{on $U_G\times U_1$}. 
\end{equation}
Let $p\in Y$  and let $x=(x_1,\ldots,x_{2n+1})$ be the local coordinates,
and $U$ an open neighborhood of $p$
as  at the discussion in the beginning 
of Section~\ref{s-gue180308I}. Then under the notations 
at \eqref{e-fal}, \eqref{e-gue170304I}, there exists
$\alpha(\underline{x}'',y'',t)\in S^{n-\frac{3}{4}d}_{{\rm cl\,}}
(W\times U\times\mathbb{R}_+)$ such that
\begin{equation}\label{e-gue170304ryIII}
\begin{split}
&\sigma(\underline{x}'',y)\equiv
\int^\infty_0e^{i\Phi(\underline{x}'',y'')t}\alpha(\underline{x}'',y'',t)dt\ \ 
\mbox{on $W\times U$},  \\
&\alpha_0(\underline{x}'',\underline{x}'')=2^{-n+2d-1}
\pi^{\frac{d}{2}-n-1} \frac{1}{\sqrt{V_{{\rm eff\,}}(\underline{x}'')}}
\abs{\det\,\cL_{\underline{x}''}}
\abs{\det\,R_{\underline{x}''}}^{-\frac{3}{4}},\ \  \text{ for }
\underline{x}''\in W.
\end{split}\end{equation}
\end{theorem} 
\begin{proof}
By Theorem~\ref{t-gue181021}, $S_{G}$ is smoothing away $Y$,
which implies \eqref{e-gue170304ryI}. Let $x_0, y_0\in Y$.  
Assume that $\pi(x_0)\neq\pi(y_0)$. Let $V_1$ be a $G$-invariant
neighborhood of $y_0$ and $x_0\notin V_1$. 
We have 
\[S_{G}(x,y)=\int_GS_{\leq\lambda_0}(x,g\circ y)d\mu(g),\]
where $\lambda_0>0$ is as in 
Theorem~\ref{t-gue181021a}. 
Since $S_{\leq\lambda_0}$ is smoothing away the diagonal near $Y$, 
for any neighborhood $U_1$ of $x_0$ in $X$ with 
$\ol U_1\cap V_1=\emptyset$, we have 
\begin{equation}\label{e-gue170305}
S_G\equiv0\ \ \mbox{on $U_1\times V_1$}. 
\end{equation}
Let $\widehat V_G:=\set{\pi(y)\in X_{G};\, y\in V_1\cap Y}$. 
By \eqref{e-gue170304} there is an open set 
$\widehat U_G$ of $\pi(x)$ in $X_{G}$ such that 
\begin{equation}\label{e-gue170305I}
S_{X_{G}}\equiv0\ \ \mbox{on $\widehat U_G\times\widehat V_G$}. 
\end{equation}
The definition \eqref{e-gue180308II} of $\sigma$ and relations
\eqref{e-gue170305} and \eqref{e-gue170305I} yield 
\eqref{e-gue170304ryII}. 

Fix $u=(u_1,\ldots,u_{2n+1})\in Y\cap U$. 
In view of \eqref{e-gue170304ryI} 
and \eqref{e-gue170304ryII}, we only need to show 
that \eqref{e-gue170304ryIII} 
holds near $u$. We may assume that 
$u=(0,\ldots,0,u_{2d+1},\ldots,u_{2n+1})=\underline{u}''$. 
Let $U_2$ be a small neighborhood of $u$. 
Let $\chi(\underline{x}'')\in\cC^\infty_0(\Omega_3)$. 
By \eqref{e-gue170227c} we can extend $\chi(\underline{x}'')$ to
$G\cdot\Omega_3$ 
by $\chi(g\circ \underline{x}''):=\chi(\underline{x}'')$, for every $g\in G$. 
Assume that $\chi=1$ on some neighborhood of $\overline{U}_2$. 
Let $\chi_1\in\cC^\infty_0(X_{G})$ with $\chi_1=1$ 
on some neighborhood of $\pi(U_2\cap Y)\subset X_G$ and 
 $\supp\chi_1\subset\pi(\set{x\in Y;\, \chi(x)=1})$. 
We have by \eqref{e-gue180308II},
\begin{equation}\label{e-gue170227bq}
\begin{split}
\chi_1\sigma&=\chi_1S_{X_{G}}\circ E\circ\iota_G
\circ f_G\circ\iota^*\circ S_{G}\\
&=\chi_1S_{X_{G}}\circ E\circ\iota_G\circ 
f_G\circ \iota^*\circ \chi S_{G}\\
&\quad+\chi_1S_{X_{G}}\circ E\circ\iota_G\circ f_G
\circ \iota^*\circ(1-\chi)S_{G}.
\end{split}
\end{equation}
If $u\in Y$ but $u\notin\set{x\in X;\, \chi(x)=1}$. 
 Then $\pi(u)\notin\supp\chi_1$. 
From this observation and \eqref{e-gue170304}, we get 
\begin{equation}\label{e-gue170227bIIq}
\chi_1S_{X_{G}}\circ E\circ\iota_G\circ f_G\circ
\iota^*\circ(1-\chi)S_{G}\equiv0\ \ \mbox{on $X_{G}\times X$}. 
\end{equation}
From \eqref{e-gue170227bq} and \eqref{e-gue170227bIIq}, we get 
\begin{equation}\label{e-gue170227bIIIq}
\chi_1\sigma\equiv\chi_1S_{X_{G}}\circ E\circ
\iota_G\circ f_G\circ \iota^*\circ \chi S_{G}\ \ 
\mbox{on $X_{G}\times X$}. 
\end{equation}
From Theorem~\ref{t-gue180505I}, \eqref{e-gue170304ry} 
and \eqref{e-gue170227bIIIq}, we can check that on $W\times U$, 
\begin{equation}\label{e-gue170227yq}
\begin{split}
&\chi_1(\underline{x}'')\sigma(\underline{x}'',y)\\
&\equiv\int^\infty_0\int^\infty_0\int e^{i\Phi(\underline{x}'',\Td v'')t}
\beta(\underline{x}'', \Td v'',t)E\circ\Bigr(\chi(\Td v'')
f_G(\Td v'')e^{i\Phi(\Td v'',y)s}a(\Td v'', y,s)\Bigr)dv(\Td v'')dsdt. 
\end{split}
\end{equation}
From  Theorem~\ref{t-gue170301wI}, \eqref{e-gue170227yq}
and a straightforward calculation we see
that \eqref{e-gue170304ryIII} 
holds near $u$. The theorem follows. 
\end{proof} 

Let $\sigma^*: \cC^\infty(X_{G})\To \cD'(X)$ be 
the formal adjoint of $\sigma$. 
From Theorem~\ref{t-gue170304ry} we deduce that 
\begin{equation}\label{e:ma4.73}
\sigma^*: \cC^\infty(X_{G})\To H^0_{b}(X)^G
\cap\cC^\infty(X)^G.
\end{equation}
Let  
\begin{equation}\label{e-gue170305a}
\begin{split}
&A_{1}:=\sigma^*\sigma: \cC^\infty(X)\To H^0_{b}(X)^G,\\
&A_{2}:=\sigma\sigma^*: \cC^\infty(X_{G})
\To H^{0}_{b}(X_{G}).
\end{split}
\end{equation}
Let $A_{1}(x,y)$ and $A_{2}(x,y)$ be the distribution kernels of 
$A_{1}$ and $A_{2}$, respectively. In view of 
Theorems~\ref{t-gue170301w},~\ref{t-gue170301wI}
we can repeat the proof of 
Theorem~\ref{t-gue170304ry} with minor changes 
and deduce the following two theorems.
\begin{theorem}\label{t-gue170305a}
With the notations used above, if $y\notin Y$, then for any neighborhood 
$D$ of $y$ with $\ol D\cap Y=\emptyset$, we have
\begin{equation}\label{e-gue170304ryIp}
A_{1}\equiv0\ \ \mbox{on $X\times D$.}
\end{equation}
Let $x, y\in Y$.  If $\pi(x)\neq\pi(y)$, then there are neighborhoods
$D_1$ of $x$ in $X$ and $D_2$ of $y$ in $X$ such that 
\begin{equation}\label{e-gue170305b}
A_{1}\equiv0\ \ \mbox{on $D_1\times D_2$}. 
\end{equation}

Let $p\in Y$  and let $x=(x_1,\ldots,x_{2n+1})$ be the 
local coordinates as in the discussion in the beginning of 
Section~\ref{s-gue180308I}. Then there exists 
an open neighborhood $U$ of $p$ and a symbol 
$a(x'',y'',t)\in S^{n-\frac{d}{2}}_{{\rm cl\,}}
(U\times U\times\mathbb{R}_+)$
such that the following holds  under the notation \eqref{e-fal}, 
\begin{equation}\label{e-gue170305bI}
A_{1}(x,y)\equiv\int^\infty_0e^{i\Phi(x'',y'')t}a(x'',y'',t)dt\ \ 
\mbox{on $U\times U$}
\end{equation}
 with
\begin{equation}\label{e-gue170305bII}
a_0(\underline{x}'',\underline{x}'')=2^{-3n+4d-1}\pi^{-n-1}
\frac{1}{V_{{\rm eff\,}}(\underline{x}'')}\abs{\det\,\cL_{\underline{x}''}}
\abs{\det\,R_{\underline{x}''}}^{-\frac{1}{2}},\ \ \mbox{ for 
$\underline{x}''\in U\cap Y$}.
\end{equation}
\end{theorem} 
\begin{theorem}\label{t-gue170305aI}
Let $x, y\in Y$. If $\pi(x)\neq\pi(y)$, then there are neighborhoods $D_G$ 
of  $\pi(x)$ and $V_G$ of $\pi(y)$ in $X_{G}$ such that 
\begin{equation}\label{e-gue170305bp}
A_{2}\equiv0\ \ \mbox{on $D_G\times V_G$}. 
\end{equation}
Let $p\in Y$  and  let $x=(x_1,\ldots,x_{2n+1})$ be the 
local coordinates as in the discussion  at the beginning of 
Section~\ref{s-gue180308I}. Then there exists 
$\widehat a(\underline{x}'',\underline{y}'',t)\in S^{n-d}_{{\rm cl\,}}
(W\times W\times\mathbb{R}_+)$ such that by using the 
notations \eqref{e-fal}, \eqref{e-gue170304I}, we have
\begin{equation}\label{e-gue170305bIp}
\begin{split}
&A_{2}(\underline{x}'',\underline{y}'')
\equiv\int^\infty_0e^{i\Phi(\underline{x}'',\underline{y}'')t}
\widehat a(\underline{x}'',\underline{y}'',t)dt\ \
\mbox{on $W\times W$},\\
\end{split}
\end{equation}
with
\begin{equation}\label{e-gue170305bIIp}
\widehat a_0(\underline{x}'',\underline{x}'')=
2^{-3n+\frac{5}{2}d-1}\pi^{-n+\frac{d}{2}-1}
\abs{\det\,\cL_{X_{G},\underline{x}''}},\ \   \mbox{for}\ \  
\underline{x}''\in W.
\end{equation}
\end{theorem}
Set 
\begin{equation}\label{e-gue200228yyd}
Q:=C_0\,\sigma^*\sigma+S_{G}:=
C_0A_{1}+S_{G}:\cC^\infty(X)\To H^0_{b}(X)^G, \quad
\text{with $C_0=2^{3(n-d)}\pi^{d/2}$.}
\end{equation}
Since $A_{1}=A_{1}S_{G}=S_GA_{1}$,
it is clear that 
\begin{equation}\label{e-gue180515p}
C_0A_{1}=S_{G}-Q=S_{G}-QS_{G}=(I-Q)S_G=S_G(I-Q)
\end{equation}
and 
\begin{equation}\label{e-gue180515pI}
Q^*=Q,
\end{equation}
where $Q^*$ is the formal adjoint of $Q$. 
From Theorems~\ref{t-gue180505I}, \ref{t-gue180514mpq} 
and \ref{t-gue170305a}, we get: 
\begin{theorem}\label{t-gue180515p}
The operator $Q$ belongs to the class $G_{n-\frac{d}{2},1}(X)$ 
and hence $I-Q$ extends by density to a bounded self-adjoint operator
$I-Q: L^2(X)\To L^2(X)$.
\end{theorem}
By Theorem~\ref{t-gue180515p} there exists $C>0$
such that for every 
$u\in H^0_b(X)^G\cap\cC^\infty(X)$, we have 
\begin{equation}\label{e-gue180515z}
(\,\sigma u,\sigma u\,)_{X_{G}}=
(\,\sigma^*\sigma u,u\,)=\frac{1}{C_0}(\,(I-Q)u,u)
\leq C\norm{u}^2.
\end{equation}
From \eqref{e-gue180515z}, we deduce: 
\begin{corollary}\label{c-gue180515p}
There exists $C>0$ such that 
\begin{equation}\label{e-gue180515zI}
(\,\sigma u,\sigma u\,)_{X_{G}}\leq C\norm{u}^2,\ \  
 \mbox{for any}\ \ 
u\in H^0_b(X)^G\cap\cC^\infty(X).
\end{equation}
Hence we can extend $\sigma$ by density to a bounded operator
\[\sigma: H^0_b(X)^G\To H^0_b(X_{G}).\] 
\end{corollary}
From now on, we consider $\sigma$ as a bounded operator 
$\sigma: H^0_b(X)^G\To H^0_b(X_{G})$.
\begin{theorem}\label{t-gue180515h}
${\Ker\,}\sigma$ is a finite 
dimensional subspace of $H^0_b(X)\cap\cC^\infty(X)$. 
\end{theorem}
\begin{proof}
From Theorem~\ref{t-gue180514mpq} we see that ${\Ker\,}(I-Q)$ 
is a finite dimensional subspace of the space 
$\cC^\infty(X)$. Note that
${\Ker\,}\sigma\subset H^0_b(X)^G\cap{\Ker\,}(I-Q)$, 
so the theorem follows. 
\end{proof}
We repeat the proof of Corollary~\ref{c-gue180515p} with 
minor changes and deduce that there exists $\widehat C>0$ such that 
\begin{equation}\label{e-gue180515h}
(\,\sigma^*v,\sigma^*v\,)
\leq\widehat C\norm{v}^2_{X_{G}},\ \  
\text{ for any }   \, v\in H^0_b(X_{G})\cap\cC^\infty(X_{G}).
\end{equation}
Therefore we can extend $\sigma^*$ by density to a bounded operator
\[\sigma^*: H^0_b(X_{G})\To H^0_b(X)^G.\]
We repeat the proof of Theorem~\ref{t-gue180515h} 
with minor changes and deduce: 
\begin{theorem}\label{t-gue180515hI}
$\Ker\sigma^*$ is a finite dimensional subspace of 
$H^0_b(X_{G})\cap\cC^\infty(X_{G})$. 
\end{theorem}
Finally, we obtain:
\begin{theorem}\label{t-gue180515hII}
${\Ker\,}\sigma$ and 
 $(\Ran\sigma)^\perp$ are finite dimensional subspaces of 
 $H^0_b(X)^G\cap\cC^\infty(X)$ and
 $H^0_b(X_{G})\cap\cC^\infty(X_{G})$, respectively.
\end{theorem}
\begin{proof}
We only need to prove that 
 $(\Ran \sigma)^\perp$ is
a finite dimensional subspace of $\cC^\infty(X_{G})$.  Note that 
$ (\Ran \sigma)^\perp\subset{\Ker\,}\sigma^*$. 
From this observation and Theorem~\ref{t-gue180515hI},
the theorem follows. 
\end{proof}
Theorem~\ref{t-gue180515hII} implies
Theorem~\ref{t-gue180428zy}. 
\section{Proof of Theorem~\ref{t-180428zy}}\label{s-gue200926yyd}

The main goal of this section is to prove Theorem~\ref{t-180428zy}. 
Recall that the Riemannian metrics $g^{TX}$ on $X$ and 
$g^{TX_G}$ on $X_G$ are given by the Convention \ref{conv_met}.
Let $\Delta^X$ and $\Delta^{X_G}$ be the (positive) Laplacians
on $(X,g^{TX})$ and $(X_G,g^{TX_G})$, respectively.
For $s\in\R$ we consider the classical pseudodifferential operators
 of order $s$ on $X$ and $X_G$, respectively,
\begin{equation}\label{eq:ma5.1} 
\Lambda_s=( 1+\Delta^X)^{s/2}\,,\:\:
\widehat\Lambda_s=(1 +\Delta^{X_G})^{s/2}\,.
\end{equation}
They are self-adjoint and positive with respect to the inner
products $(\cdot,\cdot)_X$ and $(\cdot,\cdot)_{X_G}$,
respectively.
In particular, the maps
\begin{equation}\label{eq:ma5.2} 
\Lambda_s:H^\ell(X)\to H^{\ell-s}(X)\,,\:\:
\widehat\Lambda_s:H^\ell(X_G)\to H^{\ell-s}(X_G)\,,
\end{equation}
are injective for any $\ell\in\R$. For $u,v\in H^s(X)$,
$u',v'\in H^s(X_G)$, we define the inner products
\begin{equation}\label{eq:ma5.3} 
(u,v)_s:=(\Lambda_su,\Lambda_sv)_X
\,,\:\: (u',v')_s:=(\widehat\Lambda_su,\widehat\Lambda_sv)_{X_G}\,,
\end{equation}
and let $\|\cdot\|_s$, $\|\cdot\|_{X_G,s}$ be the corresponding
norms.

Recall that the space $G_{k,\ell}(X)$ is given by 
Definition~\ref{d-gue180514} and $S_G\in G_{n-\frac{d}{2},0}(X)$. 
\begin{theorem}\label{t-gue200926yyd}
Let $A\in G_{n-\frac{d}{2}+\gamma,0}(X)$, $\gamma\in\mathbb R$. 
Then $A$ can be extended continuously to 
$A: H^s(X)\To H^{s-\gamma}(X)$ for every $s\in\mathbb R$.
\end{theorem}
\begin{proof}
For every $s\in\mathbb R$, 
put $B:=\Lambda_{s-\gamma}A\Lambda_{-s}$. From complex stationary 
phase formula of Melin-Sj\"ostrand, we 
see that $B\in G_{n-\frac{d}{2},0}(X)$. Near $Y$, write
$B(x,y)\equiv\int^{+\infty}_0e^{i\Phi(x,y)t}b(x,y,t)dt$ 
as in \eqref{e-gue180514I} and let $b_0(x,y)$ be the leading term 
of $b(x,y,t)$. 
Let $P$ be a classical pseudodifferential operator of order $0$
on $X$ with 
\begin{equation}\label{eq:ma5.4}
\sigma_P(x,\Phi_x(x,x))a_0(x,x)=b_0(x,x),
\end{equation}
for every $x\in Y$, where 
$\sigma_P$ denotes the principal symbol of $P$ and $a_0$
is the leading term of the expansion of $S_G$
(see \eqref{e-gue180305yI}). 
From the complex stationary phase formula, we have 
\begin{equation}\label{e-gue200926yydb}
B=PS_G+R,\ \ R\in G_{n-\frac{d}{2},1}(X). 
\end{equation}
From Theorem~\ref{t-gue180514mp}, \eqref{e-gue200926yydb} 
and the fact that $S_G$ is $L^2$ bounded, we deduce that 
$B=\Lambda_{s-\gamma}A\Lambda_{-s}$ 
is $L^2$ bounded. From this observation, 
we conclude that there exists $C_{1}>0$ such that for
every $u\in\cC^\infty(X)$,
\begin{align}\label{eq:ma5.5}
\norm{Au}_{s-\ell} 
= \norm{\Lambda_{s-\gamma}A\Lambda_{-s}\Lambda_s u}
\leq C_{1}\norm{\Lambda_su} = C_{1}\norm{u}_s .
\end{align}
The theorem follows. 
\end{proof}
From Theorem~\ref{t-gue200926yyd} and note that 
$S_G\in G_{n-\frac{d}{2},0}(X)$, we deduce
the following regularity property of the $G$-invariant 
Szeg\H{o} projector.

\begin{corollary}\label{c-gue200927yyd}
For every $s\in\mathbb R$, 
$S_G:H^s(X)\To H^s(X)$ is continuous, and
in particular, $H^0_b(X)^G\cap\cC^\infty(X)$ is dense in 
$H^0_b(X)^G_s$ in $H^s(X)$. 
\end{corollary} 
Due to Theorem~\ref{t-gue200211yyd} the map $\sigma_G$ 
given by \eqref{180308Im} has a well-defined extension
\[\begin{split}
\sigma_G: \cC^\infty(X)\To H^0_b(X_G)\cap\cC^\infty(X_G),\:\:
u\mapsto\iota_G\circ\iota^*\circ S_G  u\,.
\end{split}\]
\begin{theorem}\label{t-gue200926yyda}
For every $s\in\mathbb Z$ there exists $C_s>0$ such that 
\begin{equation}\label{e-gue201005ycd}
\norm{\sigma_Gu}^2_{X_G,s-\frac{d}{4}}\leq C_s\norm{u}^2_s,\ \ 
\mbox{for every $u\in\cC^\infty(X)$}.
\end{equation}
Moreover, the map \eqref{180308Im} can be extended continuously
by density to a bounded operator 
\begin{equation}\label{eq:ma5.8}
\sigma_{G,s}:=\sigma_G: H^0_b(X)^G_s
\To H^0_b(X_G)_{s-\frac{d}{4}},\ \ 
\end{equation} 
for every $s\in\mathbb R$.
\end{theorem} 
\begin{proof}
Fix $s\in\mathbb R$. Let $(\hat\Lambda_{s-\frac{d}{4}}\sigma_G)^*: 
\cC^\infty(X_G)\To\cD'(X)$ be the formal adjoint 
of $\hat\Lambda_{s-\frac{d}{4}}\sigma_G: \cC^\infty(X)
\To\cC^\infty(X_G)$. For $u\in\cC^\infty(X)$, we have 
\begin{equation}\label{e-gue200927ycd}
\norm{\sigma_Gu}_{X_G,s-\frac{d}{4}}^{2}
=(\,\hat\Lambda_{s-\frac{d}{4}}\sigma_Gu\,,\,
\hat\Lambda_{s-\frac{d}{4}}\sigma_Gu\,)_{X_G}
=(\,(\hat\Lambda_{s-\frac{d}{4}}\sigma_G)
^*(\hat\Lambda_{s-\frac{d}{4}}\sigma_G)u\,,\,u\,).
\end{equation}
We repeat the proof of Theorem~\ref{t-gue170305a} and conclude 
that 
\begin{equation}\label{e:ma5.10}
(\hat\Lambda_{s-\frac{d}{4}}\sigma_G)^*(\hat\Lambda_{s-\frac{d}{4}}
\sigma_G)\in G_{n-\frac{d}{2}+2s,0}(X).
\end{equation} 
From Theorem~\ref{t-gue200926yyd} and \eqref{e:ma5.10}, 
we deduce that there exist $C, C_1>0$ such that
\begin{equation}\label{e-gue200927ycdI}
\abs{(\,(\hat\Lambda_{s-\frac{d}{4}}
\sigma_G)^*(\hat\Lambda_{s-\frac{d}{4}}\sigma_G)u\,,\,u\,)}\leq 
C\norm{(\hat\Lambda_{s-\frac{d}{4}}\sigma_G)
^*(\hat\Lambda_{s-\frac{d}{4}}\sigma_G)u}_{-s}\norm{u}_s
\leq C_1\norm{u}^2_s,
\end{equation}
for every $u\in\cC^\infty(X)$. 
From \eqref{e-gue200927ycd} and \eqref{e-gue200927ycdI} 
we get \eqref{e-gue201005ycd}. 

By Theorem~\ref{t-gue200211yyd} and 
Corollary~\ref{c-gue200927yyd} we see that 
$H^0_b(X)^G\cap\cC^\infty(X)$ is dense in $H^0_b(X)^G_s$, 
for every $s\in\mathbb R$. From this observation and 
\eqref{e-gue201005ycd}, we deduce that for every $s\in\mathbb R$
the map \eqref{180308Im} 
can be extended continuously by density to a bounded operator 
$\sigma_{G,s}$ as in \eqref{eq:ma5.8}.
\end{proof}

 Let $f_G\in\cC^\infty(Y)^G$ be as in \eqref{e-gue180308m}. 
We identify $f_G$ with a smooth function on $X_G$. 

\begin{theorem}\label{t-gue200930yyd}
For every $s\in\mathbb R$, ${\rm Ker\,}\sigma_{G,s}$
is a finite dimensional subspace of $\cC^\infty(X)$. 
Moreover, ${\rm Ker\,}\sigma_{G,s}$ is independent of $s$. 
\end{theorem}

\begin{proof} 
Let $E$ be a classical pseudodifferential operator on $X_G$ 
with principal symbol $\sigma_E(x,\xi)=\abs{\xi}^{-\frac{d}{4}}$.
Let for every $s\in\R$, 
\begin{equation}\label{e:ma5.11}
\hat\sigma_G:=S_{X_G}\circ E\circ f_G\circ
\sigma_G: H^0_b(X)^G_s\To H^0_b(X_G)_s\,.
\end{equation}
We repeat the proof of Theorem~\ref{t-gue180515h} 
and deduce that ${\rm Ker\,}\hat\sigma_G$ is a finite dimensional
subspace of $\cC^\infty(X)$. 
Since ${\rm Ker\,}\sigma_G\subset{\rm Ker\,}\hat\sigma_G$, 
the theorem follows. 
\end{proof} 

\begin{theorem}\label{t-gue201001yyd}
With $\sigma_{G,s}$ as in \eqref{eq:ma5.8},
$(\Ran\sigma_{G,s})^\perp$ in \eqref{eq:ma1.8} is
a finite dimensional subspace of $\cC^\infty(X_G)$
for every $s\in\R$.
\end{theorem}
\begin{proof}
Fix $s\in\R$. By Corollary~\ref{c-gue200927yyd} and 
Theorem~\ref{t-gue200926yyda} we can extend $\sigma_{G,s}$
in \eqref{eq:ma5.8} to $H^s(X)$ by
\begin{equation}\label{eq:ma5.14}
\sigma_G:H^s(X)
\To H^{s-d/4}(X_G),\:\:
u\mapsto\sigma_GS_Gu. 
\end{equation}
%
%
We have 
$\hat\Lambda_{s-\frac{d}{4}}\sigma_G: \cC^\infty(X)
\To\cC^\infty(X_G)$. Let 
$(\hat\Lambda_{s-\frac{d}{4}}\sigma_G)^*: \cD'(X_G)\To\cD'(X)$ 
be the formal adjoint of $\hat\Lambda_{s-\frac{d}{4}}\sigma_G$. 
We repeat the proof 
of Theorem~\ref{t-gue170304ry} with minor changes and deduce that 
$(\hat\Lambda_{s-\frac{d}{4}}\sigma_G)^*: \cC^\infty(X_G)
\To\cC^\infty(X)$.  
Let 
\begin{equation}\label{eq:ma5.16}
F_s:=S_{X_G}f_G^2\hat\Lambda_{-2s}
\sigma_G(\hat\Lambda_{s-\frac{d}{4}}
\sigma_G)^*\hat\Lambda_{s-\frac{d}{4}}S_{X_G}: \cD'(X_G)
\To\cD'(X_G).
\end{equation} 
For any
$u\in(\Ran\sigma_{G,s})^\perp$ 
and $v\in\cC^\infty(X)$, by \eqref{eq:ma1.8}, we have 
\begin{equation}\label{eq:ma5.15}\begin{split}
\big(v,S_G(\hat\Lambda_{s-\frac{d}{4}}\sigma_G)^*
\hat\Lambda_{s-\frac{d}{4}}S_{X_G}u\big)
&=\big(\hat\Lambda_{s-\frac{d}{4}}\sigma_GS_Gv, 
\hat\Lambda_{s-\frac{d}{4}}S_{X_G}u\big)_{X_G}\\
&=\big(\sigma_Gv,u\big)_{X_G,s-\frac{d}{4}}=0.
\end{split}\end{equation} 
In view of \eqref{eq:ma5.15},     
we see that 
\begin{equation}\label{e-gue201002yydI}
(\Ran\sigma_{G,s})^\perp\subset{\rm Ker\,}F_s\cap H^0_b(X_G)_s.
\end{equation}
We repeat the proof of Theorem~\ref{t-gue170305aI} 
and conclude that $F_s=C(I-R)S_{X_G}$, where $C\neq0$ is 
a constant and $R$ is a complex Fourier integral 
operator with the same phase and order as $S_{X_G}$ and 
vanishing leading term on the diagonal. More precisely, 
in the local coordinates $\underline{x}''$ of $X_G$ defined in an open set 
$W$ of $X_G$, we have 
\begin{equation}\label{eq:ma5.17}
R(\underline{x}'',\underline{y}'')
\equiv\int^\infty_0e^{i\Phi(\underline{x}'',\underline{y}'')t}
\widehat a(\underline{x}'',\underline{y}'',t)dt, 
\end{equation}
where $\Phi$ is 
as in \eqref{e-gue170305bIp}, 
$r\in S^{n-d}_{{\rm cl\,}}(W\times W\times\mathbb R_+)$, 
$r_0(\underline{x}'',\underline{x}'')=0$ for every $\underline{x}''\in W$, 
where $r_0$ is the leading term of $r$. We repeat now the proof of 
Theorem~\ref{t-gue180514mpq} with minor changes and deduce 
that ${\rm Ker\,}(I-R)$ is a finite dimensional subspace of 
$\cC^\infty(X_G)$. From \eqref{e-gue201002yydI} we deduce that 
$(\Ran\sigma_{G,s})^\perp\subset{\rm Ker\,}(I-R)$. The theorem follows. 
\end{proof}

\begin{theorem}\label{t-gue201003yyd}
$\dim(\Ran\sigma_{G,s})^\perp$ is independent 
of $s$. 
\end{theorem}

\begin{proof}
Fix $s\in\R$ and let $u\in(\Ran\sigma_{G,s})^\perp$.
By Theorem \ref{t-gue201001yyd} we have 
$u\in\cC^\infty(X_G)\cap H^0_b(X_G)$.
We have an orthogonal decomposition
\begin{equation}\label{eq:ma5.19}
u=\sigma_{G,\frac{d}{4}}v+w,\:\:  \quad 
v\in H^0_b(X)^G_{\frac{d}{4}},\: 
w\in(\Ran\sigma_{G,\frac{d}{4}})^\perp. 
\end{equation}
We define a linear map
\begin{equation}\label{eq:ma5.18}
\gamma_s: (\Ran\sigma_{G,s})^\perp
\To (\Ran\sigma_{G,\frac{d}{4}})^\perp,\:\:
\gamma_su:=w\in(\Ran\sigma_{G,\frac{d}{4}})^\perp. 
\end{equation}
We claim that $\gamma_s$ is injective. 
Assume that $\gamma_su=0$. 
Then there exists $v\in H^0_b(X)^G_{\frac{d}{4}}$
such that $u=\sigma_{G,\frac{d}{4}}v$. 
Let $(v_j)$ be a sequence in $\cC^\infty(X)$ with 
$v_j\to v$ in $H^{\frac{d}{4}}(X)$ as $j\To\infty$. 
By Corollary~\ref{c-gue200927yyd} we have 
$S_Gv_j\to S_{G}v=v$ in $H^{\frac{d}{4}}(X)$ as $j\To\infty$. 
By Theorem~\ref{t-gue200926yyda}
we have $\sigma_{G,\frac{d}{4}}S_Gv_j\to
\sigma_{G,\frac{d}{4}}v=u$ in $L^2(X)$ as $j\To\infty$. 
Since $u\in(\Ran\sigma_{G,s})^\perp$, we have as $j\To\infty$
\[0=(\,\hat\Lambda_{s-\frac{d}{4}}\sigma_G S_{G}v_j\,,\,
\hat\Lambda_{s-\frac{d}{4}}u\,)_{X_G}
=(\,\sigma_GS_Gv_j\,,\,(\hat\Lambda_{s-\frac{d}{4}})^2u\,)_{X_G}
\To(\,u\,,(\hat\Lambda_{s-\frac{d}{4}})^2u\,)_{X_G}.
\]
Hence $(\,u\,,(\hat\Lambda_{s-\frac{d}{4}})^2u\,)_{X_G}=0$ 
and thus $\hat\Lambda_{s-\frac{d}{4}}u=0$. Since 
$\hat\Lambda_{s-\frac{d}{4}}$ is injective, 
we get $u=0$ and hence $\gamma_s$ is injective. Since $\gamma_s$
is injective, ${\rm dim\,}(\Ran\sigma_{G,s})^\perp
\leq{\rm dim\,}(\Ran\sigma_{G,\frac{d}{4}})^\perp$. 

Similarly, we can repeat the procedure above and conclude that 
${\rm dim\,}(\Ran\sigma_{G,\frac{d}{4}})^\perp
\leq{\rm dim\,}(\Ran\sigma_{G,s})^\perp$. Thus, 
${\rm dim\,}(\Ran\sigma_{G,s})^\perp
={\rm dim\,}(\Ran\sigma_{G,\frac{d}{4}})^\perp$.
The theorem follows. 
\end{proof}
\begin{proof}[Proof of Theorem ~\ref{t-180428zy}]
Theorems~\ref{t-gue200930yyd}, \ref{t-gue201001yyd} 
and \ref{t-gue201003yyd} yield Theorem~\ref{t-180428zy}
for the case when $\dim X_G\geq3$. 

Assume now that $\dim X_G=1$. Then $X_G$ is a union of circles. 
For simplicity, suppose that $X_G=S^1$.
The circle $X_G$ admits a natural $S^1$
action $e^{i\theta}: S^1\times X_G\To X_G$, $(e^{i\theta},z)
\To e^{i\theta}z$. For every $m\in\mathbb Z$, put 
\[L^2_m(X_G):=\set{u\in L^2(X_G);\, 
\mbox{$(e^{i\theta})^*u=e^{im\theta}u$, for every 
$e^{i\theta}\in S^1$}}.\]
It is clear that $L^2(X_G)=\oplus_{m\in\mathbb Z}L^2_m(X_G)$. 
By definition $H^0_b(X_G)=\oplus_{m\in\N}L^2_m(X_G)$.
The Szeg\H{o} projection on $X_G$ is the orthogonal projection:
$S_{X_G}: L^2(X_G)\To H^0_b(X_G)$. 
The $S^1$ action induces a smooth vector field $\frac{\pr}{\pr\theta}$
on $X_G$. Fix a point $p\in X_G$. Let $x$ be local coordinate of $X_G$
such that $x(p)=0$ and $\frac{\pr}{\pr x}=\frac{\pr}{\pr\theta}$. 
Then
\begin{equation}\label{e-gue201012yyd}
S_{X_G}(x,y)\equiv\frac{1}{2\pi}\int^{\infty}_0e^{it(x-y)}dt.
\end{equation}
In particular, $S_{X_G}(x,y)$ is a Fourier integral operator
with complex phase. 
Therefore the above proof of Theorem~\ref{t-180428zy}
in the case $\dim X_G\geq3$ works also when $\dim X_G=1$. 
\end{proof}
\begin{theorem}\label{t:ma5.7}
Let $X$ be a three dimensional compact orientable 
CR manifold and let $G$ be a connected compact Lie group acting on $X$
such that the $G$-action preserves $J$ and $\omega_{0}$.
We assume that $X$ is pseudoconvex of finite type and
that $\ddbar_{b}$ has closed range in $L^2$ on $X$.
Then the conclusions of Theorem \ref{t-180428zy} hold.
\end{theorem}
\begin{proof}
It was shown in \cite[Proposition 4.1]{Ch89} that there exists a 
bounded linear operator $\mG:\Ran\ddbar_{b}\rightarrow L^{2}(X)$
such that $S=I-\mG\bar{\partial}_{b}$. Furthermore $\mG$
 is smoothing away the diagonal and maps smooth functions to 
 smooth functions, where $S$ denotes the Szeg\H{o} projector. 
Hence, the Szeg\H{o} projector is 
smoothing outside the diagonal and preserves the space of
smooth functions.
Moreover, $\dim X_G=1$ in this case, thus the arguments
above apply again.
\end{proof}

\begin{example}
If $X$ be a compact pseudoconvex three dimensional CR manifold
of finite type admitting a transversal CR circle action, then 
$\ddbar_b$ has closed range in $L^2$, see \cite[\S 5.2]{HS20}.

Let $M$ be a compact Riemann surface
and $(L,h^L)$ be semi-positive line bundle over $M$ whose
curvature $R^L$ vanishes to finite order at any point.
Then the Grauert tube $X$ of $L^*$ \eqref{eq:GraTube}
is a compact pseudoconvex three dimensional CR manifold
of finite type (cf.\ \cite{HS20}, \cite[Proposition 11]{MS18}) 
admitting a transversal CR circle action.
If $G$ is a connected compact Lie group 
acting holomorphically on $M$ and whose action lifts to $(L,h^L)$,
Theorem \ref{t:ma5.7} applies to $X$.

\end{example}

\section{An almost complex version of Theorem \ref{t-gue181211}}
\label{s:6}

In this section we will prove a version of Theorem \ref{t-gue181211}
in the case of an almost complex manifold.
We will mostly follow \cite{TZ98} and adopt the notations therein.

Let $(L ,h^{L})$ be a Hermitian line bundle with Hermitian connection
$\nabla^{L}$ and associated curvature $R^{L}$ on a compact
almost complex manifold $(M,J)$. 

Let $G$ be a compact connected Lie group with Lie algebra 
$\kg$ acting (on left) on $M$, whose action lifts on $L$ such that 
$h^{L}$, $\nabla^{L}$ are $G$-equivariant. Then the moment map
$\mu: M\to \kg^{*}$
is defined by the Kostant formula
\begin{align}\label{0.3}
 2 \pi \sqrt{-1}\left\langle \mu, \xi\right\rangle : 
 =\nabla ^L_{\xi_{M}}- L_\xi, \, \,   \text{ for }\xi\in \kg.
 \end{align}
 For any $\xi\in \kg$, we have
 \begin{align}\label{0.4}
 d\left\langle \mu, \xi\right\rangle= i_{\xi_{M}}\omega,
 \quad \text{ with }\omega= \frac{\sqrt{-1}}{2\pi} R^{L}.
 \end{align}

 We assume that the almost complex structure $J$ on 
 $M$ and $R^{L}$ are $J$-invariant,
 $G$ acts freely on $\mu^{-1}(0)$ and $\omega(\cdot, J\cdot)$
 defines a metric on $TM|_{\mu^{-1}(0)}$.
 
 By choosing any $G$- and $J$-invariant Riemannian metric 
 $g^{TM}$ on $TM$, we can define an associated Dirac operator $D^{L}$ 
 on $\Lambda(T^{*(0,1)}M)\otimes L$
 and $D^{L}_{\mp}$ its restriction on $\Omega^{0, 
 \text{even/odd}}(M,L)$ (cf. \cite[Definition 1.1]{TZ98}).
 Its index as a finite dimensional virtual representaion of $G$,
 \begin{align}\label{b0.4}
\mathrm{Ind} (D^{L}_{+}) = \ker   (D^{L}_{+}) 
- \ker   (D^{L}_{-}) \in R(G),
\end{align}
does not depend on the choice of $g^{TM}$.

Moreover $(L,h^{L}, \nabla^{L})$, $J, \omega$ on $M$
induce canonically $(L_{G},h^{L_{G}}, \nabla^{L_{G}})$, $J_{G}, 
\omega_{G}$ on $M_{G}= \mu^{-1}(0)/G$. In particular,
$(M_{G}, \omega_{G})$ is a compact symplectic manifold with compatible 
almost complex structure $J_{G}$. 
Thus $\mathrm{Ind}(D^{L_{G}}_{+})$ is 
well-defined as a virtual vector space.

\begin{theorem}\label{eq:t6.1} 
Let $(M,J)$ be a compact almost complex manifold and 
$(L,h,\nabla^L)$ be a Hermitian line bundle with connection on $M$.
Let $G$ be a compact connected Lie group 
acting (on left) on $M$, whose action lifts on $L$ such that 
$J$, $h^{L}$ and $\nabla^{L}$ are $G$-equivariant.
We assume that $G$ acts freely on $\mu^{-1}(0)$ and 
$\omega(\cdot, J\cdot)$
defines a metric on $TM|_{\mu^{-1}(0)}$.
Then there exists $m_{0}\in \N$ such that 
	for any $m\geq m_{0}$ we have
	 \begin{align}\label{eq:6.4}
\mathrm{Ind} (D^{L^{m}}_{+})^{G} = \mathrm{Ind}(D^{L^{m}_{G}}_{+}).
\end{align}
	\end{theorem}
	
Note that by the main result of \cite{M98} and \cite[Theorem 0.1]{TZ98},
if $\omega(\cdot, J\cdot)>0$ on the whole $M$, 
then Theorem \ref{eq:t6.1} holds for $m_{0}=1$.
 
 \begin{proof}
We adapt directly the notation and  argument from \cite{TZ98}.

We fix a $G$- and $J$-invariant metric $g^{TM}$ on $TM$ such that 
near $\mu^{-1}(0)$, $g^{TM}$ is given by $\omega(\cdot, J\cdot)$.
We fix an $\text{Ad}_{G}$-invariant scalar product on $\kg$
and identify $\kg$ and $\kg^{*}$ via this product.

Let $h_{1}, \cdots, h_{d}$ be an orthonormal basis of $\kg$. 
Set
 \begin{align}\label{eq:6.5}
X^{\mH}(x) = 2 (\mu(x))_{M}(x) = 2\sum_{i}\mu_{i}(x) V_{i}(x)
\quad \text{with }  \mu= \sum_{i}\mu_{i}(x) h_{i}
\, \, \text{and   } V_{i}= h_{i, M}.
\end{align}
Following \cite[Definition 1.2]{TZ98}, for $T\in \R$, set
 \begin{align}\label{eq:6.6}
D^{L^{m}}_{T}= D^{L^{m}}+ \frac{\sqrt{-1}}{2} T c (X^{\mH})
: \Omega^{0,*}(M,L) \to \Omega^{0,*}(M,L),
\end{align}
where $c(\cdot)$ is the Clifford action.

Then  we have (cf. \cite[(1.26)]{TZ98})
 \begin{align}\label{eq:6.7}
(D^{L^{m}}_{T})^{2}= (D^{L^{m}})^{2}
+ \frac{\sqrt{-1}}{2} T \sum_{j} c(e_{j})c ( \nabla_{e_{j}}X^{\mH})
- \sqrt{-1} T \nabla_{X^{\mH}}
+  \frac{T^{2}}{4} \Big|X^{\mH}\Big|^{2},
\end{align}
where $\{e_{j}\}_{j}$ is an orthonormal frame
of $(TM, g^{TM})$.

Now as in \cite[(1.27)]{TZ98},
 \begin{align}\label{eq:6.8}
 \nabla_{X^{\mH}}
 = 2\sum_{i}\mu_{i} L_{V_{i}} + 4m \sqrt{-1} T \mH 
 + A,
\end{align}
where $A$ is an  endomorphism of $\Lambda(T^{*(0,1)}M)$
and does not depend on $m$.

We fix sufficiently small $G$-invariant open neighborhoods
$U'\Subset  U$ of $\mu^{-1}(0)$ such that $G$ acts freely on 
$U$ and $\omega(\cdot, J\cdot)>0$ on $U$.

Note that on $\Omega^{0,*}(M,L)^{G}$, $L_{V_{i}}=0$. From 
\eqref{eq:6.7} and \eqref{eq:6.8}, there exists $m_{0}>0$
such that for any $m\geq m_{0}$, 
\cite[Theorem 2.1]{TZ98} holds for $U'$. 
Then on $U$, it is exactly the situation 
considered in \cite[\S 3b)-3e)]{TZ98}. Thus we get 
Theorem \ref{eq:t6.1} as in \cite[(3.36), (3.37)]{TZ98}.
\end{proof}

It is an interesting question to show that in the holomorphic 
situation, under the assumption of this section, we have
\begin{align}\label{eq:6.10}
H^{j}(M, L^{m})^{G}=0\quad \text{  for  any } j>0, m\gg 1.
\end{align}

\bibliographystyle{plain}

\end{document}